\documentclass[a4paper,12pt]{article}
\usepackage[left=2.5cm,right=2.5cm]{geometry}
 
\usepackage[british]{babel}
\usepackage[autostyle=false, style=british]{csquotes}
\usepackage[utf8]{inputenc}
\usepackage{amsthm}
\usepackage{graphicx}
\usepackage{latexsym}
\usepackage{setspace}

\usepackage[frak=euler,cal=pxtx,bb=ams]{mathalfa}
\usepackage{enumerate}
\usepackage[shortlabels]{enumitem}
\usepackage[T1]{fontenc}
\usepackage{titlesec}

\usepackage{amsmath}
\usepackage{rotating}
\usepackage{scalerel}[2016/12/29]

\usepackage{tikz-cd}
\usetikzlibrary{cd}

\usepackage{fancyhdr}

\newif\iflight
\lighttrue
\newif\ifnocomment
\newif\ifcolor

\titleformat{\chapter}[display]
{\LARGE\bfseries\scshape}
{{\chaptertitlename}\ \thechapter}
{.5em}
{\huge}

\titleformat{\section}[hang]
{\large\bfseries\scshape}
{\thesection.}{0.5em}{}

\numberwithin{equation}{section}
\let\oldsection\section
\renewcommand{\section}{
	\renewcommand{\theequation}{\thesection-\arabic{equation}}
	\oldsection}

\titleformat{\subsection}[hang]
{\large\bfseries}
{\thesubsection.}{0.5em}{}

\setlist[enumerate]{leftmargin=1.8em,itemsep=0.2em,topsep=0.2em,listparindent=\parindent,parsep=0em}

\newtheoremstyle{newdefinition}%
{\topsep}
{\topsep}
{\normalfont}
{0pt}
{\bfseries\scshape}
{.}
{ }
{}

\newtheoremstyle{newplain}%
{\topsep}
{\topsep}
{\itshape}
{0pt}
{\bfseries\scshape}
{.}
{ }
{}

\newtheoremstyle{remark}%
{\topsep}
{\topsep}
{\normalfont}
{0pt}
{\scshape}
{.}
{ }
{}

\makeatletter
\renewenvironment{proof}[1][\proofname]{\par
	\pushQED{\qed}%
	\normalfont \topsep6\p@\@plus6\p@\relax
	\trivlist
	\item[\hskip\labelsep
	\bfseries\scshape
	#1\@addpunct{.}]\ignorespaces
}{%
	\popQED\endtrivlist\@endpefalse
}
\makeatother

\usepackage[backend=bibtex,style=alphabetic,sorting=nyt,firstinits=true,doi=false,isbn=false,url=false,eprint=false,maxnames=99,useprefix=true]{biblatex}

\let\oldbibliography\thebibliography
\renewcommand{\thebibliography}[1]{%
	\oldbibliography{#1}%
	\setlength{\itemsep}{0pt}%
}

\usepackage[labelfont=sc,bf]{caption}
\theoremstyle{newdefinition} \newtheorem{definition}{Definition}[section]
\theoremstyle{newdefinition} \newtheorem{remark}[definition]{Remark}
\theoremstyle{newplain} \newtheorem{theorem}[definition]{Theorem}
\theoremstyle{newplain} \newtheorem{lemma}[definition]{Lemma}
\theoremstyle{newplain} \newtheorem{fact}[definition]{Proposition}
\theoremstyle{newplain} \newtheorem{corollary}[definition]{Corollary}
\theoremstyle{newplain}
\theoremstyle{newplain}
\theoremstyle{newdefinition} \newtheorem{metaremark}[definition]{Metaremark}
\newtheorem*{namedthm}{\namedthmname}
\newcounter{namedthm}


\theoremstyle{newplain}
\newtheorem{introtheorem}{Theorem}

\theoremstyle{newdefinition} \newtheorem{introremark}[introtheorem]{Remark}

\theoremstyle{newplain} \newtheorem{introcorollary}[introtheorem]{Corollary}

\theoremstyle{newplain} \newtheorem{introfact}[introtheorem]{Proposition}

\newcommand{\ray}[2]{(#1#2)^{\infty}}
\newcommand{\edgeinf}[6]
{\mathcal{E}_{\infty}(\ray{#1}{#2},\ray{#3}{#4},#5,#6)}
\newcommand{\edge}[4]
{\mathcal{E}(#1,#2,#3,#4)}

\newcommand{\intr}{\mathrm{int}}

\newcommand{\square}{\Box}

\newcommand{\bdry}[2]{\,\overline{\!#1}_{#2}}
\newcommand{\ebdryx}[1]{\,\overline{\!X}^{#1}}
\newcommand{\Fix}{\mathrm{Fix}}
\newcommand{\geod}[2]{\varrho_{#1,#2}}
\newcommand{\im}{\mathrm{im}}

\newcommand{\Min}{\mathrm{Min}}
\newcommand{\rGamma}{\text{\reflectbox{$\Gamma$}}}
\newcommand{\norm}[2]{{\Vert #2 \Vert_{#1}}}
\newcommand{\bVert}{\big\Vert}
\newcommand{\bnorm}[2]{{\bVert #2 \bVert_{#1}}}
\newcommand{\spl}{\mathrm{spl}}
\newcommand{\St}{\mathrm{St}}
\newcommand{\st}{\mathrm{st}}
\newcommand{\ccc}{$\mathfrak{ccc}$}

\newcommand{\simcpl}{\mathsf{c}}
\newcommand{\id}{\mathrm{id}}
\newcommand{\dstfun}{\mathfrak{d}}
\newcommand{\rips}{\mathrm{Rips}}
\newcommand{\Stab}{\mathrm{Stab}}
\newcommand{\csch}[4]{#1^{#2}(#3,#3 \setminus #4)} 

\newcommand{\bpl}{\big(}
\newcommand{\bpr}{\big)}

\newcommand{\C}{\mathbb{C}}
\newcommand{\R}{\mathbb{R}}
\newcommand{\N}{\mathbb{N}}

\newcommand{\Z}{\mathbb{Z}}

\newcommand{\nb}[1]{\textcolor{blue}{\bf\large \#}\footnote{\textcolor{blue}{#1}}}
\newcommand{\nr}[1]{\textcolor{red}{\bf\large \#}\footnote{#1}}
\definecolor{darkgreen}{rgb}{0.05, 0.50, 0.06}
\newcommand{\ngr}[1]{\textcolor{darkgreen}{\bf\large \#}\footnote{\textcolor{darkgreen}{#1}}}
\definecolor{purp}{rgb}{0.53, 0.38, 0.56}
\newcommand{\npu}[1]{\textcolor{orange}{\bf\large
		\#}\footnote{#1}}
\newcommand{\nv}[1]{\textcolor{violet}{\bf\large
		\#}\footnote{\textcolor{violet}{#1}}}
\definecolor{afb}{rgb}{0.36, 0.54, 0.66}
\newcommand{\nby}[1]{\textcolor{afb}{\bf\large
		\#}\footnote{\textcolor{afb}{#1}}}
\definecolor{bro}{rgb}{0.59, 0.29, 0.0}
\newcommand{\nbz}[1]{\textcolor{bro}{\bf\large
		\#}\footnote{\textcolor{bro}{#1}}}
\definecolor{dmb}{rgb}{0.55, 0.0, 0.0}
\newcommand{\nce}[1]{\textcolor{dmb}{\bf\large
		\#}\footnote{\textcolor{dmb}{#1}}}

\ifnocomment
\renewcommand{\nb}[1]{}
\renewcommand{\nr}[1]{}
\renewcommand{\ngr}[1]{}
\renewcommand{\npu}[1]{}
\renewcommand{\nv}[1]{}
\renewcommand{\nby}[1]{}
\renewcommand{\nbz}[1]{}
\renewcommand{\nce}[1]{}
\fi

\ifcolor

\else

\fi

\ifcolor

\else
\fi

\ifcolor

\else

\fi

\usepackage{tikz}
\usetikzlibrary{calc}
\usetikzlibrary{decorations.pathreplacing}
\usetikzlibrary{decorations.pathmorphing}
\usetikzlibrary{decorations.text}
\usetikzlibrary{decorations.markings}
\usetikzlibrary{decorations.shapes}

\usepackage{hyperref}
\makeatletter \newcommand{\myhypertarget}[1]{\Hy@raisedlink{\hypertarget{#1}{}}} \makeatother

\title{{\sc\bfseries On boundaries of bicombable spaces}\footnote{\today}}
\author{Daniel Danielski\\{\footnotesize Daniel.Danielski@math.uni.wroc.pl}\\{\small University of Wroc\l{}aw, Faculty of Mathematics and Computer Science}\\{\small Mathematical Institute, pl.~Grunwaldzki 2, 50-384 Wroc\l{}aw, Poland}}
\date{\vspace{-5ex}}

\begin{document}
	\setlength{\abovedisplayskip}{0.5em}
\setlength{\belowdisplayskip}{0.5em}
\setlength{\abovedisplayshortskip}{0.5em}
\setlength{\belowdisplayshortskip}{0.5em}

\maketitle
\pagestyle{fancy}
\fancyhead{} 
\fancyhead[R]{\today}

\begin{abstract}
We initiate systematic study of EZ-structures (and associated boundaries) of groups acting on spaces 
that admit 
consistent and conical (equivalently, consistent and convex)
geodesic bicombings.
Such spaces 
recently 
drew
a lot of attention 
due to the fact that many classical groups act 
`nicely' 
on them. 
We 
rigorously 
construct
EZ-structures, discuss their uniqueness (up to homeomorphism), provide examples, and prove some 
boundary-related
features 
analogous to 
the ones exhibited by CAT(0) spaces and groups,
which form a subclass of the discussed class of spaces and groups.
\end{abstract}

\begin{spacing}{0.9}
	\tableofcontents 
\end{spacing}

\setcounter{section}{-1}

\section{Introduction}

Non-positive curvature (NPC) plays a prominent role in the Geometric Group Theory, and appears there in many forms.
Two of the most important instances of NPC are CAT(0) 
\cite{BrHae99}
and (Gromov/word) hyperbolic 
\cite{Gromov87}
spaces and groups. Recently, a lot of attention has been attracted by another form of non-positive curvature, 
in a sense 
generalising
these
two concepts, and reaching far beyond the world of CAT(0) and hyperbolic spaces and groups. 
Namely, we focus our attention on geodesic metric spaces that admit 
so-called 
\ccc{} geodesic bicombings.
A 
\emph{geodesic bicombing} 
$\sigma$ on a space $X$ is a
continuous function 
$X\times X\times[0,1]\to X$ which is
a continuous choice of geodesics in $X$,
and
a bicombing is a 
\emph{\ccc{} bicombing} 
if it 
satisfies properties called 
consistency and conicality (or, equivalently, 
the ones called
consistency and convexity)
---
see Definition~\ref{def:bicombing} and Remark~\ref{r:cccbic} for more details.   

CAT(0) spaces are particular examples of 
\ccc-bicombable spaces,
and with every 
word hyperbolic group there is a canonically associated 
\ccc-bicombable space 
--- 
namely the so-called injective hull of the 
word hyperbolic 
group. Further important examples of 
\ccc-bicombable spaces
come from the realm of injective metric spaces \cite{Lang13,DeLa15}
--- see Corollary \ref{c:injective}\ref{c:injective1}; 
and other important examples of groups acting 
`nicely' 
on 
\ccc-bicombable spaces
include Helly groups 
\cite{CCGHO25}
--- see Corollary \ref{c:injective}\ref{c:injective2}. 
In particular, the last class of groups 
includes 
many classical families 
--- 
FC-type Artin groups, some lattices in Euclidean buildings, Garside groups, some small cancellation groups, as well as all CAT(0) cubical groups, 
and all 
word hyperbolic groups. 
For some of these groups the structure of a group acting geometrically on a 
\ccc-bicombable
space is the only known form of non-positive curvature, and this 
allows 
to 
exhibit many important features of such groups and to answer 
a
few 
open questions about them. The theory of Helly groups, 
groups acting geometrically on injective metric spaces, and, more generally, groups acting 
in a controlled way
on 
\ccc-bicombable
spaces is currently being intensively developed, bringing many new achievements:
\cite{Lang13,DeLa15,DeLa16,Miesch17,Basso18,CCHO20,HaOs21,HuOs21,GuMoSch22,Haettel22a,Haettel22b,Haettel23,HaHoPe23,Hoda23,OsVa24,Haettel24,CCGHO25}.

In the current article, 
we 
initiate 
a systematic study of boundaries of 
\ccc-bicombable
spaces, in particular, such spaces acted upon 
geometrically 
by
a group. More precisely, for a group $G$ acting geometrically on a 
\ccc-bicombable
space $X$ we 
construct and
study its 
EZ-structures. 
Following \cite{Bestvina96,FaLa05,Dranishnikov06,OsPr09},
where 
Z-structures and 
their 
equivariant version, EZ-structures, have been defined in ways slightly differing in 
some details,
an 
\emph{EZ-structure}
$\bdry{X}{}$
is a 
$G$-equivariant
compactification of $X$ 
where the boundary 
$\partial X:=\bdry{X}{}\setminus X$ 
is 
a
`small'
subset of $\bdry{X}{}$
---
see Definition \ref{def:ezstr} and the discussion below 
it.

The
visual-boundary
compactification of a CAT(0) space 
admitting a geometric group action,
or 
the
compactification of a suitable Rips complex of a 
word hyperbolic
group by 
adding
its Gromov boundary, 
are 
the 
two most important
--- 
and historically 
the 
first
---
examples of 
EZ-structures. 
Already the existence of an EZ-structure has very important consequences, e.g.~it implies the Novikov Conjecture for the group in the torsion-free case. 
Furthermore, the topology of the boundary reflects some algebraic properties of the group (see 
e.g.~\cite{Bowditch98,Swenson99,PaSw09})
and various topological invariants of the boundary are invariants of the group 
(see e.g.~\cite{BeMe91,Bestvina96,Dranishnikov06,GeOn07}).
 Besides CAT(0) and 
word hyperbolic groups, 
(E)Z-structures 
have
been  constructed 
for various other families of groups
in e.g.~%
\cite{
Dahmani03,OsPr09,Tirel11,Martin14,Guilbault14,Pietsch18,GuMoTi19,GuMoSch22,EnWu23,CCGHO25}.
For the existence of 
a
Z-structure
in the torsion-free case,
it is required for a group to be
of type $F$, 
that is 
to
admit
a finite classifying space.
It is 
a well-known
open 
question
whether all groups of type $F$
possess
(E)Z-structures
\cite[Question in 3.1]{Bestvina96}. 

\bigskip
\smallskip
Recall 
that an action of a group $G$ on a metric space $X$ via isometries is \emph{geometric} if it is \emph{proper}, that is the set $\{g\in G:gK\cap K\neq\emptyset\}$ is finite for all 
compact
sets $K\subseteq X$, and \emph{cocompact}, that is the quotient $G\backslash X$ is compact.
Our first 
result is that 
at least two constructions lead to 
EZ-structures 
for 
groups acting geometrically on
\ccc-bicombable
spaces.

\begin{introtheorem}\label{t:main}
	Let $G$ be a group acting geometrically on a finite-dimensional proper
	geo\-desic
	metric space $(X,d)$,
	which 
	possesses 
	a 
	\ccc,
	$G$-equivariant 
	bicombing $\sigma$. Then $G$ admits an EZ-structure. 
\end{introtheorem}

One of the aforementioned constructions 
uses
the 
construction
via the
Gelfand dual from \cite{EnWu23},
and 
the other
one goes along 
the
more standard approach, known from the CAT(0) and 
word
hyperbolic cases, namely 
via 
equivalence classes of infinite rays (compatible with the lines) of the bicombing, initiated in the 
article
\cite{DeLa15}
--- see Subsections \ref{sbs:enwubdez} and \ref{sbs:deslanez}, respectively, for more details.
These constructions turn out to 
be
equivalent,
see Corollary~\ref{c:equivbdry}.
We note here that the construction of
an
EZ-structure
following \cite{DeLa15}
---
where the authors give the construction of the compactification and prove several of the Z-structure--properties for it
---
was claimed in 
\cite{CCGHO25},
however
there was no 
rigorous proof there. 
We also note that the second of the constructions has been recently carried out by Basso \cite{Basso24} under different assumptions to obtain 
EZ-structures
under more general assumptions than these in
Theorem~\ref{t:main}, namely the \ccc{} condition is relaxed to the assumption of conicality 
--- see Remark \ref{u:bassobdry} for more details.
For a space $X$ admitting a \ccc{} bicombing $\sigma$, we denote below by $\partial_\sigma X$ the boundary resulting from the construction from \cite{DeLa15}, and put $\bdry{X}{\sigma}:= X\cup\partial_\sigma X$.

\smallskip
Since \cite{ArPa56},
injective metric spaces have
appeared 
and have been rediscovered
in many research areas
of mathematics and computer science, and 
are known 
under many different names:
in 
the 
contexts
closest to 
ours 
---
in topology and metric geometry 
---
they are known
as 
\emph{hyperconvex spaces}, or \emph{absolute retracts} or \emph{injective objects} (in the category of metric spaces with
1-Lipschitz maps); 
they
have
also appeared
in
e.g.~%
functional
analysis and ﬁxed point theory 
\cite{Sine79,Soardi79},
and theoretical computer science
\cite{ChLa94}.
In this article we use the following definition, known under the name \emph{`hyperconvexity'}:
a metric space $(X,d)$ is injective if for every family of 
points $x_i\in X$ and radii $r_i>0$
such that $d(x_i,x_j)\leq r_i+r_j$ for every $i,j\in I$ 
the intersection $\bigcap \overline{B}(x_i,r_i)$
(of closed balls)
is non-empty.
As is shown in \cite{Isbell64},
each metric space $X$ admits an \emph{injective hull} $E(X)$ --- the `smallest' injective metric space that contains $X$.
Then 
the \emph{combinatorial dimension} \cite{Dress84} is defined as $\dim_{\mathrm{comb}}X=\sup\{\dim E(Y):Y\subseteq X,|Y|<\infty\}$. 
For 
more 
details
about
the notion of an injective metric space and 
of
the
injective hull $E(X)$ of a metric space $X$
one may
also 
see
e.g.~\cite{Lang13}. 
For more 
about Helly graphs and groups 
one may
see e.g.~\cite{CCGHO25}.

\begin{introcorollary}\label{c:injective}
	A group $G$ acting geometrically on: 
	\begin{enumerate}[(i)]
		\item\label{c:injective1} 
		either a finite-dimensional proper injective metric space $X$ 
		\item\label{c:injective2}
		or a locally finite Helly graph $\Gamma$
	\end{enumerate}
	admits an EZ-structure.
\end{introcorollary}

\begin{introremark}\label{u:equivariant}
	In view of \cite[Theorems 1.1 and 1.2]{DeLa15}, if $X$ 
	is a proper metric space 
	of
	finite combinatorial dimension or 
	such that
	every bounded subset of $X$ has finite combinatorial dimension, and admits a conical bicombing, then 
	$X$
	admits a 
	\ccc,
	reversible
	bicombing,
	which is 
	a
	unique
	convex bicombing on $X$, 
	therefore 
	is 
	equivariant with respect to the full isometry group 
	of $X$. 	
\end{introremark}

\begin{proof}
	\ref{c:injective1}
	By \cite[Proposition 3.8(1)]{Lang13}, 
	$X$ admits a conical bicombing and, since for each subspace $Y\subseteq X$ the injective hull $E(Y)$ embeds into $E(X)=X$ (see e.g.~\cite[Proposition 3.5(1)]{Lang13}), 
	we have that 
	$\dim_{\mathrm{comb}}X\leq\dim X<\infty$. Therefore the claim follows by Remark \ref{u:equivariant} and Theorem \ref{t:main}.
	
	\medskip
	\ref{c:injective2}
	By \cite[Theorem 6.3]{CCGHO25}, 
	the action of $G$ extends to a geometric action on the
	injective hull $E(\Gamma)$, 
	which is a proper injective metric space, and which is finite-dimensional as a locally finite polyhedral complex on which 
	the 
	group 
	$G$
	acts cocompactly. The claim follows by 
	\ref{c:injective1}.
\end{proof}

\medskip
Croke and Kleiner 
\cite{CrKl00}
showed that in the case of a group acting geometrically on two CAT(0) spaces $X$ and $Y$, their boundaries $\partial X$ and $\partial Y$ may be 
non-homeomorphic.
Since CAT(0) spaces are 
\ccc-bicombable, 
this provides examples of non-homeomorphic boundaries of 
a group acting geometrically on a \ccc-bicombable space.
However, even in the CAT(0) world, restricting to some particular cases of spaces and/or groups brings one to a situation when a boundary of a group is well defined, up to homeomorphism. Such phenomenon is called `boundary rigidity' 
\cite{Ruane99,Hosaka03,HrKl05}.
Since injective spaces form one of the most important examples of 
\ccc-bicombable
spaces, and their geometry seems very 
`rigid' 
in a sense, a natural question is whether the boundary in this case is unique up to homeomorphism. We 
give a negative answer
in the following theorem. 

\begin{introtheorem}[Theorem \ref{t:nonunique}]\label{t:nonuniqueintro}
There exists a group $G$ acting geometrically on two
proper
finite-dimensional
injective metric spaces $X^1,X^2$ with 
convex
bicombings $\sigma^1,\sigma^2$, respectively, such that $\partial_{\sigma^1} X^1$ and $\partial_{\sigma^2} X^2$ are not homeomorphic. 	
\end{introtheorem}

The example of the group $G$ from the above theorem is the original 
Croke--Kleiner 
example 
\cite{CrKl00}.
The two injective spaces $X^1,X^2$ are basically the 
Croke--Kleiner 
polygonal
complexes 
carefully equipped with two different injective metrics. It is intriguing that, unlike in the CAT(0) case, the injective metric structure imposes severe restrictions on the gluing pattern within the complexes. In particular, we are able to produce only two homeomorphism types of boundaries for the given group $G$. In the CAT(0) case 
infinitely 
many pairwise 
non-homeomorphic 
boundaries have been produced 
\cite{Wilson05,Mooney10}.
Furthermore, at the moment we do not know whether boundaries of Helly groups are unique, that is, whether
there exists a group acting on two Helly graphs whose associated boundaries are non-homeomorphic
--- see \ref{q:Helly} for more detail. 

\medskip
Since CAT(0) groups and 
word
hyperbolic groups are examples of groups acting geometrically on 
\ccc-bicombable
spaces, a natural question arises about relations between all types of boundaries. It is clear that if a 
word
hyperbolic group acts geometrically on a 
space
admitting a \ccc{} bicombing $\sigma$, 
then
the 
boundary 
$\partial_\sigma X$
coincides with the Gromov boundary. 
For
the CAT(0) spaces the situation is much more subtle, even in, otherwise restricted, case of CAT(0) cubical complexes. Here we have been able to obtain the following result in dimension $2$.

\begin{introtheorem}[Corollary \ref{c:ccc}]\label{t:cccintro}
Let $X$ be a 
CAT(0) cube complex
of dimension at most $2$,
and 
let $\sigma^p$ be the convex bicombing on 
$X$ equipped with the 
piecewise-$\ell^p$ metric
for 
$p\in[1,\infty]$. 
Then the identity of $X$ extends to a homeomorphism between $\bdry{X}{\sigma^p}$ and $\bdry{X}{\sigma^r}$ for any $p,r\in[1,\infty]$, in particular 
all of 
the boundaries $\partial_{\sigma^p}X$
for $p\in[1,\infty]$
are homeomorphic. 
\end{introtheorem}

In fact, we prove a more general result that the geodesics from both bicombings follow the same lines --- see Theorem \ref{t:ccc} --- which in view of Proposition \ref{p:bilipbic} implies the 
theorem
above. 

\bigskip
Our further results concern properties of the boundary analogous to the ones of CAT(0) boundary and the Gromov boundary. We were able to extend some results concerning 
the
CAT(0) case to the more general setting of 
\ccc-bicombable
spaces.

\medskip
The first of such results is the analogue of the result of 
Osajda--\'Swiatkowski 
\cite{OsSw15} 
(cf.~also 
\cite{Moran16})
concerning the existence of a particular metric on the boundary, and consequently, existence of a well-defined quasisymmetric structure. 

\begin{introfact}[see Proposition \ref{p:quasisym}\ref{p:quasisym5}]\label{p:quasisymintro}
	Let $(X,d)$ be a complete metric space that admits
	an action of a group $G$
	and a 
	\ccc{},
	$G$-equivariant
	bicombing $\sigma$.
	Then 
	there exists a metric $d_q$ on $\partial_\sigma X$ such that
	the extension of 
	the
	action of each element of $G$ to $\bdry{X}{\sigma}$ 
	restricts to 
	a quasisymmetry of $(\partial_\sigma X,d_q)$. 
	
\end{introfact}

Recall 
that in the 
word hyperbolic case an analogous quasisymmetric structure plays an important role in 
understanding 
the group. 
No
such 
spectacular
applications 
are known in the CAT(0) setting
at the moment, 
due to 
a much looser connection between 
the
boundary and the group.
The 
metric 
from the above proposition
has been 
used
in research
regarding the linearly-controlled dimension and the asymptotic dimension \cite{Moran16};
we 
also
use 
it, 
as a convenient tool in the further course of the text.

\medskip
An important application of EZ-structures is a consequence of the relation between the topology of the boundary and algebraic properties of the group. Such relations are quite well understood in the 
word 
hyperbolic and CAT(0) case, see eg.~\cite{Bowditch98,Swenson99,PaSw09}.
We extend such studies to the case of groups acting geometrically on 
\ccc-bicombable
spaces. 
Here is 
a minor
result in this direction.

\begin{introfact}[Proposition \ref{p:amalgamated}]\label{p:amalgamatedintro}
Let $G=G_1*_ZG_2$ (with $G_1\neq Z\neq G_2$), 
where $Z$ is virtually $\Z$, 
act geometrically on a proper metric space $(X,d_X)$ that admits a 
\ccc,
reversible, $G$-equivariant bicombing $\sigma$. Then there exists a separating pair of points in the boundary $\partial_\sigma X$.
\end{introfact}

Another result analogous to the CAT(0) case is the following theorem about boundaries of groups containing abelian subgroups. In particular, besides boundaries of 
\ccc-bicombable
spaces coinciding with the Gromov or CAT(0) boundaries, this provides the first basic examples of EZ-structures for groups acting geometrically on 
\ccc-bicombable
spaces 
--- 
such boundaries of free abelian groups are spheres, the same as in the CAT(0) case. 

\begin{introfact}[Proposition \ref{p:sphereinbdry}]\label{p:sphereinbdryintro}
Let $G$ be group that contains a free abelian subgroup $\Z^n\cong A<G$ and acts geometrically on a proper metric space $X$ that admits a
\ccc,
reversible, $G$-equivariant bicombing $\sigma^X$. Then $\partial_{\sigma^X} X$ contains a homeomorphic copy of $S^{n-1}$. 
Moreover, 
if $A$ is of finite index in $G$, then $\partial_{\sigma^X} X\cong S^{n-1}$.
\end{introfact}

The proof of an analogous statement in the CAT(0)
setting 
easily follows from the fact that the minset of $A$, which is a 
convex
subset of $X$,
splits by the Flat Torus Theorem \cite[{}II.7.2]{BrHae99} as a 
metric
product with one of the factors being an $n$-dimensional Euclidean space $F$, giving an $n$-dimensional flat $F$ as a convex subset of the space $X$.
However,
in the context of 
spaces admitting 
\ccc, reversible 
bicombings,
even though flats exist, there is no such splitting nor 
such
convexity 
of flats
--- see the first paragraph of the proof of the above 
proposition 
--- which leads us to 
approach
the problem in a more elementary way. 

\medskip
In 
the case 
of
a
group 
$G$
acting geometrically on a space $X$,
the
Alexander--Spanier
cohomology
(which 
is
equivalent to the \v{C}ech cohomology)
of the boundary describes 
some
cohomological properties of the group
--- e.g.~one has the Bestvina--Mess formula \cite{Bestvina96} (in the torsion-free case),
or
in the case of groups acting 
on
\ccc-bicombable spaces, 
one has
the isomorphism between the group cohomology $H^{*+1}(G,\Z G)$ and the reduced cohomology of the boundary $\tilde{\bar{H}}^*(\partial_\sigma X)$ 
(see Remark 
\ref{u:groupch}).
In this paper
we present two 
\ccc-counterparts 
of important results concerning the topology of CAT(0) boundaries of spaces acted 
upon
cocompactly 
by
a group, see \cite{GeOn07,Ontaneda05}.
The meaning of the first 
of them
is that the Alexander--Spanier cohomology 
behaves 
quite well with respect to the topology 
--- 
it
indicates 
the dimension of the boundary. 

\begin{introtheorem}[Theorem \ref{t:agcmt}]\label{t:agcmtintro}
Let $X$ be a 
finite-dimensional proper metric space 
that admits a
\ccc{}
geodesic bicombing $\sigma$ and a cocompact group action via isometries,
such that $|\partial_\sigma X|\geq 2$. 
Then the
reduced Alexander--Spanier cohomology 
group $\tilde{\bar{H}}^{\dim \partial_\sigma X}(\partial_\sigma X)$ is non-zero.
\end{introtheorem}

The second result is a pretty immediate consequence of the 
first
one, but is very useful by itself. We 
say that
a space 
$X$ that admits a \ccc{} bicombing $\sigma$
is
\emph{almost
$\sigma$-geodesically 
complete}
if the (the images of) $\sigma$-rays originating from some (equivalenty, any) point of $X$ are coarsely dense in $X$. 
This property allows one to
perform a 
`push-out'
on various 
objects 
(e.g.~paths or subspaces) 
from the space $X$ 
`towards infinity', in particular, to the boundary. It makes 
it
possible to 
`transfer' 
various reasonings between the space and its boundary hence and forth. 

\begin{introtheorem}[Theorem \ref{t:agcc3}]\label{t:agcc3intro}
Assume that $X$ is a proper 
finite-dimensional 
geodesic metric space that admits a 
\ccc{}
geodesic bicombing 
$\sigma$
and a cocompact group action via isometries,
such that $|\partial_\sigma X|\geq 2$. 
Then $X$ is 
almost 
$\sigma$-geodesically 
complete.	
\end{introtheorem}

As a note on the assumptions of Theorems \ref{t:agcmtintro} and \ref{t:agcc3intro}, we remark here that we do not know whether a space with a 1-point boundary exists ---
it does not if the bicombing $\sigma$ is additionally equivariant w.r.t.~the group action. 
Furthermore, one may change the assumption that $|\partial_\sigma X|\geq 2$ for the one that $X$ is non-compact in Theorems \ref{t:agcmtintro} and \ref{t:agcc3intro} iff such a space does not exist. See Remarks \ref{u:oneptbdry1} and \ref{u:oneptbdry2} for more detail.

\bigskip
In the course of this article, we 
prove
also the following minor results that may be themselves of some interest.

In Proposition \ref{p:fingpfixpt} we note that 
the
construction 
from \cite{BaMi19}
of a conical, reversible bicombing out of a conical one,
is equivariant, which leads to the fact that a finite group acting on a space that admits a conical bicombing has a fixed point.

In Proposition \ref{f:loctoglobgeod} we give 
a direct
proof 
of the fact
that local geodesics of a \ccc{} bicombing are global geodesics of this bicombing 
(such
fact
could have been 
deduced from a more general setting in \cite{Miesch17}).

In Proposition \ref{p:join} we state and prove 
an observation 
that the boundary of a product of proper \ccc-bicombable spaces $X,Y$, with respect to a naturally defined product bicombing, is the join 
of the boundaries of $X$ and $Y$.

\bigskip
\textbf{Organisation of the paper.}
In Section \ref{s:prelim} we establish some notation, 
as well as
provide and discuss some general notions used later in this article. We also prove Proposition \ref{p:fingpfixpt}.

In Section \ref{s:deslan} we 
recall
the construction of
the 
boundary-compactification via geodesic rays as in \cite{DeLa15},
prove some preliminary facts about this construction,
and in Subsection~\ref{sbs:deslanez} we prove the existence of an EZ-structure (Theorem \ref{t:main}) relying on this construction.

In Section \ref{s:enwubd} we 
recall 
the construction of the boundary-compactification via the Gelfand dual from \cite{EnWu23} (in the, less general compared to \cite{EnWu23}, setting of metric spaces); 
in Subsection~\ref{sbs:enwubdez} we prove the existence of an EZ-structure (Theorem \ref{t:main}) relying on this construction, 
and discuss the obtained EZ-structures;
in Subsection~\ref{sbs:equivbdry} we prove that the EZ-structures obtained in Subsections \ref{sbs:deslanez} and \ref{sbs:enwubdez} are equivalent (Proposition \ref{p:equivbdrymain} and Corollary \ref{c:equivbdry}).

In Section \ref{s:nonunique} we prove the non-uniqueness of the boundary in the injective case 
---
Theorem \ref{t:nonunique} (Thm.~\ref{t:nonuniqueintro}).
During the preparations for this proof we prove Propositions \ref{f:loctoglobgeod} and \ref{p:bilipbic}. 

In Section \ref{s:ccc} we discuss the boundaries of CAT(0) cube complexes 
equipped 
with piecewise-$\ell^p$ 
metrics 
---
justifying Theorem \ref{t:ccc} and Corollary \ref{c:ccc} (Thm.~\ref{t:cccintro})
--- 
and prove a proposition 
concerning
the boundary of the product of spaces (Proposition \ref{p:join}). 

In Section \ref{s:quasisym} we discuss a metric on the boundary that leads to a quasisymmetric structure on the boundary of \ccc-bicombable spaces (Proposition \ref{p:quasisym}, cf.~Prop.~\ref{p:quasisymintro}).

In Section \ref{s:flats} 
we use the existence of axes and flats in bicombable spaces to prove Proposition \ref{p:amalgamated} (Prop.~\ref{p:amalgamatedintro}) and Proposition \ref{p:sphereinbdry} (Prop.~\ref{p:sphereinbdryintro}).

Section \ref{s:agc} concerns almost geodesic completeness for \ccc-bicombable spaces:
in Subsection \ref{sbs:agcprep} we prove some preparatory lemmas,
and in Subsection \ref{sbs:agcmain} we
prove 
Theorem~\ref{t:agcmt} (Thm.~\ref{t:agcmtintro}) and Theorem \ref{t:agcc3} (Thm.~\ref{t:agcc3intro}).

In Section \ref{s:open} we collect and formulate 
some
problems and open questions.

\bigskip
\textbf{Acknowledgements.} 
I would like to thank Damian Osajda
for the introduction to the topic,
and Jan Dymara 
for helpful conversations.

I thank Giuliano Basso and Craig Guilbault for helpful comments on the manuscript.
I thank Harry Petyt for useful comments regarding Section \ref{s:ccc} and for pointing us the paper \cite{HaHoPe25}.

The 
author was partially supported by (Polish) Narodowe Centrum
Nauki, grant no 2020/37/N/ST1/01952.

\medskip
This paper comprised the bulk of the PhD thesis of the author, defended in February 2025 at the University of Wrocław.

\section{Preliminaries}\label{s:prelim}

Let
$(X,d)$ 
be
a 
metric space.
Let $x\in X$, $r>0$ and $A\subseteq X$. We denote by $B_X(x,r)$ (resp.~$\overline{B}_X(x,r))$) the open (resp.~closed) ball of radius $r$ around $x$, and put $B_X(A,r):=\bigcup_{x\in A}B_X(x,r)$ 
and 
$\overline{B}_X(A,r):=\bigcup_{x\in A}\overline{B}_X(x,r)$.
We denote by 
$\id_X$ 
the identity (map) 
on
the space $X$.

In this article we tend to omit the 
space-related
sub- and superscripts 
of various objects
when the space we are referring to is clear from the context.

\bigskip
\textbf{Basic notions from basic coarse geometry.}
Let $(X,d_X)$ and $(Y,d_Y)$ be metric spaces.

For a subset $A\subseteq X$ and $C>0$,
we say that the set $A$ is \emph{$C$-dense} in $X$
if $\overline{B}(A,C)=X$.

For a pair of maps $f,f'\colon X\to Y$, we say that they \emph{are at finite distance (from each other)}
whenever there exists a constant $C>0$ such that for all $x\in X$ the inequality $d_Y(f(x),f'(x))\leq C$ holds.

We say that 
a map 
$f\colon X\to Y$ is: 
\begin{enumerate}[$\bullet$]
\item \emph{coarsely Lipschitz}, if there exists $C>0$ such that for all $x,x'\in X$ the inequality $d_Y(f(x),f(x'))\leq Cd_X(x,x')+C$ holds;
\item a \emph{quasi-isometric embedding}, if there exists $C>0$ such that for all $x,x'\in X$ the inequality $C^{-1}d_X(x,x')-C\leq d_Y(f(x),f(x'))\leq Cd_X(x,x')+C$ holds;
\item a \emph{quasi-isometry}, if it is coarsely Lipschitz and admits a \emph{quasi-inverse}, i.e.~a coarsely Lipschitz function $g\colon Y\to X$ such that $f\circ g$ and $g\circ f$ are 
at finite distance from
$\id_Y$ and $\id_X$, respectively. 
\end{enumerate}
An elementary argument shows that any quasi-inverse 
of
a quasi-isometry $X\to Y$ is automatically a quasi-isometry $Y\to X$.
Another elementary argument shows that a function $f\colon X\to Y$ is a quasi-isometry iff it is a quasi-isometric embedding whose image is $C$-dense in $Y$ for some $C>0$.

\bigskip
\textbf{Geodesic bicombings.}
We denote by $\im\,\gamma$ 
the image of the 
curve $\gamma\colon I\to X$, where $I\subseteq\R$ is a possibly infinite interval. 

\begin{definition}\label{def:bicombing}
Let $(X,d)$ be a geodesic metric space.
A \emph{(geodesic) bicombing} $\sigma$ is a continuous function $\sigma\colon X\times X\times[0,1]\to X$ 
such 
that for each $x,y\in X$ the function $\sigma_{xy}:=\sigma(x,y,\cdotp)$ is a constant speed geodesic from $x$ to $y$;
we call 
each
function $\sigma_{xy}$ 
a 
\emph{$\sigma$-geodesic}.
We say that a bicombing $\sigma$ is:
\begin{enumerate}
\item[$\bullet$] \emph{consistent}, if $\im\,\sigma_{\sigma_{xy}(s),\sigma_{xy}(t)}=\im\,\sigma_{xy}|_{[s,t]}$
for all $x,y\in X$ 
and
$s,t\in[0,1]$ with $s<t$;

\item[$\bullet$] \emph{conical}, if 
$d(\sigma_{xy}(t),\sigma_{x'y'}(t))\leq (1-t)d(x,x')+td(y,y')$
for all $x,y,x',y'\in X$, $t\in[0,1]$;

\item[$\bullet$] \emph{convex}, if the function $[0,1]\ni t\mapsto d(\sigma_{xy}(t),\sigma_{x'y'}(t))$ is convex for all 
$x,y,x',y'\in X$; 

\item[$\bullet$] \emph{reversible}, if  
$\sigma_{xy}(t)=\sigma_{yx}(1-t)$
for all $x,y\in X$, $t\in[0,1]$;

\item[$\bullet$] given an action of a group $G$ on $X$ and $g\in G$, 
$g$\emph{-equivariant},
if 
$g\sigma_{xy}=\sigma_{gx,gy}$
for all $x,y\in X$,
and $G$\emph{-equivariant}, if it is $g$-equivariant 
for all $g\in G$.
\end{enumerate}

We say that a subset $A\subseteq X$ is \emph{$\sigma$-convex} if $\im\,\sigma_{xy}\subseteq A$ for any $x,y\in A$. 
\end{definition}

\begin{remark}\label{r:cccbic}
It is clear that 
every convex bicombing is conical.
The 
converse
holds for consistent bicombings --- we need to show that given $x_1,y_1,x_2,y_2\in X$, $0\leq s<t\leq 1$ 
and $\alpha\in[0,1]$, 
we have that 
\begin{multline*}
d(\sigma_{x_1,y_1}(\alpha s+(1-\alpha)t),\sigma_{x_2,y_2}(\alpha s+(1-\alpha)t))\\
\leq\alpha d(\sigma_{x_1,y_1}(s),\sigma_{x_2,y_2}(s))+(1-\alpha)d(\sigma_{x_1,y_1}(t),\sigma_{x_2,y_2}(t)).
\end{multline*} 
Indeed, 
this inequality 
holds,
since
by consistency we have that $\sigma_{\sigma_{x_i,y_i}(s),\sigma_{x_i,y_i}(t)}(\alpha)=\sigma_{x_i,y_i}(\alpha s+(1-\alpha)t)$ for any $i\in\{1,2\}$ and $\alpha\in[0,1]$, so the inequality above
reduces to an inequality asserted in the definition of conicality.
\end{remark}

Bearing in mind the remark above,
we 
call a geodesic bicombing a 
\emph{\ccc\,\,bicombing} 
if it is consistent, 
conical, 
and 
--- 
therefore 
--- 
convex.

\bigskip
We note that the following proposition follows from an appropriate application of methods and theorems 
due to
Basso, Descombes, Lang and Miesch.

\begin{fact}\label{p:fingpfixpt}
Assume that $G$ is a finite group 
acting of a metric space $X$ 
via isometries,
and
that 
$X$
admits a conical, $G$-equivariant geodesic bicombing $\sigma$.
Then the action of $G$ on $X$ has a fixed point.
\end{fact}

\begin{proof}
Basso and Miesch 
in \cite[Proposition 1.3]{BaMi19}
used the following procedure to produce
from a conical bicombing $\sigma$ 
a `midpoint map' $m\colon X\times X\to X$,
and then
a
conical, reversible bicombing $\sigma^R$.
Let $x,y\in X$. Define $x_0:=x$, $y_0:=y$, and $x_{n+1}:=\sigma_{x_n,y_n}(1/2)$, $y_{n+1}:=\sigma_{y_n,x_n}(1/2)$ for $n\in\N$. Then the sequences $(x_n)$ and $(y_n)$ are convergent to the same limit; 
define $m(x,y)$ to be this limit. Finally, $\sigma^R_{xy}(t):=m(\sigma_{xy}(t),\sigma_{yx}(1-t))$ defines a conical, reversible bicombing on $X$.
It is easy to check that, since $\sigma$ is $G$-equivariant, the constructions of the sequences $x_n$ and $y_n$, of the map $m$, and of the bicombing $\sigma^R$ are all also $G$-equivariant.

Descombes, Lang and Basso \cite{DeLa16,Basso18}, building on \cite{ESHa99,Navas13}, introduced a construction of barycentre maps  $\mathrm{bar}_n\colon X^n\to X$ for $n\in\N$ for spaces admitting conical bicombings.
If the input bicombing is additionally reversible, then the functions $\mathrm{bar}_n$ are invariant under permuting their arguments (see \cite[Theorem 4.1(2)]{DeLa16} or \cite[Proposition 3.4(4)]{Basso18})
and,
if the input bicombing is $G$-equivariant, then the barycentre maps are also $G$-equivariant (see \cite[Theorem 4.1(3)]{DeLa16} or \cite[around formula (3.5)]{Basso18}).
It is now easy to see that the barycentre
of any orbit of $G$ 
constructed using the bicombing $\sigma^R$ is a fixed point of the action of $G$ on $X$.
\end{proof}

\bigskip
\textbf{Euclidean retracts and absolute retracts.}
A 
space is a \emph{Euclidean retract (ER)} if it can be embedded in some Euclidean space as its retract.

A space is an \emph{absolute retract (AR)} if, whenever it is embedded into a metric space $Y$ as a closed subspace $A$, the subspace $A$ is a retract of the space $Y$.

\smallskip
A locally compact, separable metric space is an ER iff 
it is finite-dimensional, contractible 
and locally contractible.
We provide an outline of the proof of this characterisation in the following proposition, for completess. 

\begin{fact}\label{p:charerar}
Let $X$ be a metric space. Then:
\begin{enumerate}[(i)]
\item\label{p:charerar1} 
if $X$ locally compact and separable, then $X$ is an ER iff $X$ is a finite-dimensional AR;

\item\label{p:charerar2} 
if $X$ is finite-dimensional, then $X$ is an AR iff $X$ is contractible and locally contractible.
\end{enumerate}
\end{fact}

\begin{proof}
\ref{p:charerar1}{\sc, the $\Longrightarrow$ implication.} 
Let $\imath\colon X\to\R^n$ be an embedding such that the set $\imath(X)$ is a retract of the space $\R^n$.
Since $X$ is locally compact, there exists an open subset $U\subseteq\R^n$ such that $\imath(X)$ is contained in $U$ as a closed subset.
Then a standard trick of taking the product of the map $\imath$ with the map $X\ni x\mapsto d(\imath(x),\R^n\setminus U)^{-1}\in\R$ gives an embedding of $X$ into $\R^{n+1}$ as a closed subset. Therefore $\dim X\leq n+1$, see \cite[Theorem 3.1.4]{EngelkingBook78}.

Assume that 
we have an embedding $\jmath$ of $X$ into a metric space $Y$ such that $\jmath(X)$ is closed in $Y$. By the Tietze's Extension Theorem, 
the homeomorphism 
$\imath\circ\jmath^{-1}\colon\jmath(X)\to\imath(X)\subseteq\R^n$ can be extended to a map $Y\to\R^n$. Composing 
this map
with the retraction of $\R^n$ onto $\imath(X)$, and then with the inverse of the 
homeomorphism
$\imath\circ\jmath^{-1}$
gives a retraction of $Y$ 
\linebreak[1]
onto $\jmath(X)$.

\smallskip
\ref{p:charerar1}{\sc, the $\Longleftarrow$ implication.} 
By the Embedding Theorem, see \cite[Theorem 1.11.4]{EngelkingBook78}, the separable metric space $X$ embeds into a Euclidean space $\R^n$. As in the first paragraph of the proof, by local compactness of $X$, the space $X$ can be embedded into $\R^{n+1}$ as a closed subset. Then the AR-ness of $X$ gives the desired retraction.

\bigskip
\ref{p:charerar2} This is exactly the equivalence of (a) and (b) in \cite[Theorem V.11.1]{SzeTsenBook65}.
\end{proof}

For more information on the topic one may refer to \cite{SzeTsenBook65} or \cite{vanMillBook89}, or to 
the 
concise introduction in \cite[Section 2]{GuMo19}.

\bigskip
\textbf{EZ-structures.}
A compact subset $Z$ of a compact space $\mathfrak{X}$ is a 
\emph{Z-set} 
in $\mathfrak{X}$ if there exists a homotopy $\{H_t\colon\! \mathfrak{X}\to \mathfrak{X}:t\in[0,1]\}$ such that $H_0=\mathrm{id}_\mathfrak{X}$ and $H_t(\mathfrak{X})\cap Z=\emptyset$ for any $t>0$.

A sequence $K_n$ of subsets of a topological space $X$ is called a \emph{null-sequence} if for every open cover $\mathcal{U}$ of $X$, 
for all but finitely many $n$ the set $K_n$ is $\mathcal{U}$\emph{-small}, that is, there exists a set $U\in\mathcal{U}$ such that $K_n\subseteq U$.

\begin{definition}\label{def:ezstr}
Let $Z\subseteq \mathfrak{X}$ and $G$ act geometrically on $\mathfrak{X}\setminus Z$. The pair $(\mathfrak{X},Z)$ is an \emph{EZ-structure} for $G$ if $\mathfrak{X}$ an ER, $Z$ is a Z-set in $\mathfrak{X}$, $(gK)_{g\in G}$ is a null-sequence in $\mathfrak{X}\setminus Z$ for every compact set $K\subseteq \mathfrak{X}\setminus Z$, and the action of $G$ on $\mathfrak{X}\setminus Z$ extends to an action by homeomorphisms on $\mathfrak{X}$.
\end{definition}

The notion of the Z-structure has been introduced by Bestvina in \cite{Bestvina96}, and its equivariant 
version,
the EZ-structure, has been introduced in \cite{FaLa05}.
These notions have later been studied, 
perfected
and generalised in e.g.~\cite{Dranishnikov06,OsPr09}.
We note that there exist several variations of the definition of the EZ-structure,
mainly differing with the above by altering the ER-ness assumption to an AR-ness assumption, or by altering the condition that the group action is geometric to other conditions imposed on the group action.
In the definition above we followed the last of the mentioned articles. 

\bigskip
\textbf{Gluings.} Further in this article we often encounter the following gluing setting. 
We are given a family of metric spaces $(X_i,d_i)$ for $i\in I$ and a surjective map $\pi\colon\bigsqcup_{i\in I}X_i\to X$ such that $\pi$ restricted to each of the $X_i$ is one-to-one, and such that for each $i,j\in I$ 
the map
$\pi|_{X_{j}}^{-1}\circ\pi|_{X_i}$ induces an isometry from $(\pi|_{X_i}^{-1}(\pi(X_i)\cap\pi(X_{j})),d_i)$ to $(\pi|_{X_{j}}^{-1}(\pi(X_i)\cap\pi(X_{j})),d_{j})$. 
Such a map $\pi$ is called a \emph{gluing map}.
We construct the following gluing pseudometric on $X$, which often turns out to be a metric in our settings.
Let $d(x,x'):=\inf\{\mathrm{len}(P):P\text{ is a }(x,x')\text{--gluing path}\}$,
where
an 
\emph{$(x,x')$--gluing path} is a sequence of points $P=(x=x_0,\ldots,x_n=x')$ such that $n\in\N$ and for each $0\leq j\leq n-1$ there exists $i_j\in I$ such that $x_j,x_{j+1}\in\pi(X_{i_j})$, and we define its length as $\mathrm{len}(P):=\sum_{i=0}^{n-1}d_{X_{i_j}}(x_j,x_{j+1})$.
We often do not distinguish the set $X_i$ from $\pi(X_i)$ and treat it as a subset of $X$.

\section{Boundary via geodesic rays}\label{s:deslan}

We begin with a brief summary of the construction in \cite[Section 5]{DeLa15}.

Let $(X,d)$ be a complete geodesic metric space with a 
\ccc{}
bicombing $\sigma$.
A 
$\sigma$-ray
in $X$ is an isometric embedding 
$\xi\colon [0,\infty)\to X$ such that 
$\im\,\sigma_{\xi(0)\xi(t)}=\im\,\xi|_{[0,t]}$ for every $t\geq0$.
Two
$\sigma$-rays $\xi,\zeta$ are asymptotic if their images are at finite Hausdorff distance, equivalently, $\sup_{t\geq0}d(\xi(t),\zeta(t))<\infty$, and we denote by $[\xi]$ the set of $\sigma$-rays asymptotic to 
the $\sigma$-ray
$\xi$.

The boundary $\partial_\sigma X$ is a topological space whose underlying set consists of the set of classes of asymptotic $\sigma$-rays.
By \cite[Proposition 5.2]{DeLa15}, for each basepoint $o\in X$, 
every
class $\bar{x}\in\partial_\sigma X$ has a unique representative $\geod{o}{\bar{x}}$ that originates in $o$, 
therefore the boundary $\partial_\sigma X$
may be identified
with the set of 
$\sigma$-rays 
originating at some arbitrary fixed point $o$. We 
often consider the set $\bdry{X}{\sigma}:=X\cup\partial_\sigma X$, and 
extend the definition of $\varrho$ to $X$
by defining functions $\geod{o}{x}\colon[0,\infty)\to X$ for $x\in X$ by stopping after reaching $x$, i.e.~$\geod{o}{x}(t)=\sigma_{ox}(\min(t/d(o,x),1))$ for any $o,x\in X$,
and
identify point $x\in X$ with $\geod{o}{x}$.

The topology on $\bdry{X}{\sigma}$ is given by the following base: for a fixed basepoint $o\in X$, take $\{U_o(\bar{x},t,\epsilon):\bar{x}\in\bdry{X}{\sigma},t\geq0,\epsilon>0\}$ where $U_o(\bar{x},t,\epsilon)=\{\bar{y}\in 
\bdry{X}{\sigma}
:d(\geod{o}{\bar{y}}(t),\geod{o}{\bar{x}}(t))<\epsilon\}$. One can show that the resulting topology does not depend on the choice of the basepoint $o$, the topology on 
$\bdry{X}{\sigma}$
extends the topology on $X$ 
---
see \cite[Lemma 5.3 and above]{DeLa15}
--- 
and the topology on $\partial_\sigma X$ has a base $\{U_o(\bar{x},t,\epsilon)\cap\partial_\sigma X:\bar{x}\in \partial_\sigma X,t\geq0,\epsilon>0\}$. 
One may observe that the topology coincides with the 
topology arising from viewing $\bdry{X}{\sigma}$ as the inverse limit of the system
$(\{X_R:R\geq0\},\{\pi^R_r\colon X_R\to X_r:0\leq r\leq R\})$ 
where $X_R=\overline{B}(o,R)(=\bigcup\{\geod{o}{x}([0,R]):x\in X\})$ and $\pi_r^R(\geod{o}{x}(R))=\geod{o}{x}(r)$ for all $x\in X$ (which is well-defined by consistency of $\sigma$). Thus $\bdry{X}{\sigma}$ admits a metric 
\begin{equation}\label{eq:do}
d_o(\bar{x},\bar{y}):=\sum_{n=1}^\infty2^{-n}\min\!\bpl d(\geod{o}{\bar{x}}(n),\geod{o}{\bar{y}}(n)),1\bpr,
\end{equation}
arising from viewing $\bdry{X}{\sigma}$ as the subspace $\{(x_n)_{n\in\N}:x_n\in X_n,\pi^{n+1}_n(x_{n+1})=x_n\}\subseteq\prod X_n$ (we note that this metric is in the same spirit as the metric $D_o$ introduced in \cite{DeLa15}), and 
the space
$\bdry{X}{\sigma}$ is compact iff $X$ is proper. 
Note that for 
all
$r>0$ the 
map 
given by $\pi_r(\bar{x})=\geod{o}{\bar{x}}(r)$ 
corresponds 
to the projection 
map  
$\pi_r\colon\bdry{X}{\sigma}\to X_r$ from the definition of an inverse limit.

\medskip
We finish this introduction with three useful observations.
Proposition \ref{f:neweq52} will be used in the proof of Theorem \ref{t:main} in this section, and Proposition \ref{f:asymptotic} and Proposition \ref{p:ellexp} will be used in other parts of this 
article.

\begin{fact}\label{f:asymptotic}
	Assume that $\sigma$ is a 
	\ccc{}
	bicombing on a complete metric space $(X,d)$ and $\zeta,\eta$ are $\sigma$-rays.
	\begin{enumerate}[(i)]
		\item\label{f:asymptotic1}
		The function $D(t):=d(\zeta(t),\eta(t))$ is convex. In particular, if $\zeta$ and $\eta$ are asymptotic, then $D$ is non-increasing.
		
		\item\label{f:asymptotic2}
		Let $\Xi$ be a family of $\sigma$-rays originating from a common point and such that the set $\{[\xi]:\xi\in\Xi\}$ is compact in $\partial_\sigma X$.
		If 
		\begin{equation*}
		(\exists D\geq0)(\forall r>0)(\exists t\geq r,\xi\in\Xi)(d(\zeta(t),\im\,\xi)\leq D), 
		\end{equation*}
		then $\zeta$ is asymptotic to some $\sigma$-ray from $\Xi$.
		
		\item\label{f:asymptotic3} 
		The $\sigma$-rays
		$\zeta$ and $\eta$ are asymptotic iff $(\exists D\geq0)(\forall r>0)(\exists t\geq r)(d(\zeta(t),\im\,\eta)\leq D)$.

	\end{enumerate}	
\end{fact}

\begin{proof}
	\ref{f:asymptotic1} Follows from convexity of $\sigma$. 
	
	\medskip
	\ref{f:asymptotic2} Assume 
	that there 
	exists $D\geq 0$
	and
	sequences $t_n\to\infty$, $s_n\geq0$
	and $\xi_n\in\Xi$, 
	such that $d(\zeta(t_n),\xi_n(s_n))\leq D$.
	Since $\zeta$ and $\xi_n$ are geodesic rays, 
	\begin{multline*}
	|t_n-s_n|=|d(\zeta(t_n),\zeta(0))-d(\xi_n(0),\xi_n(s_n))|\\
	\leq d(\zeta(t_n),\xi_n(s_n))+d(\zeta(0),\xi_n(0))
	\leq  D+d(\zeta(0),\xi_n(0)),
	\end{multline*} 
and $d(\zeta(t_n),\xi_n(t_n))\leq2D+d(\zeta(0),\xi_n(0))$ for any $n\in\N$.
	By compactness, one may choose a subsequence $n_k$ such that $[\xi_{n_k}]$ converge to $[\xi]$ for some $\xi\in\Xi$.
	Then, for any $k$ and $t\leq t_{n_k}$, 
	by convexity of $\sigma$
	we have that 
	\begin{multline*}
	d(\zeta(t),\xi(t))
	\leq d(\zeta(t),\xi_{n_k}(t))+d(\xi_{n_k}(t),\xi(t))\\
	\leq \max\!\bpl\!\!\; d(\zeta(0),\xi_{n_k}(0)),d(\zeta(t_{n_k}),\xi_{n_k}(t_{n_k}))\bpr+d(\xi_{n_k}(t),\xi(t))\\
	\leq 2D+d(\zeta(0),\xi_{n_k}(0))+d(\xi_{n_k}(t),\xi(t)).
	\end{multline*}
	Therefore, since the $\sigma$-rays from $\Xi$ are based in one point, 
	and
	$t_{n_k}\to\infty$, and $\xi_{n_k}(t)\to\xi(t)$ for all $t\geq0$, the $\sigma$-rays $\xi$ and $\zeta$ are asymptotic.
	
	\medskip
	\ref{f:asymptotic3} The $\Longrightarrow$ implication follows from \ref{f:asymptotic1}.
	The $\Longleftarrow$ implication 
	follows from \ref{f:asymptotic2} with $\Xi=\{\eta\}$.
\end{proof}

The following builds on \cite[inequality (5.2)]{DeLa15}.

\begin{fact}\label{f:neweq52}
	Let $(X,d)$ be a metric space with a 
	\ccc{}
	bicombing $\sigma$, and $x,y,o\in X$, $r>0$ be such that $\max(d(o,x),d(o,y))\geq r$. Then $d(\geod{o}{x}(r),\geod{o}{y}(r))\leq 2\cdot d(x,y)\cdot r/d(o,x)$.	
\end{fact} 

\begin{proof}
When $d(o,x)\geq d(o,y)$ and $d(o,x)\geq r$, the inequality is (almost) \cite[inequality (5.2)]{DeLa15}
and can be justified as follows:
note that $\geod{o}{x}(r)=\sigma_{o,x}(r/d(o,x))$, 
and denote $y_{\sim r}=\sigma_{o,y}(r/d(o,x))$; then
\begin{align*}
d(\geod{o}{x}(r)&,\geod{o}{y}(r))
\leq d(\geod{o}{x}(r),y_{\sim r})+d(y_{\sim r},\geod{o}{y}(r))\\[0.3em]
&\leq \frac{r}{d(o,x)}d(x,y)+ d(o,\geod{o}{y}(r))-d(o,y_{\sim r})\\[0.3em]
&= \frac{r}{d(o,x)}d(x,y)+ \min(r,d(o,y))-\frac{r}{d(o,x)}d(o,y)\\[0.3em]
&\leq \frac{r}{d(o,x)}d(x,y)+r\frac{d(o,x)-d(o,y)}{d(o,x)}
\leq \frac{r}{d(o,x)}d(x,y)+r\frac{d(x,y)}{d(o,x)}
= 2r\frac{d(x,y)}{d(o,x)}.
\end{align*}

\medskip
\noindent
If $d(o,y)\geq d(o,x)$ and $d(o,y)\geq r$, the 
above
gives that 
\begin{equation*}
d(\geod{o}{x}(r),\geod{o}{y}(r))\leq 2\:\!r\:\!d(x,y)/d(o,y)\leq 2\:\!r\:\!d(x,y)/d(o,x).\qedhere
\end{equation*}
\end{proof}

\begin{definition}\label{d:ellexp}
Let $o\in X$.
We define a map
$\ell_o\colon\bdry{X}{\sigma}\to[0,\infty]$ by 
putting
$\ell_o(x)=d(o,x)$ for $x\in X$ and $\ell_o(\bar{x})=\infty$ for $\bar{x}\in\partial_\sigma X$;
and 
define
an exponential 
map $\exp_o\colon\bdry{X}{\sigma}\times[0,\infty]\to\bdry{X}{\sigma}$ 
by 
$\exp_o(\bar{x},t)=\geod{o}{\bar{x}}(t)$
for $\bar{x}\in\bdry{X}{\sigma}$ and $t<\infty$,
and $\exp_o(\bar{x},\infty)=\bar{x}$. 
\end{definition}

\begin{fact}\label{p:ellexp}
Let $(X,d)$ be a complete metric space with a 
\ccc{}
bicombing $\sigma$. 
Then 
the 
maps $\ell_o$ and $\exp_o$ are continuous for any basepoint $o\in X$.
\end{fact}

\begin{proof}
The function $\ell_o$ is continuous
since it
is
continuous on $X$ 
as the metric $d$ is continuous with respect to itself, and for $\bar{x}\in\partial_\sigma X$ we have that $\ell_o(U_o(\bar{x},R,\delta))\subseteq[R-\delta,\infty]$.

\medskip
For the proof of continuity of $\exp_o$, assume that we have sequences $(\bar{x}_n)\subseteq\bdry{X}{\sigma}$ convergent to $\bar{x}\in\bdry{X}{\sigma}$ and $(t_n)\subseteq[0,\infty]$ convergent to $t\in[0,\infty]$. 
If 
$t=\infty$, then for any $s<\infty$ we have for sufficiently large $n$ that $\geod{o}{\exp_o(\bar{x}_n,t_n)}(s)=\geod{o}{\bar{x}_n}(s)$, as $t_n\to t=\infty$; the latter converges to $\geod{o}{\bar{x}}(s)=\geod{o}{\exp_o(\bar{x},\infty)}(s)$, as $\bar{x}_n\to\bar{x}$; therefore $\exp_o(\bar{x}_n,t_n)$ converges to $\exp_o(\bar{x},\infty)$.
Assume that $t<\infty$.
If $\bar{x}\in X$, then $\bar{x}_n\in X$ for sufficiently large $n$, so we can write by continuity of $\sigma$ that
\smallskip
\begin{equation*}
\exp_o(\bar{x}_n,t_n)=\sigma\!\left(\!o,\bar{x}_n,\frac{\min(t_n,d(o,\bar{x}_n))}{d(o,\bar{x}_n)}\right)\;\longrightarrow\,\,\sigma\!\left(\!o,\bar{x},\frac{\min(t,d(o,\bar{x}))}{d(o,\bar{x})}\right)=\exp_o(\bar{x},t).
\end{equation*} 

\smallskip
\noindent
Assume that $\bar{x}\in\partial_\sigma X$. For sufficiently large $n$ we have that $t_n\leq R$ (and thus ${t\leq R}$) for some $R<\infty$. Then, 
by consistency of $\sigma$,
we have for sufficiently large $n$ that $\exp_o(\bar{x}_n,t_n)=\exp_o(\exp_o(\bar{x}_n,R),t_n)$.
Since $\exp_o(\bar{x}_n,R)$ converges to $\exp_o(\bar{x},R)$, as 
$\bar{x}_n$
converges to
$\bar{x}$, we can use the previous case to conclude that $\exp_o(\bar{x}_n,t_n)$ converges to $\exp_o(\exp_o(\bar{x},R),t)=\exp_o(\bar{x},t)$.
\end{proof}

\subsection{The EZ-structure}\label{sbs:deslanez}

Below we describe how the above construction of boundary may be used to construct an EZ-structure.

\begin{proof} ({\sc of Theorem} \ref{t:main})
We 
show that $(\bdry{X}{\sigma},\partial_\sigma X)$ is an EZ-structure for $G$ (see Definition \ref{def:ezstr}).
Fix any basepoint $o\in X$.

\smallskip
The boundary $\partial_\sigma X$ is a Z-set in 
$\bdry{X}{\sigma}$
by \cite[Theorem 1.4]{DeLa15}.
Fix any homotopy $\{H_t:t\in[0,1]\}$ from the definition of $Z$-set. 

\smallskip
Now we show that the 
compact
space $\bdry{X}{\sigma}$ is an ER,
using the characterisation from Proposition \ref{p:charerar}.
The space $\bdry{X}{\sigma}$
is contractible and locally contractible by 
\cite[Theorem 1.4]{DeLa15}. The fact that the
dimension of 
$\bdry{X}{\sigma}$
is not greater than the dimension of $X$, therefore finite, is a standard task in topology 
and can be justified as follows. 
Let $\mathcal{U}$ be an open cover of 
$\bdry{X}{\sigma}$.
Let $\lambda$ be the Lebesgue number of $\mathcal{U}$.
Since $X$ is finite-dimensional,
there exists an open cover $\mathcal{V}$ of $X$ consisting of sets of $d_o$-diameter at most $\lambda/3$ and having empty intersections of each of its subfamilies of cardinality $\dim X+2$. By compactness of 
$\bdry{X}{\sigma}$,
there exists $t_0>0$ such that for all 
$x\in \bdry{X}{\sigma}$ 
we have $d_o(H_{t_0}(x),x)<\lambda/3$.
It follows that $\{H_{t_0}^{-1}(V):V\in\mathcal{V}\}$ is an open cover of 
$\bdry{X}{\sigma}$,
is a refinement of $\mathcal{U}$, and 
has the property that intersections of any of its $\dim X+2$ sets are empty.   
We note that a similar proof of the fact that $\dim\bdry{X}{\sigma}\leq\dim X$ when $\dim X<\infty$ is also present in the proof of \cite[Theorem 1.4]{DeLa15}, where $X$ is not assumed to be proper and a suitable value $t_0$ is stated explicitly using the metric $D_o$, or in the proof of \cite[Theorem 7.10]{EnWu23}. 

\smallskip
We 
claim
that 
for any
compact set $K\subseteq X$
and
any
open cover $\mathcal{U}$ of 
$\bdry{X}{\sigma}$
all but finitely many of the translates $(gK)_{g\in G}$ are $\mathcal{U}$-small.
By properness of $X$,
$K\subseteq B(o,R)$ for some 
$R\geq0$.
By compactness, let $\{U_o(\bar{x}_i,t_i,\epsilon_i)\cap\partial_\sigma X:1\leq i\leq n\}$ for some $\bar{x}_i\in\partial_\sigma X$, $t_i\geq0$, $\epsilon_i>0$ be a finite cover of $\partial_\sigma X$ such that each element of $\{U_o(\bar{x}_i,t_i,2\epsilon_i):1\leq i\leq n\}$ is contained in some element of $\mathcal{U}$. 
Since 
$\bdry{X}{\sigma}\setminus\bigcup U_o(\bar{x}_i,t_i,\epsilon_i)$ is a compact subset of $X$, it is contained in a ball $B(o,R_0)$ for some $R_0\geq 0$. 
Take $g\in G$ such that 
$d(o,go)>R_0$.
Then
there exists some $1\leq k\leq n$ such that $go\in U_o(\bar{x}_k,t_k,\epsilon_k)$.
If, additionally, $d(o,go)>t_k$ and $2Rt_k/d(o,go)\leq \epsilon_k$, then,
by 
Proposition \ref{f:neweq52} with $r=t_k$ and $x=go$,
\begin{equation*}
gK\subseteq gB(o,R)=B(go,R)
\subseteq U_o(go,t_k,2Rt_k/d(o,go))
\subseteq U_o(go,t_k,\epsilon_k)\subseteq U_o(\bar{x}_k,t_k,2\epsilon_k),
\end{equation*}
which is contained in some set from the family $\mathcal{U}$.
Since $G$ acts on $X$ properly, all but finitely many $g\in G$ satisfy 
$d(o,go)>\max\!\bpl\{R_0\}\cup\{t_i:1\leq i\leq n\}\cup\{2Rt_i/\epsilon_i:1\leq i\leq n\}\bpr$,
and the claim follows.

\smallskip
We extend the action of the
group $G$
on $X$ to $\bdry{X}{\sigma}$ by defining 
$g\bar{x}
:=[t\mapsto g\geod{o}{\bar{x}}(t)]$ for all 
$g\in G$, which is well-defined since $G$ acts via isometries and $\sigma$ is $G$-equivariant.
It 
remains
to show that such an extended action is an action via homeomorphisms.
By $G$-equivariance of $\sigma$,
for any $g\in G$, $\bar{x}\in \bdry{X}{\sigma}$, $t\geq 0$ and $\epsilon>0$ 
we have
that $gU_o(\bar{x},t,\epsilon)=U_{go}(g\bar{x},t,\epsilon)$. 
Therefore $g$ maps a base of the topology of $\bdry{X}{\sigma}$ arising from using the basepoint $o$ to another base, resulting from taking $go$ as the basepoint.
\end{proof}

\begin{remark}\label{u:bassobdry}
	The construction presented in \cite{DeLa15} has been recently carried out by Basso \cite{Basso24} to construct 
	boundaries 
	in the setting of complete metric spaces 
	that admit consistent, reversible bicombings 
	such 
	that the function 
	considered in the definition 
	of a convex bicombing (see Definition \ref{def:bicombing}) is convex for points $x,y,x',y'$ 
	satisfying the property
	that the distance between $x$ and $y$, and the distance between $x'$ and $y'$ are equal.
	
	This construction is used
	\cite[Corollary 1.6]{Basso24}
	to
	construct EZ-structures 
	for groups 
	$G$
	acting geometrically on proper metric spaces
	$X$
	admitting conical bicombings
	$\sigma$,
	with AR-ness condition 
	in the definition of an EZ-structure
	in the place of the ER-ness condition 
	considered in this article 
	--- this is a minor difference, as the latter differs from the former by an additional finite-dimensionality condition (see Definition \ref{def:ezstr} and Proposition \ref{p:charerar}\ref{p:charerar1}, and the dimension-related part of the proof above). 
	
	First, the conical bicombing $\sigma$ is used to construct a bicombing $\sigma'$
	with properties 
	as discussed 
	in the first paragraph of
	this remark, and then the aforementioned boundary construction gives a Z-structure.
	This Z-structure is 
	an EZ-structure if the initial conical bicombing 
	is
	$G$-equivariant:
	the bicombing $\sigma'$ is $G$-equivariant as 
	the first sentence of the last paragraph of the proof of \cite[Theorem 1.4]{Basso24} applies to any subgroup of the group of isometries of $X$; 
	and the fact that 
	whenever $\sigma'$ is $G$-equivariant
	the group action extends to an action on the whole boundary-compactification is implicit in the second paragraph of the proof of \cite[Corollary 1.6]{Basso24}.
	
	We also give a minor remark that the proof of \cite[Corollary 1.6]{Basso24} uses an integral metric,
	the same as the metric
	$D_o$ from \cite{DeLa15} (and similar to the metric $d_o$ from 
	the current 
	article), 
	in the proofs of the null-sequence--condition condition and the above-mentioned fact that the action extends continuously to the boundary.
\end{remark}

\section{Boundary via Gelfand dual}\label{s:enwubd}
In this section we discuss the EZ-structure resulting from 
the construction of the boundary introduced by Engel and Wulff in \cite{EnWu23}.
Then we prove 
equivalence of this 
EZ-structure with the one constructed in Section \ref{s:deslan}.
We try to give a more elementary 
treatment of the subject compared to the original paper.

Below
we 
briefly 
describe
the construction of the boundary 
from \cite{EnWu23},
which applies to 
the
so called coarse spaces, in our less general metric setting
(see \cite[Example 2.3]{EnWu23}).
Let $(X,d)$ be a metric space
and
fix a basepoint $o\in X$.
We say that a function 
$\Sigma\colon X\times\N\to X$, where we often use the notation $\Sigma_n(x)=\Sigma(x,n)$,
is a \emph{combing},
and say that $(X,\Sigma)$ is a \emph{combed space},
if the following are satisfied:
\begin{enumerate}
\item[$\bullet$]
$\Sigma_n(o)=o=\Sigma_0(x)$ for all $x\in X$, $n\in\N$;
\item[$\bullet$]
for all $R\geq0$ there exists $N\in\N$ such that $\Sigma_n(x)=x$ for all $n\geq N$ and $x\in\overline{B}(o,R)$;
\item[$\bullet$]
for all $D>0$ there exists $C$ such that for all $x,x'\in X$ with $d(x,x')\leq D$, 
and for all $n,n'\in\N$ with $|n-n'|\leq D$, 
it holds that $d(\Sigma(x,n),\Sigma(x',n'))\leq C$.
\end{enumerate}
(An example of a combing, related to 
Section \ref{s:deslan}, is $\Sigma(x,n):=\geod{o}{x}(n)$.)
We say that a combing $\Sigma$ is \emph{coherent} if there exists $R\geq 0$ such that $d(\Sigma_m(x),\Sigma_m(\Sigma_n(x)))\leq R$ for all $m\leq n\in\N$ and $x\in X$. 
A map $\alpha\colon X\to Y$ is a morphism between combed spaces $(X,\Sigma_X)$ and $(Y,\Sigma_Y)$ 
if the function $(x,n)\mapsto d(\alpha(\Sigma_X(x,n)),\Sigma_Y(\alpha(x),n))$ is bounded (see \cite[Remark 2.5(a)]{EnWu23}).
Morphisms, coarsely Lipschitz maps, 
and being at bounded distance work with each other in the following way.

\begin{fact}\label{p:enwumorfizmy}
	Assume that metric spaces $(X,d_X),(Y,d_Y)$ admit combings $\Sigma_X,\Sigma_Y$, respectively.
	\begin{enumerate}[(i)]
		\item 
		Let $f\colon X\to Y$, $g\colon Y\to X$ be coarsely Lipschitz such that $f\circ g$ is at finite distance from the identity of $Y$. Then, if $f$ is a morphism 
		from
		$(X,\Sigma_X)$ 
		to
		$(Y,\Sigma_Y)$, then $g$ is a morphism 
		from
		$(Y,\Sigma_Y)$ 
		to
		$(X,\Sigma_X)$.
		\label{p:enwumorfizmy1}
		
		\item	
		Assume that $f'\colon X\to Y$ is at finite distance from a morphism $f$
		from
		$(X,\Sigma_X)$ 
		to
		$(Y,\Sigma_Y)$. Then $f'$ is also a morphism from $(X,\Sigma_X)$ to $(Y,\Sigma_Y)$.
		\label{p:enwumorfizmy2}
	\end{enumerate}	
\end{fact}

\begin{proof}
\ref{p:enwumorfizmy1} Since $f\colon X\to Y$ is coarsely Lipschitz, there exists a constant $C$ such that for any $y\in Y$ and $n\in\N$ we have
\begin{multline*}
d_X\!\!\:\bpl g(\Sigma_Y(y,n)),\Sigma_X(g(y),n)\bpr\!\!\:
\leq Cd_Y\!\!\:\bpl f(g(\Sigma_Y(y,n))),f(\Sigma_X(g(y),n))\bpr\!\!\:+C\\
\leq C\bpl d_Y\!\!\:\bpl f(g(\Sigma_Y(y,n))),\Sigma_Y(y,n)\bpr
+d_Y\!\!\:\bpl\Sigma_Y(y,n),\Sigma_Y(f(g(y))),n\bpr\\
+d_Y\!\!\:\bpl\Sigma_Y(f(g(y)),n),f(\Sigma_X(g(y),n))\bpr\!\!\:\bpr\!\!\:+C.
\end{multline*}
The claim follows, as each of 3 terms in the parentheses is bounded by a constant independent of $y$ and $n$:
the first one by the fact 
that
$f\circ g$ is at finite distance from the identity of $Y$;
the second one by the fact that $f\circ g$ is at finite distance from the identity of $Y$, 
and the third ($\bullet$) in the definition of a combing;
the last one by the fact that $f$ is a morphism between $(X,\Sigma_X)$ and $(Y,\Sigma_Y)$.

\medskip
\ref{p:enwumorfizmy2} 
For any $x\in X$, $n\in\N$ we have the following
inequality:
\begin{multline*}
d_Y\!\!\:\bpl f'(\Sigma_X(x,n)),\Sigma_Y(f'(x),n)\bpr
\leq d_Y\!\!\:\bpl f'(\Sigma_X(x,n)),f(\Sigma_X(x,n))\bpr\\
+d_Y\!\!\:\bpl f(\Sigma_X(x,n)),\Sigma_Y(f(x),n)\bpr+d_Y\!\!\:\bpl\Sigma_Y(f(x),n),\Sigma_Y(f'(x),n)\bpr.
\end{multline*}
The claim follows, as each of these 3 terms is bounded by a constant independent of $x$ and $n$:
the first one by the fact that the functions $f$ and $f'$ are at finite distance from each other;
the second one by the fact that $f$ is a morphism;
the last one by the fact the functions $f$ and $f'$ are at finite distance from each other, 
and the third ($\bullet$) 
in
the definition of combing.  		
\end{proof}

Given a coherent combing $\Sigma$ on a proper metric space $X$, 
one can construct the compactification $\ebdryx\Sigma$ (technically, the construction below applies in a more general setting of proper combings, see \cite[Definition 2.6 and Lemma 2.7]{EnWu23}) as follows.
Define $C_\Sigma(X)$ 
to be
the $C^*$-algebra of all continuous, bounded functions $f\colon X\to\C$ such that $f\circ \Sigma_n\to f$ (in the supremum metric) and $f$ has bounded variation, i.e.~for 
all
$R>0$ the variation function $\mathrm{Var}_R(f)(x):=\sup\{|f(y)-f(x)|:y\in \overline{B}(x,R)\}$ is below any $\epsilon>0$ outside of a compact set $K(\epsilon,f,R)$. 
We 
consider 
the Gelfand dual $\ebdryx\Sigma$ of $C_\Sigma(X)$ --- the space of non-zero 
multiplicative $\C$-linear functionals
on the algebra $C_\Sigma(X)$ 
(these properties imply continuity (with norm 1), see \cite[Theorem 1.3.2(i)]{LinBook01})
with the weak-${}^*$ topology. 
Taking the kernel of such functionals gives an identification of 
the underlying set of $\ebdryx\Sigma$ 
with the set of
maximal (proper) ideals in $C_\Sigma(X)$,
see \cite[Theorem 1.3.2(ii)]{LinBook01}.
The map $X\ni x\mapsto\delta_x\in\ebdryx\Sigma$, where $\delta_x$ is the evaluation 
at $x$, is a well-defined homeomorphic embedding (for more details, see the proof of Proposition \ref{p:equivbdrymain}).

\subsection{The EZ-structure}\label{sbs:enwubdez}

Below we show a proof of existence of an EZ-structure, relying on the construction presented previously in this section. Next, we 
describe and compare the EZ-structures
obtained in this proof. 

\begin{proof} ({\sc of Theorem} \ref{t:main})
Fix a basepoint $o\in X$.
By the \v{S}varc--Milnor lemma, the group $G$ is finitely generated, 
and 
the map $\alpha\colon G\to X$
given by the formula $\alpha(g)=go$
is
a $G$-equivariant quasi-isometry; 
therefore the map $\Sigma(x,n):=\geod{o}{x}(n)$, which is clearly a combing of $X$, can be moved to a combing $\Sigma_G$ of $G$ 
using
$\alpha$ as follows:
let $\beta\colon X\to G$ be a quasi-inverse of $\alpha$ and put $\Sigma_G(g,n):=\beta(\Sigma(\alpha(g),n))$ (the basepoint for $\Sigma_G$ is $\beta(o)$).
We 
list
the assumptions of \cite[Theorem 7.10 and Remark 7.15]{EnWu23}, whose conclusion is that $G$ admits an EZ-structure
(the yet-unexplained terms appearing in 
this
list 
are
addressed
when needed
later in
the proof, 
in a form adapted to
our metric setting):
$X$ is a Euclidean retract, 
the space $X$ and the Rips complex $\rips(G)$ are $G$-equivariantly homotopy equivalent,
and the combing $\Sigma_G$ is coherent, expanding and $G$-coarsely equivariant. 
We shall show that all these conditions are satisfied.

\medskip
By the characterisation from Proposition \ref{p:charerar},
the space $X$ is an ER, 
as
it is 
locally compact, 
finite-dimensional (by assumption), contractible (by
the 
existence of bicombing) and locally contractible (the balls in $X$ are $\sigma$-convex because $\sigma$ is conical). 

\smallskip
Now we check that the space $X$ and the Rips complex of $G$ are $G$-equivariantly 
homotopy equivalent.
The Rips complex $\rips(G)$ is defined in \cite[Definition 4.1 and Example 4.4]{EnWu23} as the (increasing) union of the 
Rips complexes $\rips_n(G)$ for $n\in\N$, where the complex $\rips_C(G)$ for $C>0$ is the typically discussed Rips complex, i.e.~a 
finite subset
$T$ of elements of $G$ spans a simplex of $\rips_C(G)$ iff every two elements of $T$ are at distance at most $C$. 
(This way every finite tuple of elements of $G$ spans a simplex of $\rips(G)$, and $\rips(G)$ is equipped with the 
direct limit
topology induced by the inclusions $\rips_n(G)\subseteq\rips(G)$.)
In view of 
\cite[Lemma 3.3]{Lueck05},
it is sufficient to show that 
$\rips(G)$ is a model 
of $\underline{\mathrm{E}}G$
and 
$X$ is a model of $\underline{\mathrm{J}}G$ (a `numerable' version of $\underline{\mathrm{E}}G$, see \cite[Definition 2.3]{Lueck05}).

We use the
homotopy characterisation of $\underline{\mathrm{E}}G$, see \cite[Theorem 1.9(ii)]{Lueck05},
to show that 
the Rips complex $\rips(G)$ is
a model for $\underline{\mathrm{E}}G$.
Namely, we check that
the stabiliser of each element of $\rips(G)$ under the action of $G$ 
is
finite,
and that for each finite subgroup $H$ of $G$ the set $\rips(G)^H$ of fixed points of the action of $H$ on $\rips(G)$ is contractible.
Each element $x\in\rips(G)$ is contained in the interior of a simplex spanned by some elements $g_1,\ldots,g_m\in G$, therefore every element of the stabiliser $\Stab_G(x)$ of $x$ fixes setwise the set $\{g_i:1\leq i\leq m\}$, so $|\Stab_G(x)|\leq m<\infty$.
For each finite subgroup $H$ of $G$ 
the set 
$\rips(G)^H$ 
is non-empty, 
as it contains the barycentre of the simplex spanned by the elements of 
$H$.
It is easy to verify that
for each 
pair of
points 
$x,y\in\rips(G)^H$,  
the set $\rips(G)^H$ 
also
contains the linear segment in $\rips(G)$ between the points $x$ and $y$. 
Therefore, after fixing any point $x\in\rips(G)^H$, going along linear segments in $\rips(G)$ towards $x$ gives 
a 
homotopy $c_x\colon\rips(G)\times[0,1]\to\rips(G)$, 
which
restricts
to a homotopy $\rips(G)^H\times[0,1]\to\rips(G)^H$, 
which 
contracts the set $\rips(G)^H$ to the point $x$.
The continuity of $c_x$ may be justified by the following argument.
Let 
$\Delta_x$
be
a finite subset of $G$ 
such that $x$ belongs to a simplex spanned by 
the elements of $\Delta_x$.
Consider a point $y\in\rips(G)$ contained in a simplex 
spanned by 
the elements of a finite subset $\Delta_y$ of $G$.
We now check continuity of $c_x$ at all points of the form $(y,t)$ for $t\in[0,1]$. 
Let $U_y$ be the union of interiors of all 
of the
simplices of $\rips(G)$ that contain $y$. 
Then the set $U_y$ is an open neighbourhood of $x$ in $\rips(G)$.
Let $C\in\N$ be greater than any distance between a pair of points from the set 
$\Delta_x\cup\Delta_y$.
Then, 
for each $n\in\N$ the contraction $c_x$
restricts to a map 
$(U_y\cap\rips_n(G))\times[0,1]\to\rips_{C+n}(G)$, 
which is continuous.
Passing to the limit with $n$ gives continuity of $c_x$ on $U_y\times[0,1]$.

\newcommand{\rx}{2r_x}
\newcommand{\rxp}{r_x} 

We use the homotopy characterisation of $\underline{\mathrm{J}}G$, see \cite[Theorem 2.5(ii) and Definition 2.1]{Lueck05}, to show that $X$ is a model of $\underline{\mathrm{J}}G$.
Namely, we need to show that
(i) $X$ admits an open cover $\{U_i:i\in I\}$ such that 
(a) each of the sets $U_i$ is $G$-invariant and
admits 
a
\linebreak[2] 
$G$-equivariant map $U_i\to G/G_i$ for  
some finite subgroup 
$G_i$ 
of $G$,
and (b)
there exists a 
locally finite 
partition of unity on $X$ via $G$-invariant functions, subordinate to $\{U_i:i\in I\}$;
(ii) each finite subgroup of $G$ has a fixed point in $X$;
and 
(iii)
the projection maps $X\times X\to X$ onto the first coordinate and onto the second coordinate are homotopic via a $G$-equivariant homotopy.
Regarding 
property (i), first observe that properness of the action of $G$ gives that 
for each point $x\in X$ there exists a 
number
$r_x>0$ such that for all $g\in G$ either $gB(x,\rx)\cap B(x,\rx)=\emptyset$ or $g\in\Stab_G(x)$, and that the stabiliser $\Stab_G(x)$ is finite.
Let $U^x:=\bigcup\{gB(x,\rx):g\in G\}$ and $V^x:=\bigcup\{gB(x,\rxp):g\in G\}$.
It is 
easy to verify that:
assigning to an element $gy$, 
where $g\in G$ and $y\in B(x,\rx)$, 
the coset $g\:\!\Stab_G(x)$ gives a well-defined $G$-equivariant map $U^x\to G/\Stab_G(x)$;
and that the map $\varphi^x\colon X\to[0,\infty)$ defined by
$\varphi^x(y)=\sum_{g\in G}\max(0,\rxp-d(gx,y))$ is a well-defined $G$-invariant continuous map that 
is non-zero 
precisely
in the set 
$V^x$.
Cocompactness of the action of $G$ now gives that there exists a finite set $\{x_i:i\in I\}$ such that the sets 
$V^{x_i}$
for $i\in I$ form an open cover of $X$; let 
$U_i:=U^{x_i}$ and
$\varphi_i:=\varphi^{x_i}$ for $i\in I$.
The 
collection
$\{U_i:i\in I\}$ 
satisfies 
the property (a).
It is easy to check that the functions 
$X\ni y\mapsto\varphi_i(y){\big/}\sum_{j\in I}\varphi_j(y)$ 
for $i\in I$ form a locally finite partition of unity subordinate to $\{U_i:i\in I\}$, and are $G$-invariant since the functions $\varphi_i$ are $G$-invariant.
Property 
(ii)
is satisfied by Proposition \ref{p:fingpfixpt}. 
Property 
(iii)
is satisfied by the fact that the bicombing $\sigma$ gives the desired homotopy.

\medskip
Observe
that $\alpha$ is a morphism from $(G,\Sigma_G)$ to $(X,\Sigma)$,
as for any $h\in G$ and $n\in\N$ we have that $d(\alpha(\Sigma_G(h,n)),\Sigma(\alpha(h),n))=d(\alpha(\beta(\Sigma(\alpha(h),n))),\Sigma(\alpha(h),n))$ is universally bounded since $\alpha\beta$ is at bounded distance from the identity of $X$.
Therefore, by \cite[Lemma 2.8]{EnWu23},
in order to prove coherence and expandingness of $\Sigma_G$,
it suffices to check coherence and expandingness of $\Sigma$.

\smallskip
The fact that $\Sigma$ is coherent 
follows directly from consistency of $\sigma$ (it suffices to take $R=0$ in the definition of coherent combing).

\smallskip
Expandingness (cf.~\cite[Definition 2.6]{EnWu23}) of $\Sigma$ is equivalent to: there exists $R\geq 0$ such that for every $r\geq 0$ and $n\in\N$ there exists $D\geq0$ such that 
$\Sigma_n(\overline{B}(x,r))\subseteq \overline{B}(\Sigma_n(x),R)$ 
for all $x\in X\setminus \overline{B}(o,D)$. 
In fact, this statement holds for all $R>0$, and this is what we will show. 
Fix 
any $R>0$, $r\geq0$, $n\in\N$, 
and take $x,x'\in X$ with $d(o,x)\geq n$ and $d(x,x')\leq r$. Then by 
Proposition \ref{f:neweq52}
we have
$d(\Sigma_n(x),\Sigma_n(x'))\leq 2nr/d(o,x)$,
thus it suffices 
to take 
$D=\max(n,2nr/R)$.

\smallskip
A combing $\Sigma_Y$ on a space $Y$ 
admitting
an
action of $G$ is $G$-coarsely equivariant if the action of each $g\in G$ on $Y$ induces an endomorphism of $(Y,\Sigma_Y)$ (see \cite[Definition 5.14]{EnWu23}).
To 
show that the combing $\Sigma_G$ is $G$-coarsely equivariant, 
we first show that the combing $\Sigma$ is $G$-coarsely equivariant, 
which 
means that 
for all $g\in G$
there exists $R(g)\geq 0$ such that 
$d(\Sigma_n(gx),g\Sigma_n(x))\leq R(g)$ for all $x\in X$ 
and
$n\in\N$.
Indeed, take $g\in G$, $x\in X$ and $n\in\N$. 
Since $\sigma$ is $G$-equivariant, 
$g\sigma_{ox}=\sigma_{go,gx}$.
If $n\leq d(o,x)$ and $n\leq d(o,gx)$, 
we have that 
\begin{align*}
	\omit $d(\Sigma_n(gx),g\Sigma_n(x))$ \hfill&\\
	\leq d\!\left(\!\sigma_{o,gx}\!\left(\!\frac{n}{d(o,gx)}\!\right)\!,\sigma_{o,gx}\!\left(\!\frac{n}{d(go,gx)}\!\right)\!\right)+d\!\left(\!\sigma_{o,gx}\!\left(\!\frac{n}{d(go,gx)}\!\right)\!,\sigma_{go,gx}\left(\!\frac{n}{d(go,gx)}\!\right)\!\right)\qquad&\\[0.4em]
	\leq \frac{n|d(go,gx)-d(o,gx)|}{d(o,gx)d(go,gx)}\,d(o,gx)+\left(1-\frac{n}{d(go,gx)}\right)d(o,go)\qquad&\\[0.3em]
	\leq d(o,go)\left(\frac{n}{d(go,gx)}+1-\frac{n}{d(go,gx)}\right)
	=d(o,go).&
\end{align*}
If 
$n\geq d(o,x)$ 
(equivalently, $\Sigma_n(x)=x$), 
then
\begin{multline*}
d(\Sigma_n(gx),g\Sigma_n(x))=d(\Sigma_n(gx),gx)=\max(0,d(o,gx)-n)\\
\leq\max(0,d(o,gx)-d(o,x))\leq|d(o,gx)-d(go,gx)|\leq d(o,go).
\end{multline*}
If 
$n\geq d(o,gx)$
(equivalently, $\Sigma_n(gx)=gx$), then, 
by the above with $gx$ and $g^{-1}$ in the place of $x$ and $g$, respectively, we obtain
$d(\Sigma_n(x),g^{-1}\Sigma_n(gx))\leq d(o,g^{-1}o)$, which by 
\linebreak[1]
$g$-invariance of the metric $d$ gives $d(g\Sigma_n(x),\Sigma_n(gx))\leq d(go,o)$.
Summarising, it suffices to take $R(g)=d(o,go)$, and $\Sigma$ is $G$-coarsely equivariant. 

Now we show that $\Sigma_G$ is $G$-coarsely equivariant.
Let $g\in G$.
By $G$-equivariance of $\alpha$ we have that $\alpha g=g\alpha$, therefore $\beta\alpha g=\beta g\alpha$. 
The right-hand side is a morphism as a composition of morphisms ($\beta$ is a morphism by Proposition \ref{p:enwumorfizmy}\ref{p:enwumorfizmy1} as a quasi-inverse of $\alpha$), see \cite[Remark 2.18]{EnWu23}, 
therefore $\beta\alpha g$ is an endomorphism of $(G,\Sigma_G)$. 
Since $\beta\alpha$ is at finite distance from the identity map of $G$,
the map $\beta\alpha g$ 
is at finite distance from $g$; 
therefore, 
by Proposition \ref{p:enwumorfizmy}\ref{p:enwumorfizmy2}, 
the action of $g$ on $G$ induces an endomorphism of $(G,\Sigma_G)$. Therefore $\Sigma_G$ is $G$-coarsely equivariant.
\end{proof}

\newcommand{\stpref}{-}

In the remaining part of this subsection we show that all EZ-structures resulting from the proof above (for the choice involved, see Remark \ref{r:equivbdryconstruction}\stpref\ref{r:equivbdryconstructionsigmaprim} below)
are equivalent to the EZ-structure resulting from the compactification $\ebdryx\Sigma$ one would naturally consider (where the combing $\Sigma$ is as above, and its definition is recalled in Remark \ref{r:equivbdryconstruction}\stpref\ref{r:equivbdryconstructionsigma}), and which is among the EZ-structures resulting from the discussed proof (see Metaremark \ref{r:meta}\ref{r:meta2}). 
For a precise meaning of 
`equivalent',
see Proposition \ref{p:equivbdry0}\ref{p:equivbdry02}.

\medskip
We begin with a summary of the construction of the compactification from the proof above.

\begin{remark}\label{r:equivbdryconstruction}
	Formally, the EZ-structure 
	constructed
	in the proof of Theorem \ref{t:main} 
	above
	is 
	$(\ebdryx{\Sigma'},\overline{G}^{\Sigma_G}\setminus G)$
	resulting from the following multistep procedure.
(If the reader finds any of the steps below 
unsettling,
they may refer to Metaremark \ref{r:meta} below.)

	In the first paragraph of the proof 
	above
	we perform the following steps.
	
	\begin{enumerate}[label={\arabic*.}, ref=\arabic*]
		\item \label{r:equivbdryconstructionsigma}
		Fix a basepoint $o\in X$. Define a combing $\Sigma$ of $X$ 
		by
		$\Sigma(x,n):=\geod{o}{x}(n)$.
		
		\item \label{r:equivbdryconstructionsigmag}
		Let $\alpha$ be a $G$-equivariant quasi-isometry given by $G\ni g\mapsto go\in X$ (the \v{S}varc--Milnor Lemma) and $\beta\colon X\to G$ be its quasi-inverse.
		Define a combing $\Sigma_G$ of $G$ 
		by
		$\Sigma_G(g,n):=\beta(\Sigma(\alpha(g),n))$. 
	\end{enumerate}
	
	\smallskip
	\noindent
	Then we apply \cite[Theorem 7.10 with Remark 7.15]{EnWu23}, which states that 
	$(X^{\Sigma'},G^{\Sigma_G}\setminus G)$ --- which 
	is 
	the result of 
	the next two steps ---
	is an EZ-structure for $G$.
(See \cite[below Lemma 7.4, and proof of Theorem 7.10]{EnWu23}.)

\smallskip
\begin{enumerate}[label={\arabic*.}, ref=\arabic*,resume]
		\item
		\label{r:equivbdryconstructionsigmaprim}
		Let $\alpha'\colon G\to X$ 
		be a
		quasi-isometry of the form $\alpha'(g)=gx_0$ for some $x_0\in X$, 
		and 
		let
		$\Sigma'$ be 
		a combing on $X$
		such that
		$\alpha'$
		is a morphism from $(G,\Sigma_G)$ to $(X,\Sigma')$.
		
		\item
		By functoriality of the construction of the boundary, \cite[Corollary 2.18a]{EnWu23}, 
		and Proposition \ref{p:enwumorfizmy},
		the morphism $\alpha'$
		induces a homeomorphism
		from
		the boundary 
		${\overline{G}^{\Sigma_G}\setminus G}$
		of $G$
		to
		the boundary $\ebdryx{\Sigma'}\setminus X$ of $X$.
		Finally, the obtained EZ-structure for $G$ is 
		$(\ebdryx{\Sigma'},\overline{G}^{\Sigma_G}\setminus G)$, which is obtained from  $(\ebdryx{\Sigma'},\ebdryx{\Sigma'}\setminus X)$ using this identification of boundaries.
		\label{r:equivbdryconstructionbdryswap}
	\end{enumerate}

\smallskip
\noindent
	In view of the identification of $(\ebdryx{\Sigma'},\overline{G}^{\Sigma_G}\setminus G)$ with  $(\ebdryx{\Sigma'},\ebdryx{\Sigma'}\setminus X)$, we prefer to use the latter instead of the former,
	and consider in the text below $\ebdryx{\Sigma'}$  
	to be 
	the compactification resulting from the proof of Theorem \ref{t:main} from this subsection.
\end{remark}

\begin{metaremark}\phantomsection\label{r:meta}
\begin{enumerate}[(a)]
	\item\label{r:meta1}
	One may ask the question why we do not finish the construction at Step \ref{r:equivbdryconstructionsigma}, 
	and simply equip the space $X$ 
	with the combing $\Sigma$.
	The (formal) reason is that in the proof of Theorem \ref{t:main} in this
	subsection 
	we used \cite[Theorem 7.10]{EnWu23} as a black box, in which
	it is 
	the group
	$G$
	that
	is
	the object equipped with the combing,
	while 
	the space
	$X$ is 
	required
	to be $G$-equivariantly homotopy equivalent with $\rips(G)$ and to be an ER.
	This is an example of a difference between the approaches in this paper and in \cite{EnWu23}: while the natural place for 
	(bi)combings in this paper are 
	topological spaces, Engel and Wulff tend to 
	prefer
	to have the combing on the group $G$ or its Rips complex $\rips(G)$ --- while the latter 
	has nice abstract properties, and for it they can perform the homotopy from the definition of the Z-set, \cite[Theorem 5.7]{EnWu23}, 
	it does not have nice topological properties.
	This is the reason why they construct the compactification using $\rips(G)$ and exchange the space that has been compactified for $X$. (See \cite[beginning of Section 7]{EnWu23}).

	\item\label{r:meta2}
	We note that the choice $\alpha':=\alpha$ and $\Sigma':=\Sigma$ 
	satisfies the properties required in Remark \ref{r:equivbdryconstruction}\stpref\ref{r:equivbdryconstructionsigmaprim}
	(as has been shown in the proof 
	of Theorem \ref{t:main} in this subsection), 
	therefore $\ebdryx\Sigma$ is one of the compactifications constructed in the proof
	of Theorem \ref{t:main} in this subsection.
\end{enumerate}
\end{metaremark}

\begin{fact}\label{p:equivbdry0}

	Let $G$ be a group acting geometrically on a proper metric space $X$ that admits a \ccc, $G$-equivariant bicombing $\sigma$, and let $\Sigma$, $\Sigma'$ be as in Remark \ref{r:equivbdryconstruction}. Then:
	
	\begin{enumerate}[(i)]
		\item\label{p:equivbdry01}
		the identity function $\mathrm{id}_X\colon X\to X$ is a morphism from $(X,\Sigma)$ to $(X,\Sigma')$;
		
		\item\label{p:equivbdry02}
		the
		function $\mathrm{id}_X\colon X\to X$
		induces 
		a
		$G$-equivariant 
		homeomorphism 
		$(\mathrm{id}_X)_*\colon \ebdryx\Sigma\to\ebdryx{\Sigma'}$ 
		fixing 
		the
		copies of $X$ in $\ebdryx\Sigma$ and $\ebdryx{\Sigma'}$ pointwise, 
		that is 
		$(\mathrm{id}_X)_*(\delta_x)=\delta_x$ 
		for all $x\in X$.
	\end{enumerate}
	
\end{fact}

\begin{remark}\label{r:equivbdryfunctor}
	\begin{enumerate}[(i)]
		\item (cf.~\cite[Remark 7.15]{EnWu23}) Regarding the action of $G$ on the compactifications mentioned in the statement
		\ref{p:equivbdry02} 
		above,
		we justify below that
		the action of each $g\in G$ on $X$ induces 
		endomorphisms of $(X,\Sigma)$ and $(X,\Sigma')$;
		these
		induce by functoriality, \cite[Corollary 2.18b]{EnWu23},
		the desired
		homeomorphisms $g_*\colon \ebdryx{\Sigma}\to\ebdryx{\Sigma}$ 
		and $g_*\colon \ebdryx{\Sigma'}\to\ebdryx{\Sigma'}$.
		The fact that the action of $g$ induces an endomorphism of $(X,\Sigma)$ has been proved previously in this section in the proof of Theorem \ref{t:main}.
		In order to justify 
		that $g$ induces an endomorphism of $(X,\Sigma')$,
		observe
		that by construction of $\alpha'$ we have that $g\alpha'=\alpha'g$, therefore $g\alpha'\beta=\alpha'g\beta$. The right-hand side is a morphism as a composition of 3 morphisms ($\beta$ is a morphism 
		from
		$(X,\Sigma')$ 
		to
		$(G,\Sigma_G)$
		by Proposition \ref{p:enwumorfizmy}\ref{p:enwumorfizmy1} as a quasi-inverse of $\alpha'$). The left-hand side is at finite distance from $g$, 
		as $g$ acts on $X$ via isometries and $\alpha'\beta$ is at finite distance from the identity of $X$. Therefore, 
		by Proposition \ref{p:enwumorfizmy}\ref{p:enwumorfizmy2}, 
		the action of $g$ on $X$ induces an endomorphism of $(X,\Sigma')$.
		
		\item
		The functoriality of the construction of the compactification in \cite[Corollary 2.18]{EnWu23} works in the following way.
		It is a standard fact that a continuous function $F$ induces a pullback map $F^*$ between the spaces of continuous functions, and a pushforward map $F_*$ between the spaces dual to these spaces of continuous functions.
		It turns out
		that if $F$ is a morphism of combed spaces, then $F_*$ gives a well-defined continuous map between the compactifications.
	\end{enumerate}
\end{remark}

\begin{proof}
	\ref{p:equivbdry01}
	Since the quasi-isometry $\alpha'$ 
	is a morphism from $(G,\Sigma_G)$ to $(X,\Sigma')$,
	and $\beta$ is a morphism from $(X,\Sigma)$ to $(G,\Sigma_G)$ 
	(by Proposition
	\ref{p:enwumorfizmy}\ref{p:enwumorfizmy1},
	as a quasi-inverse of $\alpha$, 
	which we 
	have shown
	to be
	a 
	morphism
	from $(G,\Sigma_G)$ to $(X,\Sigma)$
	in the proof of Theorem \ref{t:main} in this subsection),
	the map 
	$\alpha'\beta$ 
	is a morphism from $(X,\Sigma)$ to $(X,\Sigma')$.
	Since both functions $\alpha'$ and $\alpha$ are of the form $g\mapsto gx$ for some $x\in X$, they are 
	at
	finite distance from each other, 
	so 
	the function $\alpha'\beta$ is at finite distance from $\alpha\beta$, which is at finite distance from the identity $\mathrm{id_X}$ of $X$, which implies by Proposition \ref{p:enwumorfizmy}\ref{p:enwumorfizmy2} that
	$\mathrm{id_X}$ is a morphism from $(X,\Sigma)$ to $(X,\Sigma')$. 
	
	\medskip
	\ref{p:equivbdry02} By Proposition \ref{p:enwumorfizmy}\ref{p:enwumorfizmy1}, the identity $\mathrm{id}_X$ is also a morphism from $(X,\Sigma')$ to $(X,\Sigma)$, 
	therefore,
	by functoriality, \cite[Corollary 2.18b]{EnWu23},
	the 
	induced 
	function $(\mathrm{id}_X)_*\colon \ebdryx{\Sigma}\to\ebdryx{\Sigma'}$
	is a homeomorphism.
	Finally,
	observe
	that for all $x\in X$ we have
	that
	$(\mathrm{id}_X)_*\delta_x=\delta_x$, 
	and 
	for all $g\in G$ we have that
	$g\:\!\mathrm{id}_X=\mathrm{id}_X\:\!g$, 
	which implies that
	$g_*(\mathrm{id}_X)_*=(\mathrm{id}_X)_*g_*$. 
\end{proof}

\subsection{Equivalence of the constructed compactifications}\label{sbs:equivbdry}
In this section we prove
that the constructions of the EZ-structures from Subsections \ref{sbs:deslanez} and \ref{sbs:enwubdez} produce the same compactification.

We note here that the 
pre-`moreover' part of the
proposition below follows from a stronger result, holding in a 
more general setting, \cite[Corollary 8.9]{FuOg20}, 
in a way described in \cite[above Examples 3.27, and 3.27.2]{EnWu23}. 
In view of this, we may see the proof of the proposition below as a more elementary proof of 
a
special case (with the minor addition of $G$-equivariance to the considerations).

\begin{fact}\label{p:equivbdrymain}
	Let $(X,d)$ be a proper metric space that admits a \ccc{} bicombing $\sigma$. Let $\Sigma$ be as in Remark \ref{r:equivbdryconstruction}\stpref\ref{r:equivbdryconstructionsigma}.
	Then there exists a 
	homeomorphism $\tau\colon \bdry{X}{\sigma}\to\ebdryx\Sigma$ fixing 
	the
	copies of $X$ in $\bdry{X}{\sigma}$ and $\ebdryx\Sigma$ pointwise, i.e.~$\tau(x)=\delta_x$ for all $x\in X$. Moreover, if a group $G$ acts on $X$
	in such a way that
	the bicombing $\sigma$ is additionally $G$-equivariant,
	then $\tau$ is additionally $G$-equivariant.
	
\end{fact}

In view of Proposition \ref{p:equivbdry0} and Remark \ref{r:equivbdryconstruction}, we have the following corollary. 

\begin{corollary}\label{c:equivbdry}
	Let $G$ be a group acting geometrically on a proper metric space $X$ that admits a \ccc, 
	$G$-equivariant bicombing $\sigma$, and 
	let
	$\bdry{X}{\sigma}$, $\ebdryx{\Sigma'}$ be the compactifications of $X$ constructed in the proofs of Theorem \ref{t:main} in Subsections \ref{sbs:deslanez}, \ref{sbs:enwubdez}, respectively.
	Then there exists a $G$-equivariant homeomorphism $\tau\colon \bdry{X}{\sigma}\to\ebdryx{\Sigma'}$ fixing the copies of $X$ in $\bdry{X}{\sigma}$ and $\ebdryx{\Sigma'}$ pointwise, i.e.~$\tau(x)=\delta_x$ for all $x\in X$.	
\end{corollary}

\begin{proof}{\sc (of Proposition \ref{p:equivbdrymain})}
	Put $\tau(\bar{x}):=\delta_{\bar{x}}$, where $\delta_{\bar{x}}(f)=\lim_nf(\geod{o}{\bar{x}}(n))$ for all $f\in C_\Sigma(X)$.
	
	Observe that for any $x\in X$ and $n\geq d(o,x)$ we have that $\geod{o}{x}(n)=x$, therefore the definition of $\delta$ from the definition of $\tau$ extends the 
	previous
	definition of $\delta$ from Section~\ref{s:enwubd}
	(from $X$ to the whole $\bdry{X}{\sigma}$),
	in particular the equality $\tau(x)=\delta_x$ holds for all $x\in X$.
	Observe that for any $\bar{x}\in \bdry{X}{\sigma}$ the limit $\lim_nf(\geod{o}{\bar{x}}(n))$ exists, 
	as the convergence of $f\circ \Sigma_n$ to $f$ in the supremum metric implies that the diameters of the sets $(f\circ\geod{o}{\bar{x}})([n,\infty))$ converge to 0
	when $n\to\infty$,
	and 
	that
	$\delta_{\bar{x}}$ is 
	a non-zero multiplicative linear functional 
	on $C_\Sigma(X)$; 
	therefore the map $\tau$ is well-defined.

	\smallskip
	The map $\tau$ is one-to-one by the following argument. Let $\bar{x}\in \bdry{X}{\sigma}$. Define $\dstfun_{\bar{x}}(x):=d_o(\bar{x},x)$ (the metric $d_o$ has been introduced in 
	equation \eqref{eq:do}).
	The function $\dstfun_{\bar{x}}$ is bounded (by 1) and continuous on $X$.
	Furthermore, for all $x\in X$ and $n\in\N$ we have that $|\dstfun_{\bar{x}}(\Sigma_n(x))-\dstfun_{\bar{x}}(x)|\leq 2^{-n+1}$,
	as for all $i\leq n$ we have 
	by consistency of $\sigma$
	that $\geod{o}{\Sigma_n(x)}(i)=\geod{o}{\geod{o}{x}(n)}(i)=\geod{o}{x}(i)$, so  the $i$-th summands
	for $i\leq n$
	in the definitions of $d_o(\bar{x},x)$ and $d_o(\bar{x},\Sigma_n(x))$ are the same; this implies that $\dstfun_{\bar{x}}\circ\Sigma_n$ converges to $\dstfun_{\bar{x}}$ in the supremum metric. 
	To see that $\dstfun_{\bar{x}}$ has bounded variation, 
	take
	$R>0$, $x\in X$ and $y\in\overline{B}(x,R)$, and 
	let $N\in\N$ be the integer part of $d(o,x)$.
	By 
	the triangle inequality and
	Proposition \ref{f:neweq52}
	we have 
	that
	\begin{multline*} |\dstfun_{\bar{x}}(x)-\dstfun_{\bar{x}}(y)|\leq\sum_{n=1}^N2^{-n}d(\geod{o}{x}(n),\geod{o}{y}(n))+\sum_{n=N+1}^\infty2^{-n}\\
		\leq \sum_{n=1}^{\infty}\frac{2Rn}{d(o,x)}2^{-n}
		+2^{-N}
		\leq\frac{4R}{d(o,x)}+2^{-d(o,x)+1},
	\end{multline*}
	which tends to $0$ when $d(o,x)$ tends to $\infty$.
	Therefore $\dstfun_{\bar{x}}\in C_\Sigma(X)$. 
	By continuity of $d_o$, for any $\bar{y}\in\bdry{X}{\sigma}$ we have that
	$\tau(\bar{y})(\dstfun_{\bar{x}})=\delta_{\bar{y}}(\dstfun_{\bar{x}})
	=d_o(\bar{y},\bar{x})$. Therefore, if $\tau(\bar{y})=\tau(\bar{x})$, then by an application of both sides to $\dstfun_{\bar{x}}$ we obtain that $d_o(\bar{y},\bar{x})=d_o(\bar{x},\bar{x})$, which is equal to~$0$, so $\bar{x}=\bar{y}$.

	The map $\tau$ is onto $\ebdryx\Sigma$ by the following argument. Assume that there exists a maximal
	proper
	ideal $I$ of $C_\Sigma(X)$ such that $I\neq\ker\delta_{\bar{x}}$ for any $\bar{x}\in\bdry{X}{\sigma}$. Then, because $I$ is closed under 
	the
	multiplication
	by the elements of the 
	(whole) 
	algebra $C_\Sigma(X)$,
	we can pick a family of functions $\{f_{\bar{x}}:\bar{x}\in\bdry{X}{\sigma}\}\subseteq I$ such that $f_{\bar{x}}$ is non-negative and $\delta_{\bar{x}}(f_{\bar{x}})=1$. If 
	$x\in X$,
	then there exists 
	$r_x>0$
	such that 
	$f_{\bar{x}}\geq 1/2$ 
	on 
	$B(x,r_x)$.
	If $\bar{x}\in\partial_\sigma X$, then, since $f_{\bar{x}}\in C_\Sigma(X)$, we can choose $t_{\bar{x}}\in\N$ 
	such that
	for all $x\in X$ with $d(o,x)\geq t_{\bar{x}}$ we have that $|f_{\bar{x}}(x)-f_{\bar{x}}(\Sigma_{t_{\bar{x}}}(x))|\leq 1/6$ (as $f_{\bar{x}}\circ\Sigma_n\to f_{\bar{x}}$) and $|f_{\bar{x}}(x)-f_{\bar{x}}(y)|\leq 1/6$ for all $y\in X$ with $d(x,y)\leq 1$ (as $f_{\bar{x}}$ has bounded variation). 
	Then 
	we have by the triangle inequality that for all $y\in U_o(\bar{x},t_{\bar{x}},1)\cap X$ and $n\geq t_x$
	\begin{multline*}
		|f_{\bar{x}}(y)-f_{\bar{x}}(\geod{o}{\bar{x}}(n))|
		\leq |f_{\bar{x}}(y)-f_{\bar{x}}(\geod{o}{y}(t_{\bar{x}}))|+|f_{\bar{x}}(\geod{o}{y}(t_{\bar{x}}))-f_{\bar{x}}(\geod{o}{\bar{x}}(t_{\bar{x}}))|\\
		+|f_{\bar{x}}(\geod{o}{\bar{x}}(t_{\bar{x}}))-f_{\bar{x}}(\geod{o}{\bar{x}}(n))|
		\leq 1/6+1/6+1/6=1/2,
	\end{multline*}
	therefore, passing to the limit with $n$, we obtain that 
	$f_{\bar{x}}(y)\geq 1-1/2=1/2$.	
	Observe that the family $\{U_o(\bar{x},t_{\bar{x}},1):\bar{x}\in\partial_\sigma X\}\cup\{B(x,r_{x}):x\in X\}$ is an open cover of $\bdry{X}{\sigma}$, thus we can choose a finite subcover $\{U_o(\bar{x}_i,t_{\bar{x}_i},1):i=1,\ldots,m_1\}\cup\{B(x_j,r_{x_j}):j=1,\ldots,m_2\}$. Then the function $f:=f_{\bar{x}_1}+\ldots+f_{\bar{x}_{m_1}}+f_{x_1}+\ldots+f_{x_{m_2}}$ belongs to $I$ and $1/2\leq f(x)$
	for all $x\in X$. Using the fact that
	\begin{equation*}
		\left|\frac{1}{f(x)}-\frac{1}{f(y)}\right|=\frac{|f(x)-f(y)|}{|f(x)f(y)|}\leq4\;\!|f(x)-f(y)| 
	\end{equation*}
	for all $x,y\in X$,
	one may 
	verify
	that the function $1/f$ belongs to $C_\Sigma(X)$, therefore $1=(1/f)\cdot f\in I$ and $I=C_\Sigma(X)$.
	This contradicts properness of $I$.

	Since 
	$\bdry{X}{\sigma}$
	is compact and $\tau$ is bijective, in order to prove that $\tau$ is a homeomophism, it suffices to show that it is continuous. Consider any subbase open subset $U_{f,a,\epsilon}:=\{\delta_{\bar{y}}\in\ebdryx\Sigma:|\delta_{\bar{y}}(f)-a|<\epsilon\}$ of $\ebdryx\Sigma$, 
	where $f\in C_\Sigma(X)$, $a\in\C$ and $\epsilon>0$.
	Then $\tau^{-1}(U_{f,a,\epsilon})=\{\bar{y}\in\bdry{X}{\sigma} :|\delta_{\bar{y}}(f)-a|<\epsilon\}$. 
	Consider a point $\bar{x}\in\bdry{X}{\sigma}$ such that $\epsilon':=|\delta_{\bar{x}}(f)-a|<\epsilon$.
	If $\bar{x}\in X$, then by continuity of $f$ there exists $\eta>0$ such that 
	\begin{equation*}
		\{\delta_y(f):y\in B(\bar{x},\eta)\}
		=f(B(\bar{x},\eta))\subseteq
		B_\C(\delta_{\bar{x}}(f),\epsilon-\epsilon')
		\subseteq B_\C(a,\epsilon)
	\end{equation*}
	If $\bar{x}\in\partial_\sigma X$, then, as in the proof that $\tau$ is onto, by the fact that $f\circ\Sigma_n\to f$ and that $f$ has bounded variation, for large enough $t$ we have that 
	\begin{equation*}
		f(U_o(\bar{x},t,1)\cap X)
		\subseteq\overline{B}_\C(\delta_{\bar{x}}(f),(\epsilon-\epsilon')/2) 
		\subseteq\overline{B}_\C(a,\epsilon'+(\epsilon-\epsilon')/2),
	\end{equation*}
	therefore 
	$\{\delta_{\bar{y}}(f):\bar{y}\in U_o(\bar{x},t,1)\}\subseteq \overline{f(U_o(\bar{x},t,1)\cap X)}\subseteq B_\C(a,\epsilon)$.
	
	\medskip
	Now we prove the `moreover' part.
	Observe that for any $g\in G$ 
	and 
	$x\in X$ we have $g_*\tau(x)=g_*\delta_x=\delta_{gx}=\tau(gx)$.
	Since $\tau$ is continuous, 
	and 
	the actions of $G$ on $\bdry{X}{\sigma}$ and $\ebdryx{\Sigma}$ are via continuous functions, and $X$ is dense in $\bdry{X}{\sigma}$,
	we have that 
	$g_*\tau(\bar{x})=\tau(g\bar{x})$ for any $\bar{x}\in\bdry{X}{\sigma}$. 
\end{proof}

\newcommand{\tnua}{\square}
\newcommand{\tnub}{\diamondsuit}

\newcommand{\lbla}{\circ}
\newcommand{\lblb}{\bullet}

\newcommand{\simple}{std}

\section{Non-uniqueness of boundary}\label{s:nonunique}

In this section we adapt the classical example by Croke and Kleiner \cite{CrKl00} to prove the following theorem on non-uniqueness of the boundary defined in Section \ref{s:deslan} in the case of injective groups.
(For the definition of 
injective metric space, see the paragraph above Corollary \ref{c:injective}.)

\begin{theorem}[Theorem \ref{t:nonuniqueintro}]\label{t:nonunique}
There exists a group $G$ acting geometrically on two
proper
finite-dimensional
injective metric spaces $X^\tnub,X^\tnua$ with 
convex
bicombings $\sigma^\tnub,\sigma^\tnua$, respectively, such that $\partial_{\sigma^\tnub} X^\tnub$ and $\partial_{\sigma^\tnua} X^\tnua$ are not homeomorphic. 
\end{theorem}

\begin{remark}\label{r:uniquebicombing}
As it 
has been 
discussed in
Remark \ref{u:equivariant}
and
in
the proof of Corollary~\ref{c:injective}\ref{c:injective1}, 
\linebreak[1]
a
proper 
injective 
metric
space $X$
that is
finite-dimensional
---
or, more generally, such that 
each of its
bounded subsets 
is of finite dimension
---
admits a unique convex bicombing, which is additionally consistent, reversible and equivariant with respect to the full isometry group $\mathrm{Iso}(X)$ of $X$.

In a CAT(0) space $X$, each pair of points is connected by a unique geodesic; 
these geodesics
give
the
unique bicombing on $X$, see \cite[Proposition II.1.1.4(1)]{BrHae99}. This bicombing therefore is automatically consistent, reversible and $\mathrm{Iso}(X)$-equivariant. It is also convex, see \cite[Proposition II.2.2.2]{BrHae99}.

Further in this section, 
we will use 
these uniqueness results
without mentioning. 
\end{remark}

We say that two bicombings $\sigma^\lbla,\sigma^\lblb$ on a complete geodesic metric space $X$ \emph{have the same trajectories} if $\im\,\sigma^\lbla_{xx'}=\im\,\sigma^\lblb_{xx'}$ for all $x,x'\in X$.

\smallskip
The strategy of the proof of Theorem \ref{t:nonunique} is as follows.
We take a group acting on two spaces with non-homeomorphic boundaries, described in \cite{CrKl00} (namely, the group of deck transformations acting on the universal covers of the spaces in Figure \ref{fig:glue}) and replace the original,
piecewise-$\ell^2$
metrics 
on these spaces with 
piecewise-$\ell^\infty$ 
metrics.
It is then sufficient, see Proposition \ref{p:bilipbic}, to show that the resulting spaces are injective and equipped with bicombings having the same trajectories as the original CAT(0) ones.
It turns out that 
these properties may be checked locally, to which end Lemma \ref{l:5sp} serves.
Lemmas \ref{l:fixpoint}, \ref{l:sigmaconvex}, \ref{l:conservation} and Proposition \ref{f:loctoglobgeod} are some preparatory lemmas useful in the proof of Lemma \ref{l:5sp}.

\medskip
We begin with the following simple, yet very useful observation.

\begin{lemma}\label{l:fixpoint}
	Let $\varphi$ be an isometry of a metric space $X$
	that
	possesses a $\varphi$-equivariant bicombing $\sigma$. Then the fixpoint set $\Fix(\varphi)$ of $\varphi$ is $\sigma$-convex. 
\end{lemma}

\begin{proof}
For any $x,y\in\Fix(\varphi)$ we have that $\sigma_{xy}=\sigma_{\varphi(x)\varphi(y)}=\varphi\circ\sigma_{xy}$. 
\end{proof}

We 
add some simple observations to the work of Miesch on gluings of injective metric spaces.
Following \cite{Miesch15}, we call a subset $A$ of a metric space $(X,d)$ \emph{strongly convex} whenever for all $x,y\in A$ the metric interval $\{z\in X:d(x,y)=d(x,z)+d(z,y)\}$ is contained in $A$, and \emph{externally hyperconvex}
(recall the 
discussion 
about names
from
around the definition 
of injective metric space, above Corollary~\ref{c:injective})
whenever for every family of points ${x_i\in X}$ 
and radii $r_i>0$ that satisfies $d(x_i,x_j)\leq r_i+r_j$ and $d(x_i,A)\leq r_i$, 
the set
$A\cap\bigcap\overline{B}(x_i,r_i)$
is non-empty.
For $\R^2$ with the 
$\ell^\infty$-metric, 
a standard example of a strongly convex subset is the diagonal line $\{(d,d):d\in\R\}$, 
and of an externally hyperconvex subset is the horizontal line $\{(x,0):x\in\R\}$.

Further in this chapter,
by an \emph{\simple-gluing} of the spaces $(X_\lambda,d_\lambda)_{\lambda\in\Lambda}$ along a subspace $A$ we mean the metric space $X$ obtained in the following process.
We assume that we have a family of metric spaces $(X_\lambda,d_\lambda)_{\lambda\in\Lambda}$,
a space $(A,d_A)$, 
and isometric embeddings $\imath_\lambda\colon(A,d_A)\to(X_\lambda,d_\lambda)$ for all $\lambda\in\Lambda$,
such that $\imath_\lambda(A)$ is closed in $X_\lambda$.
Then $X$ arises as the image of the gluing map $\pi$ defined as the quotient map 
of the relation $\sim$ on $\bigsqcup_{\lambda\in\Lambda}X_\lambda$
given by $\imath_\lambda(a)\sim\imath_{\lambda'}(a)$ for all $\lambda,\lambda'\in\Lambda$ and $a\in A$,
with the standard gluing metric.
Note that the spaces $A$ and $X_\lambda$ 
can be
isometrically embedded into $X$
using $\pi$ (and the $\imath_\lambda$), 
and therefore we will identify these spaces with their images in $X$.

\smallskip
In the following lemma, the statements about CAT(0)-ness and about injectivity, see \cite{Miesch15}, of the spaces resulting from gluings have been known previously.

\begin{lemma}\label{l:sigmaconvex}
Let $X$
be the \simple-gluing of 
metric
spaces
$(X_\lambda,d_\lambda)_{\lambda\in\Lambda}$ along some
space
$A$. 
In any of the 3 cases below:
\begin{enumerate}
\item the spaces 
$X_\lambda$ are CAT(0), and $A$ is a closed and convex subset of each 
of the
$X_\lambda$;
\item the spaces 
$X_\lambda$ are injective, $X$ is 
proper and
finite-dimensional, 
and $A$ is externally hyperconvex 
(therefore, automatically, closed) 
in each 
of the
$X_\lambda$;
\item the spaces $X_\lambda$ are injective, $X$ is 
proper and
finite-dimensional,
and $A$ is closed and strongly convex in each 
of the
$X_\lambda$;	
\end{enumerate}
the space $X$ is CAT(0) (case 1) or injective (cases 2 and 3), thus 
admits
a convex bicombing $\sigma$;
and
for all $\Lambda_0\subseteq\Lambda$ the space $X_0:=A\cup\bigcup_{\lambda\in\Lambda_0}X_\lambda$ is $\sigma$-convex, in particular $\sigma|_{X_0\times X_0\times[0,1]}$ is the convex bicombing on $X_0$. 
\end{lemma}

\begin{proof}
For each $\lambda\in\Lambda$ consider copies $X^1_\lambda,X^2_\lambda,X^3_\lambda$ of the space $X_\lambda$, 
and let the space $\widetilde{X}$ be 
the 
\simple-gluing of the family $\{X_\lambda^i:\lambda\in\Lambda,1\leq i\leq 3\}$  along $A$. 
Note that the space $X$ is a subspace of $\widetilde{X}$, 
and the latter consists of 3 copies of $X$ \simple-glued along $A$.
Applying \cite[Theorem II.11.3]{BrHae99},
\cite[Theorem 1.3 and Theorem 1.1]{Miesch15}
to the cases 1, 2 and 3, respectively, we get that the spaces $X_0$, $X$, $\widetilde{X}$ are all CAT(0), or 
proper, 
finite-dimensional (see e.g.~\cite[Theorem 3.1.4 and Proposition 3.1.7]{EngelkingBook78})
and injective, thus 
admit
convex bicombings;
denote by $\widetilde{\sigma}$ the convex bicombing on $\widetilde{X}$. Consider the action of the group $\mathrm{Sym}(\{1,2,3\})^\Lambda$ 
with the $\lambda$-th coordinate acting by permuting copies $X_\lambda^j$ of the spaces $X_\lambda$.
Observe that $X$ 
(upon the identification with the subspace $\bigcup_{\lambda\in\Lambda}X^1_\lambda$ of $\widetilde{X}$)
is the set of fixed points of $((1)(2\,3))_{\lambda\in\Lambda}$, therefore
by Lemma \ref{l:fixpoint} 
it is $\widetilde{\sigma}$-convex and $\sigma:=\widetilde\sigma|_{X\times X\times[0,1]}$ is the convex bicombing on $X$. Similarly, one can see that there exists an element of $\mathrm{Sym}(\{1,2,3\})^\Lambda$ with fixed point set equal to $X_0$, therefore by Lemma \ref{l:fixpoint} the space $X_0$ is $\widetilde{\sigma}$-convex, thus $\sigma$-convex, and $\widetilde{\sigma}|_{X_0\times X_0\times[0,1]}=\sigma|_{X_0\times X_0\times[0,1]}$ is the convex bicombing on $X_0$.
\end{proof}

\begin{lemma}\label{l:conservation}
Let $X$
be the \simple-gluing of
injective spaces $(X_\lambda,d_\lambda)_{\lambda\in\Lambda}$ along some space $A$. 
	
\begin{enumerate}[(i)]
\item Assume that $A$ is externally hyperconvex in each of the $X_\lambda$. Then for all subsets $\Lambda_0\subseteq\Lambda$ the set $X_0:=A\cup\bigcup_{\lambda\in\Lambda_0}X_\lambda$ is an externally hyperconvex subset of $X$.
\label{l:conservation1}
\item Let $A$ be closed in each of the $X_\lambda$. Suppose $B\subseteq X_{\lambda_0}$ for some $\lambda_0\in\Lambda$ is strongly convex in $X_{\lambda_0}$ and $|A\cap B|\leq1$. Then $B$ is strongly convex in $X$.
\label{l:conservation2}
\end{enumerate}
\end{lemma}

\begin{proof}
\ref{l:conservation1} Let $\{\overline{B}(x_i,r_i)\}_{i\in I}$ be a collection of closed balls in $X$ with $d(x_i,x_j)\leq r_i+r_j$ and $d(x_i,X_0)\leq r_i$. If this collection satisfies $d(x_i,A)\leq r_i$ for all $i\in I$, then the claim follows by the fact that $A$ is externally hyperconvex in $X$ by \cite[Theorem 1.3]{Miesch15}. 
Otherwise, 
there exist $i_0\in I$ and $\lambda_0\in\Lambda_0$ such that $\overline{B}(x_{i_0},r_{i_0})\subseteq X_{\lambda_0}$. 
By  \cite[Theorem 1.3]{Miesch15}, 
the space $X$ is injective, therefore $\emptyset\neq\bigcap_{\lambda\in\Lambda_0}\overline{B}(x_i,r_i)\subseteq \overline{B}(x_{i_0},r_{i_0})\subseteq X_{\lambda_0}\subseteq X_0$. The claim follows.
	
\medskip
\ref{l:conservation2} Take a geodesic $\gamma$ in $X$ with endpoints in $B$. Since $A$ is closed in each of the $X_\lambda$, the set $X_{\lambda_0}$ is closed in $X$, therefore the preimage $\gamma^{-1}(X\setminus X_{\lambda_0})$ is open, thus it is a 
union
of disjoint open intervals. 
Replace $\gamma$ on each of such intervals 
with a geodesic contained in $X_{\lambda_0}$, obtaining a new geodesic $\gamma'$ with the same endpoints as $\gamma$, which is contained in $X_{\lambda_0}$, and thus in $B$, by strong convexity. This implies that all 
of
the endpoints of the intervals where $\gamma$ went out of $X_{\lambda_0}$ are contained in $B$. Since $|A\cap B|\leq 1$, it follows that $\gamma=\gamma'$, so $\gamma$ is contained in $B$.  
The claim follows.
\end{proof}

We note that the following can be derived from \cite{Miesch17}, where the proof is 
more involved, 
as the setting is more general.

\begin{fact}\label{f:loctoglobgeod}
	Let $X$ be a 
	metric space 
	that
	admits a 
	\ccc{}
	bicombing $\sigma$. Assume that $c\colon[0,1]\to X$ is a constant speed geodesic such that locally it is a $\sigma$-geodesic, i.e.~there exists an open cover $\mathcal{U}$ of $[0,1]$ such that $\im\,\sigma_{c(s)c(t)}=\im\,c|_{[s,t]}$ for all $U\in\mathcal{U}$ and $s,t\in U$ satisfying $s\leq t$. Then $c$ is the $\sigma$-geodesic $\sigma_{c(0)c(1)}$.		
\end{fact}

\begin{proof}
	Consider the following statement $\mathtt{S}(l)$,
	where $l\in[0,1]$: 
	for all $s,t\in[0,1]$ such that $s\leq t$ and $t-s\leq l$ 
	the equality
	$\im\,\sigma_{c(s)c(t)}=\im\,c|_{[s,t]}$
	holds.
	By compactness of $[0,1]$, 
	the statement $\mathtt{S}(\epsilon)$ holds for a sufficiently small $\epsilon>0$.
	It is sufficient to show that for any $l\in[0,1]$ the statement $\mathtt{S}(2l/3)$ implies $\mathtt{S}(l)$. Take any $s,t\in[0,1]$ such that $s\leq t$ and $t-s\leq l$. Let $p:=c((2s+t)/3)$, $q:=c((s+2t)/3)$ and $p':=\sigma_{c(s)c(t)}(1/3)$, $q':=\sigma_{c(s)c(t)}(2/3)$. By conicality and $\mathtt{S}(2l/3)$, we have $d(p,p')\leq d(q,q')/2$. Similarly,
	$d(q,q')\leq d(p,p')/2$. Therefore $d(p,p')=0=d(q,q')$, 
	and, by consistency of $\sigma$, $\im\,\sigma_{c(s)c(t)}=\im\,c|_{[s,t]}$.
\end{proof}

For a space $X$ obtained by gluing up to 3 copies $P_j:=\{(x,y)^{P_j}:x,y\in\R\}$ of $\R^2$,
where $j=1,2,3$,
we
denote by $X^\infty$ (resp.~$X^2$) the 
space $X$ 
equipped
with 
the
metric arising from 
putting the
$\ell^\infty$-metric (resp.~the $\ell^2$-metric) on each 
of
the $P_j$. 

The main technical lemma we use to adapt the Croke--Kleiner example is as follows.

\newcommand{\spsl}{\,{\big/}}

\begin{lemma}\label{l:5sp}
Let $X$ be one of the following spaces: 
\begin{enumerate}
\item $\R^2$,
\label{l:5sp1}

\item $P_1\sqcup P_2\spsl\{(x,0)^{P_1}=(x,0)^{P_2}:x\in\R\}$,
\label{l:5sp2} 

\item $P_1\sqcup P_2\spsl\{(d,d)^{P_1}=(d,d)^{P_2}:d\in\R\}$,
\label{l:5sp3}

\item $P_1\sqcup P_2\sqcup P_3\spsl\{(x,0)^{P_1}=(x,0)^{P_2}:x\in\R\}\cup\{(0,y)^{P_2}=(0,y)^{P_3}:y\in\R\}$, 
\label{l:5sp4}

\item $P_1\sqcup P_2\sqcup P_3\spsl\{(x,0)^{P_1}=(x,0)^{P_2}:x\in\R\}\cup\{(d,d)^{P_2}=(d,d)^{P_3}:d\in\R\}$. 
\label{l:5sp5}
\end{enumerate}
Then $X^2$ is CAT(0), $X^\infty$ is injective, and the convex bicombings $\sigma^2$ of $X^2$ and $\sigma^\infty$ of $X^\infty$ have the same trajectories. 
\end{lemma}

\begin{proof}
The space $X^2$ is CAT(0) in each of the cases by \cite[Remark II.11.2.2]{BrHae99}.

\medskip
Let $X$ be as in \ref{l:5sp1}. 
It is well-known that the space $X^\infty$ 
is
injective.
Both 
$\sigma^\infty$
and 
$\sigma^2$
consist of linear segments --- one can refer to \cite[Theorem 3.3]{DeLa15}.

\medskip
Let $X$ be as in \ref{l:5sp2}. 
The resulting space $X^\infty$ is injective by 
Lemma \ref{l:sigmaconvex},
as
$\{(x,0):x\in\R\}$ is an externally hyperconvex subset of $\R^2$ with the $\ell^\infty$-metric. 
Furthermore, for any $p\in\{2,\infty\}$ the space $X^p$ restricted to the union of any pair of the 4 closed halfplanes induced in $P_1$ and $P_2$ by the 
$x$-axis 
(i.e.~$\{(x,y)^{P_j}:y\,\!R\:\!0,x\in\R\}$, where $R\in\{\leq,\geq\}$ and $j=1,2$) is isometric to $\R^2$ with the $\ell^p$-metric; therefore, 
by Lemma \ref{l:sigmaconvex} and 
case \ref{l:5sp1}, 
we obtain that $\sigma^2=\sigma^\infty$.

\smallskip
Let $X$ be as in \ref{l:5sp3}.
The argument is analogous to the one in \ref{l:5sp2}, 
with the difference that the space $X^\infty$ is injective 
as the set $\{(d,d):d\in\R\}$ is strongly convex in $\R^2$ with the $\ell^\infty$-metric.

\begin{figure}[h]
\includegraphics[width=\textwidth]{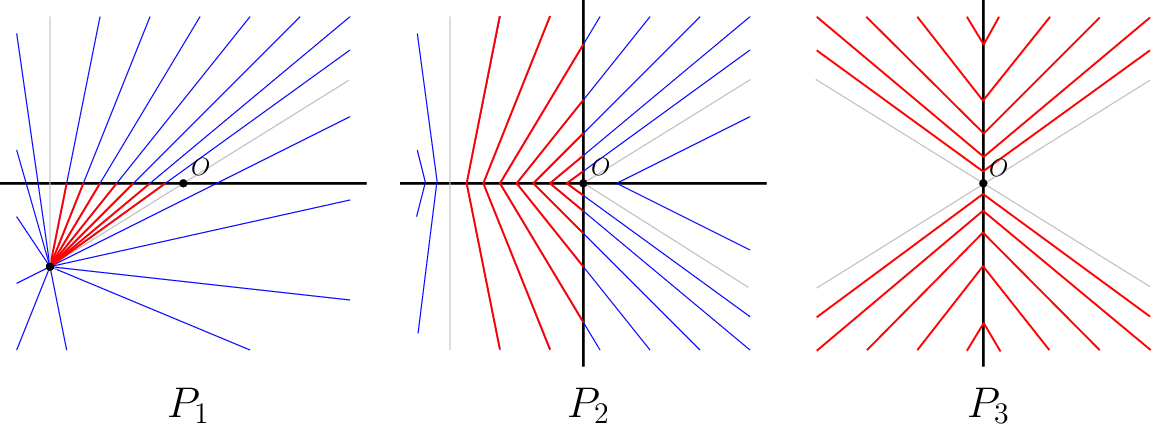}
\caption{Paths, whose intersections with $P_1\cup P_2$ and $P_2\cup P_3$ are linear segments, originating in some point $a\in P_1\setminus(P_2\cup P_3)$ and omitting $O$ (the \textcolor{red}{red} ones are these that reach $P_3$).}	
\label{fig:rayscase4}
\end{figure}

\medskip
Let $X$ be as in \ref{l:5sp4}.
The space $X$ may be seen as the \simple-gluing of the spaces $X_{1,2}$ and $X_{2,3}$ along $P_2$, where the space $X_{1,2}$ is the \simple-gluing of $P_1$ with $P_2$ along the $x$-axis and $X_{2,3}$ is the \simple-gluing of $P_2$ with $P_3$ along the $y$-axis. 
The set $P_2$ is externally hyperconvex in both $X_{1,2}^\infty$ and $X_{2,3}^\infty$ 
by Lemma \ref{l:conservation}\ref{l:conservation1};
therefore,
by Lemma \ref{l:sigmaconvex},
the
space $X^\infty$ is injective, 
and the
subsets $P_1\cup P_2$ and $P_2\cup P_3$ are both $\sigma^\infty$-convex and $\sigma^2$-convex 
in $X$.
Therefore, 
since case \ref{l:5sp2} applies to both $X_{1,2}$ and $X_{2,3}$,
it remains to show that for 
all
$a\in P_1\setminus(P_2\cup P_3)$ 
and
$b\in P_3\setminus(P_1\cup P_2)$ we have $\im\,\sigma^\infty_{ab}=\im\,\sigma^2_{ab}$. Let $O:=(0,0)^{P_1}(=(0,0)^{P_2}=(0,0)^{P_3})$ and 
$\sigma\in\{\sigma^\infty,\sigma^2\}$.
By case \ref{l:5sp2} (applied to $X_{1,2}$ and to $X_{2,3}$),
the sets $\im\,\sigma_{ab}\cap(P_1\cup P_2)$ and $\im\,\sigma_{ab}\cap(P_2\cup P_3)$
are linear segments (when considered as subsets 
of 
appropriate 
planes in $X_{1,2}$ or $X_{2,3}$,
see Figure \ref{fig:rayscase4}).
If there exists a path $\varpi$ from $a$ to $b$ such that $\varpi$ omits $O$, and $\im\,\varpi\cap(P_1\cup P_2)$ and $\im\,\varpi\cap(P_2\cup P_3)$ are linear segments, then, after an appropriate reparametrisation, $\varpi$ is locally a $\sigma$-geodesic, therefore by Proposition \ref{f:loctoglobgeod} is the $\sigma$-geodesic $\sigma_{ab}$. 
If there is 
no
such 
path 
from $a$ to $b$, then  the image $\im\,\sigma_{ab}$ contains $O$, and therefore is equal to the chain of segments with vertices $a$, $O$, $b$.
The case follows.

\smallskip
Let $X$ be as in 
\ref{l:5sp5}.
By Lemma \ref{l:conservation}\ref{l:conservation2},
the subset $\{(d,d)^{P_2}:d\in\R\}$ is 
strongly convex
in 
the \simple-gluing of $P_1$ with $P_2$ along the $x$-axis;
therefore, 
by 
Lemma \ref{l:sigmaconvex}, 
the space
$X^\infty$ is injective, 
and the subset $P_1\cup P_2$ is 
$\sigma^2$-convex and $\sigma^\infty$-convex in $X$.
The subset $P_2\cup P_3$ is $\sigma^2$-convex and $\sigma^\infty$-convex in $X$ as a consequence of Lemma \ref{l:fixpoint} applied to the map $X\to X$ fixing $P_2\cup P_3$ pointwise and reflecting $P_1$ with respect to its $x$-axis.
The remaining part of the argument is analogous to the one in \ref{l:5sp4}.
(Side note:
despite the similarities in the proofs, 
the analogue of Figure \ref{fig:rayscase4} for case \ref{l:5sp5} is 
substantially 
different from 
Figure \ref{fig:rayscase4},
e.g.~it is not just a sheared version of it.)
\end{proof}

\begin{fact}\label{p:bilipbic}
	Let $d_\lbla$ and $d_\lblb$ be two complete metrics on 
	a topological space
	$X$, 
	and assume that for $i\in\{\lbla,\lblb\}$ the space $(X,d_i)$ admits a 
	\ccc{}
	bicombing $\sigma^i$.
	Assume
	that 
	$\sigma^\lbla$ and $\sigma^\lblb$ 
	have the same trajectories. 
	Then the identity $\mathrm{id}_X$ continuously extends to a homeomorphism $\iota\colon\bdry{X}{\sigma^\lbla}\to\bdry{X}{\sigma^\lblb}$.
\end{fact}
\begin{proof}
	Consider a $\sigma^\lbla$-ray 
	$\xi_\lbla$.
	Let $l_\lblb(t):=d_\lblb(\xi_\lbla(0),\xi_\lbla(t))$. 
	Since the bicombings $\sigma^\lbla$ and $\sigma^\lblb$ have the same trajectories, 
	the function 
	$l_\lblb(t)$
	is increasing. We show that 
	$l_\lblb(t)$ 
	is unbounded. 
	If it was bounded, then it would have some limit.
	Therefore, 
	$(\xi_\lbla(n))_{n\in\N}$
	would be a Cauchy sequence in 
	$(X,d_\lblb)$, 
	so by completeness of the metric 
	$d_\lblb$
	on the space $X$, 
	it would have
	a limit in $X$. 
	On the other hand, 
	$\{\xi_\lbla(n):n\in\N\}$ 
	is a discrete set in 
	$(X,d_\lbla)$, 
	therefore 
	the sequence
	$(\xi_\lbla(n))_{n\in\N}$
	is not convergent in $X$ 
	---
	a contradiction.
	Note that the argumentation contained in the previous part of this paragraph works with $\lbla$ and $\lblb$ swapped.
	 
	Fix a basepoint $o\in X$.
	Let
	the map $\iota\colon\bdry{X}{\sigma^\lbla}\to\bdry{X}{\sigma^\lblb}$ 
	be
	induced by the map that 
	is the identity on $X$
	and 
	assigns to
	the 
	$\sigma^\lbla$-ray 
	$\geod{o}{\bar{x}}^\lbla$
	the $\sigma^\lblb$-ray 
	originating in $o$ 
	with the same 
	image
	as 
	$\geod{o}{\bar{x}}^\lbla$.
	By the previous paragraph, 
	the map $\iota$ is well-defined 
	and
	restricts to
	a bijection from the boundary $\partial_{\sigma^\lbla}X$ to the boundary $\partial_{\sigma^\lblb}X$.
	
	Now we show continuity of $\iota$. 
	Clearly, $\iota$
	is 
	continuous 
	at
	every point of
	$X$, as $X$ is an open subset of $\bdry{X}{\sigma^\lbla}$ and $\iota$ restricted to $X$ is the identity.
	Consider any $\bar{x}\in\partial_{\sigma^\lbla}X$ and
	$t_\lblb,\epsilon_\lblb>0$. Let $t_\lbla:=d_\lbla(\geod{o}{\iota(\bar{x})}^\lblb(t_\lblb))$ (so that, in particular, $\geod{o}{\iota(\bar{x})}^{\lblb}(t_\lblb)=\geod{o}{\bar{x}}^{\lbla}(t_\lbla)$).
	Since the maps $\bdry{X}{\sigma^\lbla}\ni\bar{y}\mapsto\geod{o}{\bar{y}}^\lbla(t_\lbla)\in X$ and $X\ni y\mapsto d_\lblb(y,\geod{o}{\bar{x}}^\lbla(t_\lbla))\in\R$ are continuous,
	the set $U:=\{\bar{y}\in\bdry{X}{\sigma^\lbla}:d_\lblb(\geod{o}{\bar{y}}^\lbla(t_\lbla),\geod{o}{\bar{x}}^\lbla(t_\lbla))<\epsilon_\lblb\}$ is an open neighbourhood of $\bar{x}$ in $\bdry{X}{\sigma^\lbla}$.
	Then, for any $\bar{y}\in U$, 
	by Proposition \ref{f:neweq52} 
	applied for the metric $d_\lblb$ and bicombing $\sigma^\lblb$ to 
	the points $\geod{o}{\iota(\bar{x})}^{\lblb}(t_\lblb)(=\geod{o}{\bar{x}}^{\lbla}(t_\lbla))$
	and $\geod{o}{\bar{y}}^{\lbla}(t_\lbla)$,
	and radius $t_\lblb$, 
	we obtain that
	$d_\lblb(\geod{o}{\iota(\bar{x})}^{\lblb}(t_\lblb),\geod{o}{\iota(\bar{y})}^{\lblb}(t_\lblb))\leq 2d_\lblb(\geod{o}{\bar{x}}^{\lbla}(t_\lbla),\geod{o}{\bar{y}}^{\lbla}(t_\lbla))<2\epsilon_\lblb$.
	Similarly,
	the inverse $\iota^{-1}$ is continuous as well.
\end{proof}

Now we are ready to prove the main theorem of this section.

\begin{proof} ({\sc of Theorem \ref{t:nonunique} (Thm.~\ref{t:nonuniqueintro})})
\begin{figure}[h]
	\centering
	\includegraphics[width=\textwidth]{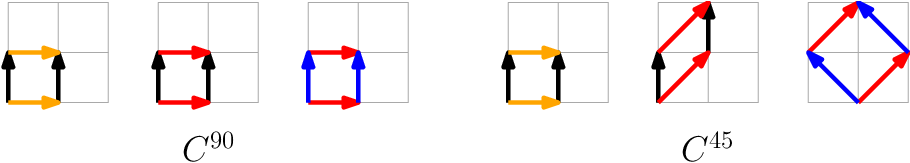}
	\caption{The tori glued to form complexes $C^{90}$ and $C^{45}$, placed on the plane. The grey lines represent the 
	$\Z\!\times\!\Z\:\!$--grid
	in
	$\R\!\times\!\R$.}
	\label{fig:glue}
\end{figure}
Consider the complexes $C^{90}$ and $C^{45}$, each consisting of appropriately glued three tori, as described on Figure \ref{fig:glue}. 
Let $X^{\alpha,p}$ for $\alpha\in\{45,90\}$ and $p\in\{2,\infty\}$ be the universal cover of $C^\alpha$ with the metric being the extension of the local metric from $C^\alpha$ in the case when the underlying planes are endowed with the $\ell^p$-metric.
The fundamental group $G:=\pi_1(C^{90})(=\pi_1(C^{45}))$ acts geometrically on each of these four spaces.
Since each point in each of the spaces $X^{\alpha,\infty}$ (resp.~$X^{\alpha,2}$) for $\alpha\in\{90,45\}$ has a neighbourhood homeomorphic to some open ball in some of the 5 spaces considered in Lemma \ref{l:5sp}, the spaces $X^{\alpha,\infty}$ (resp.~$X^{\alpha,2}$) are locally injective (resp.~locally CAT(0)), thus by 
\cite[Theorem 1.2]{Miesch17} (resp.~by the Cartan--Hadamard Theorem for CAT(0) spaces, \cite[Theorem II.4.1(2)]{BrHae99})
they are injective (resp.~CAT(0));
let $\sigma^{\alpha,\infty}$ (resp.~$\sigma^{\alpha,2}$) be their convex bicombings. 
The classical result \cite{CrKl00} gives that $\partial_{\sigma^{90,2}} X^{90,2}$ and $\partial_{\sigma^{45,2}} X^{45,2}$ are not homeomorphic.
In order to finish the proof, 
we 
show that for each 
$\alpha\in\{90,45\}$
we have
that $\partial_{\sigma^{\alpha,\infty}} X^{\alpha,\infty}\cong\partial_{\sigma^{\alpha,2}} X^{\alpha,2}$.
In view of 
Proposition \ref{p:bilipbic},
it suffices to show that $\im\,\sigma^{\alpha,\infty}_{ab}=\im\,\sigma^{\alpha,2}_{ab}$
for all $a,b\in X^\alpha:=X^{\alpha,\infty}(=X^{\alpha,2})$.
First, we check it locally.
Observe that for each $x\in X^\alpha$ the space $X^\alpha$ may be identified locally around point $x$ with some (subset of a) space $M$ considered in Lemma \ref{l:5sp} in such a way that 
for any $p\in\{2,\infty\}$
there exists $r_p>0$
such that
this identification 
gives 
an
isometric identification of 
the
ball $\overline{B}(x,r_p)$ 
in $X^{\alpha,p}$
with the ball 
$\overline{B}(x,r_p)$ in $M^p$. 
By conicality, for any $p\in\{2,\infty\}$ the ball $\overline{B}(x,r_p)$ is convex with respect to 
the 
\ccc{}
bicombings on $X^{\alpha,p}$ and $M^p$, therefore the 
\ccc{}
bicombings from these spaces restrict to 
\ccc{}
bicombings on $\overline{B}(x,r_p)$; they must restrict to the same bicombing, as the balls in injective and CAT(0) spaces are injective and CAT(0), respectively, which implies that the 
\ccc{}
bicombing on 
the ball $\overline{B}(x,r_p)$ 
in $X^{\alpha,p}$
is unique. 
Therefore, 
Lemma \ref{l:5sp} implies 
that 
for 
any 
$a,b\in B_{X^{\alpha,\infty}}(x,r_\infty)\cap B_{X^{\alpha,2}}(x,r_2)$ 
we have that $\im\,\sigma^{\alpha,2}_{ab}=\im\,\sigma^{\alpha,\infty}_{ab}$.
In the general case, by the above, the image of
a
(global) 
\linebreak[1]
$\sigma^{\alpha,\infty}$-geodesic $\sigma^{\alpha,\infty}_{ab}$ is the image of some local CAT(0)-geodesic, which is also the 
unique global CAT(0)-geodesic $\sigma^{\alpha,2}_{ab}$ by \cite[Proposition II.1.1.4(1,2)]{BrHae99} (or Proposition \ref{f:loctoglobgeod}).
\end{proof}

\begin{remark}\phantomsection\label{r:nonunique}
\begin{enumerate}[(i)]
\item\label{r:nonunique1} 
One may construct complexes as on Figure \ref{fig:glue} parametrised by the $\ell^2$-angle $\alpha\in(0,\pi/2]$ between the \textbf{black} and \textcolor{red}{red} segment (see \cite[Section 1.3]{CrKl00}). Wilson \cite{Wilson05} improved the result of Croke and Kleiner by showing that for each pair of different angles $\alpha,\beta\in(0,\pi/2]$ the resulting boundaries (of the universal cover with the locally
$\ell^2$ path metric) are 
not homeomorphic,
thus the 
fundamental
group of (all) these complexes admits $2^{\aleph_0}$ pairwise non-homeomorphic CAT(0)-boundaries. However, the only pair of lines in a pair of $\ell^\infty$-planes, along which they can be glued to obtain an injective metric space, 
are two diagonal ones, or two 
horizontal-or-vertical 
ones 
---
one may use the argument in \cite[Section 5]{Miesch15} almost verbatim. 
Since injective metric spaces are locally injective, this prevents us 
from using
Wilson's approach directly. 

\item\label{r:nonunique2} 
Since the construction of the 
spaces $X^{45,\infty}$ and $X^{90,\infty}$ consist in gluing injective planes along horizontal, vertical and diagonal lines, one may obtain graphs $\Gamma^\alpha$ for $\alpha\in\{45,90\}$ by replacing each injective plane in the construction of $X^{\alpha,\infty}$ with a graph on $\Z^2$ with vertices $(x_1,y_1),(x_2,y_2)\in\Z^2$ connected by an edge iff $|x_1-y_1|,|x_2-y_2|\leq 1$ 
(by the graph $\Gamma^\alpha$ we mean the metric space with its underlying set consisting of just the vertices, with the edges used only to induce the path metric making the ends of each edge at distance 1 from each other) 
to prove an analogue of Theorem \ref{t:nonunique} for Helly groups. However, the graph $\Gamma^{45}$ is not Helly. To see this, consider a pair of planes $P_i=\{(x,y)^{P_i}:x,y\in\R\}$ for $i=1,2$ that are glued along their diagonals $\{(d,d)^{P_i}:d\in\R\}$ in the construction of $X^{45,\infty}$ (so that $(d,d)^{P_1}=(d,d)^{P_2}$ for all $d\in\R$). Consider the unit balls in $\Gamma^{45}$ centred at points 
$a_1=(-1,0)^{P_1}$, $a_2=(1,2)^{P_1}$ and $a_3=(1,0)^{P_2}$.
They clearly intersect pairwise, however their intersection is empty. To see this, observe that the diagonal $D=\{(d,d)^{P_1}:d\in\R\}$ disconnects $\Gamma^{45}$ so that $a_1$ and $a_3$ are in different connected components of $\Gamma^{45}\setminus D$. Therefore $\overline{B}_{\Gamma^{45}}(a_1,1)\cap \overline{B}_{\Gamma^{45}}(a_3,1)\subseteq D$, consequently $\overline{B}_{\Gamma^{45}}(a_1,1)\cap \overline{B}_{\Gamma^{45}}(a_3,1)=\{(0,0)^{P_1}\}$, but $(0,0)^{P_1}\not\in \overline{B}_{\Gamma^{45}}(a_2,1)$.
\end{enumerate}
\end{remark}

\section{Products and CAT(0) cube complexes with injective metrics}\label{s:ccc}

Let $X$ be a 
cube complex. We denote by $d_p^X$ for $p\in[1,\infty]$ the 
gluing
metric on $X$ arising from endowing each cube
$C$ 
of $X$ 
with the $\ell^p$-metric 
from
the 
unit cube
$[0,1]^{\dim C}$.
\begin{remark}\label{r:lpmetric}
That $d_p^X$ is indeed a metric, not just a pseudometric, is a consequence 
(cf.~\cite[Corollary I.7.10]{BrHae99})
of the following observation:
for each point $x$ of a cube complex $X$
there exists $\epsilon_x>0$ such that
for each cube $C$ of $X$ that contains $x$
the 
\linebreak[2]
$d_p^{\:\!C}$-distance 
(note that we do not restrict here 
$d_p^X$ 
from $X$ to $C$)
from $x$ to the faces of $C$ that do not contain $x$ is 
at least
$\epsilon_x$.
Indeed,
let $C(x)$ be the cube 
of
$X$ that contains $x$ in its interior. If $x$ is not a vertex of $X$, 
let $\epsilon_x>0$ be the 
distance
in $(C(x),d_p^{\:\!C(x)})$ 
from $x$ to the faces of $C(x)$,
otherwise
put $\epsilon_x:=1$. 
Since 
every cube
$C$ 
containing $x$
is a product of $C(x)$ with another cube,  
the observation
follows. 
\end{remark}

In this section we build on the ideas introduced in Section \ref{s:nonunique} to prove the following theorem. 

\begin{theorem}\label{t:ccc} 
	Let $X$ be a 
	CAT(0) cube complex
	of dimension at most $2$,
	and 
	let $\sigma^p$ be the convex bicombing on $(X,d_p)$ 
	for $p\in[1,\infty]$.
	Then 
	all of the $\sigma^p$ for $p\in[1,\infty]$ 
	have the same trajectories.

\end{theorem}

\newcommand{\hdim}{X^{\text{\o}}}

\begin{remark}\phantomsection\label{r:ccc} 
	\begin{enumerate}[(i)]
		\item\label{r:ccc1} 
A CAT(0) cube complex $X$ endowed with the $\ell^p$-metric $d_p^X$ admits a unique convex geodesic bicombing for all $p\in[1,\infty]$, \cite[Theorem 5.18]{HaHoPe25}.
These bicombings are also consistent:  
for $p\in(1,\infty)$ the geodesics in $(X,d_p^X)$ are unique \cite[Theorem 5.17]{HaHoPe25}, and the convex bicombings for $p=1$ and for $p=\infty$ are limits of the convex bicombings for $p\in(1,\infty)$ as $p\to 1$ and $p\to\infty$, respectively, \cite[Theorem 5.18]{HaHoPe25}.
As in Section~\ref{s:nonunique}, 
further in this section we 
usually
use these uniqueness results without mentioning.

		\item\label{r:ccc2} 
		Injectivity of $(X,d_\infty)$ implies 
		CAT(0)-ness of $(X,d_2)$
		--- 
		see \cite[Theorem 1.2]{Miesch14},
		which uses the Link Condition, \cite[the 4.2.C that follows 4.2.D]{Gromov87} --- therefore the above theorem applies to the case of cube complexes that are injective when endowed with the piecewise-$\ell^\infty$ metric.

	\end{enumerate}
\end{remark}

In view of 
Proposition \ref{p:bilipbic}, 
we have the following corollary. 

\begin{corollary}[Theorem \ref{t:cccintro}]\label{c:ccc}
Let $X$ be a 
CAT(0) cube complex
of dimension at most $2$,
and 
let $\sigma^p$ be the convex bicombing on $(X,d_p)$ for 
$p\in[1,\infty]$. 
Then the identity of $X$ extends to a homeomorphism between $\bdry{X}{\sigma^p}$ and $\bdry{X}{\sigma^r}$ for any $p,r\in[1,\infty]$, in particular 
all of 
the boundaries $\partial_{\sigma^p}X$
for $p\in[1,\infty]$
are homeomorphic. 
\end{corollary}

We begin the preparations for the proof of Theorem \ref{t:ccc} with some general definitions and observations.

\newcommand{\prbi}[2]{#1 \otimes #2}

Let $(X,d_X)$ 
and 
$(Y,d_Y)$ be metric spaces. By the \emph{$\ell^p$-product $X\times_pY$} we mean the Cartesian product $X\times Y$ endowed with the metric 
\begin{equation*}
d_{X\times_pY}((x_1,y_1),(x_2,y_2))=\bnorm{p}{\bpl d_X(x_1,x_2),d_Y(y_1,y_2) \bpr}.
\end{equation*}
If the spaces $X,Y$ admit bicombings $\sigma^X\!,\sigma^Y\!$, respectively, then the \emph{product bicombing 
$\prbi{\sigma^X \!}{\sigma^Y}$
}
is defined 
to be
the map
\begin{equation*}
(X\times Y)\times(X\times Y)\times[0,1]\ni	
\bpl
(x_1,y_1),(x_2,y_2)
,
t\bpr
\:\longmapsto\,\bpl\sigma^X_{x_1,x_2}(t),\sigma^Y_{y_1,y_2}(t)\bpr
\in X\times Y.
\end{equation*}

\begin{remark}\label{r:prodbicomb}
		Let 
		$(X,d_X)$ and $(Y,d_Y)$ 
		be metric spaces that
		admit bicombings $\sigma^X$ and $\sigma^Y$, respectively. Then, 
the product bicombing 
$\prbi{\sigma^X \!}{\sigma^Y}$
is consistent, conical, convex, 
and reversible
iff
both
of
the 
bicombings
$\sigma^X$ and $\sigma^Y$ are consistent, conical, convex, 
and reversible, respectively.
	The $\Longrightarrow$ implication follows as the spaces $X$ and $Y$ embed into $X\times_p Y$ as 
	sections, to which the product bicombing 
	$\prbi{\sigma^X \!}{\sigma^Y}$
	restricts 
	as
	$\sigma^X$ and $\sigma^Y$, respectively.
Regarding the $\Longleftarrow$ implication,
the claim 
regarding
consistency and reversibility follows directly;
the claim regarding conicality and convexity follows from a direct calculation involving the Minkowski's inequality.
\end{remark}

\newcommand{\gl}{\mathrm{gl}}

\begin{lemma}\label{l:produkt}
Let
$\{(X_i,d_i):i\in I\}$ 
be 
a family of metric spaces that are glued together to form 
a
(metric)
space $(X,d_X)$, where $d_X$ is the 
gluing metric, 
and 
denote 
the gluing map $\pi\colon\bigsqcup_{i\in I}X_i\to X$.
Let
$Y$ be a geodesic metric space
and
$p\in[1,\infty]$. 
Then gluing and taking the $\ell^p$-product commute, namely: the 
gluing
metric $d_{\gl}$ on $X\times Y$ arising from the
gluing
map $\pi\times\mathrm{id}_Y\colon\bigsqcup_{i\in I}(X_i\times Y,d_{X_i\times_pY})
\to
X\times Y$ is the same as the $\ell^p$-metric $d_{X\times_pY}$ on $X\times Y$.

\end{lemma}

\begin{proof}
	Let $x,x'\in X$ and $y,y'\in Y$.
	Consider an
	$((x,y),(x',y'))$--gluing path 
	$P=((x,y)=(x_0,y_0),\ldots,(x_n,y_n)=(x',y'))$ in $X\times Y$
	and $i_j\in I$ 
	for $j=1,\ldots,n$,
	such that $x_{j-1},x_j\in X_{i_j}$.
	Then 
	$P$
	induces 
	an 
	$(x,x')$--gluing path 
	$P_X=(x_0,\ldots,x_n)$ in $X$.
	The length $\mathrm{len}(P)$ satisfies
	\begin{multline*}
		\mathrm{len}(P)=\sum_{j=1}^n\bnorm{p}{\bpl d_{X_{i_j}}(x_{j-1},x_j),d_Y(y_{j-1},y_j)\bpr}\\[-0.6em]
		\geq\bnorm{p}{\bpl\!\!\;\sum_{j=1}^nd_{X_{i_j}}(x_{j-1},x_j),\sum_{j=1}^nd_Y(y_{j-1},y_j)\bpr}\,\:\\[-0.3em]
		=\bnorm{p}{\bpl\mathrm{len}(P_X),\sum_{j=1}^nd_Y(y_{j-1},y_j))\bpr}
		\geq\bnorm{p}{\bpl d_X(x_0,x_n),d_Y(y_0,y_n)\bpr},
	\end{multline*} 
	therefore $d_{\gl}((x,y),(x',y'))\geq d_{X\times_p Y}((x,y),(x',y'))$.
	
	For the other inequality, let $P_X=(x=x_0,\ldots,x_n=x')$ be 
	an 
	$(x,x')$--gluing
	path
	of non-zero length
	in $X$ such that for every $1\leq j\leq n$ 
	the elements
	$x_{j-1},x_j$
	belong to 
	$X_{i_j}$ for some $i_j\in I$.
	Given any $y,y'\in Y$, choose a 
	sequence of points
	$P_Y=
	(y=y_0,y_1,\ldots,y_n=y')$ in $Y$ such that 
	$d_Y(y_{j-1},y_j)=d_{X_{i_j}}(x_{j-1},x_j)\cdot d_Y(y,y')/\mathrm{len}(P_X)$ 
	for 
	$1\leq j\leq n$. 
	Then $P=((x_0,y_0),\ldots,(x_n,y_n))$ 
	is 
	an 
	$((x,y),(x',y'))$--gluing
	path in $X\times Y$ 
	that satisfies $\mathrm{len}(P)=\norm{p}{(\mathrm{len}(P_X),d_Y(y,y'))}$, 
	and the other inequality follows.  
\end{proof}

Below
we 
present and discuss
several definitions concerning local fragments of a cube complex.

Let $X$ be a cube complex.
For any cube $C$ of $X$ 
we
denote by $X_C$ the 
smallest 
subcomplex of $X$ 
containing 
all 
of
the 
cubes of $X$ that contain 
$C$.
In further text,  
we 
usually
consider $X_C$ with the 
metrics 
$d_p^{X_C}$
(discarding 
the
cubes of $X$ not belonging to $X_C$ 
---
in particular 
we do not restrict the metric from $X$ to $X_C$).
We define $C^\perp$ to be 
any of the 
(identical) 
subcomplexes of 
$X_C$ 
such that 
$X_C$
is the product $C\times C^\perp$.

\begin{remark}\label{r:ltwoloc}
If $C$ is a cube of a CAT(0) cube complex $(X,d_2^X)$, then 
the above-defined cube complex
$(X_C,d_2^{X_C})$ 
is also CAT(0): since $X_C$ is contractible, it 
is sufficient 
to check that it satisfies the Link Condition.
Let $c$ be any vertex of $C$.
For each vertex $v$ of $X_C$ there exists 
the
smallest cube $C_v$ containing $C$ and $v$;
denote by $\Delta_v$ the set of neighbours of $c$ (in the 1-skeleton of $X$) that belong to $C_v$.
Consider a vertex $a$ of $X_C$ and neighbours $v_1,\ldots,v_n$ of $a$ in $X_C$ such that for 
all
$1\leq i<j\leq n$ the vertices $v_i,v_j$ span an edge in the link of $a$ in $X_C$, i.e.~there exists a cube $C_{ij}$ in $X$ containing $a,v_i,v_j$ and $C$.
Then the cube $C_{ij}$ must contain $C_a$, $C_{v_i}$ and $C_{v_j}$, hence each pair of vertices from $\Delta_a\cup\Delta_{v_i}\cup\Delta_{v_j}$ is connected by an edge in the link of $c$.
Therefore, 
every
pair of vertices of the set 
$\Delta_a\cup\bigcup\{\Delta_{v_i}:1\leq i\leq n\}$ spans an edge in the link of $c$, 
and
by the 
Link Condition 
for
$c$ in $X$, there exists a cube in $X$ 
that contains $a$, $C$ and all of the $v_i$, as required.
\end{remark}

Given a point $x\in X$, 
let
$C(x)$ be the cube in $X$ 
that contains $x$ in its interior
(we use 
the
convention that the interior of a 0-cube is the 0-cube itself).
We then 
use the shorter notation $X_x$ for $X_{C(x)}$
and
say that 
the point $x$
is of \emph{type} 
$(\dim X_x-\dim C(x),\dim C(x))$.
(Note that the first coordinate of type is equal to $\dim C(x)^\perp$.
The type 
with the 
lexicographic
order may be 
considered 
to be
a measure of complexity 
of a local neighbourhood of
$x$
in the 
cube
complex $X$.)

\newcommand{\traj}{%
\ifmmode%
\mathsf{Traj}%
\else%
$\mathsf{Traj}$%
\fi
}
\newcommand{\ged}{%
\ifmmode%
\mathsf{Par}%
\else%
$\mathsf{Par}$%
\fi
}

We say that 
a locally finite CAT(0)
cube complex 
$X$ has the \emph{property \traj} 
if all of the convex $d_p^X$-bicombings $\sigma^p$ on $X$ for $p\in[1,\infty]$
have the same trajectories, i.e.~$\im\,\sigma_{xx'}^p=\im\,\sigma_{xx'}^r$ for all $x,x'\in X$ and $p,r\in[1,\infty]$;
and say that $X$ has the \emph{property \ged} if furthermore the parametrisations agree, i.e.~the bicombings $\sigma^p$ 
for $p\in[1,\infty]$
are 
all
equal.

\medskip
The main technical lemma used in the proof of Theorem \ref{t:ccc} to carry an induction over the type is as follows.

\begin{lemma}\label{l:ccctech}
Let $X$ be a locally finite CAT(0) cube complex
and $x\in X$.
\begin{enumerate}[(i)]
\item\label{l:ccctech12}
If $x$ is of type $(0,0)$, $(0,1)$ or $(1,0)$,
i.e.~$\dim X_x\leq 1$, 
then 
$X_x$
satisfies
\ged.

\item\label{l:ccctech3}
If $C(x)^\perp$ 
satisfies 
\ged, then 
$X_x$
also 
satisfies
\ged.

\item\label{l:ccctech4}
If $x$ 
is a vertex,
and for 
all 
$x_\circ\neq x$ belonging to the interior of 
a cube containing $x$
the cube complex 
$X_{x_\circ}$
satisfies
\traj, then 
$X_x$
also 
satisfies
\traj.

\end{enumerate}
\end{lemma}

\newcommand{\spcint}{\mathsf{I}}

\begin{proof}
\ref{l:ccctech12}
{\sc Type $(0,0)$.} In this case 
$X_x$
is a single point.
The claim follows.

\smallskip
{\sc Type $(0,1)$.} In this case 
$X_x$
is a unit interval $I$, on which 
all of the $\ell^p$-metrics $d_p^I$ for $p\in[1,\infty]$ coincide,
therefore 
all of the $d_p^I$-bicombings
on $I$
are equal.

\smallskip
{\sc Type $(1,0)$.}  
In this case the complex 
$X_x$
consists of several copies of the unit interval $I$ glued in 
the
common vertex $x$. 
Since 
all of the 
metrics
$d_p^I$ are equal,
also 
all of the 
metrics
$d_p^{X_x}$ are equal,
therefore 
all of
the 
convex 
$d_p^{X_x}$-bicombings 
on $X_x$ are equal.

\medskip
\ref{l:ccctech3}
By Lemma \ref{l:produkt}
the complex $(X_x,d_p^{X_x})$ 
is 
the $\ell^p$-product of $(C(x)^\perp,d_p^{\,C(x)^\perp})$ 
with 
the $\dim C(x)$--fold $\ell^p$-product of the unit interval $(I,d_p^I)$. 
By the assumption and \ref{l:ccctech12}, the complex $C(x)^\perp$ and the 
interval $I$, respectively, satisfy \ged. 
Denote
by $\sigma^{p,\perp}$ the
convex 
$d_p^{\,C(x)^\perp}$-bicombing on $C(x)^\perp$, 
and by $\sigma^I$ the bicombing on the unit interval. 
We 
then
have the equality $\prbi{\sigma^{p,\perp}}{(\sigma^I)^{\prbi{}{\dim C(x)}}}=\prbi{\sigma^{r,\perp}}{(\sigma^I)^{\prbi{}{\dim C(x)}}}$
for all $p,r\in[1,\infty]$.
The bicombing $\prbi{\sigma^{p,\perp}}{(\sigma^I)^{\prbi{}{\dim C(x)}}}$
is the convex $d_p^{X_x}$-bicombing on
$X_x$, 
see Remark \ref{r:prodbicomb}, which finishes the proof. 

\medskip
\ref{l:ccctech4}
Denote
by $\sigma^p$ the 
convex
$d_p^{X_x}$-bicombing on $X_x$.
Observe that for
any point $x_\circ\neq x$ 
belonging
to
the interior 
of 
a 
cube
incident 
to
$x$, 
the complex
$X_{x_\circ}$
is 
the 
subcomplex 
$(X_x)_{x_\circ}$ of 
$X_x$.
Consider $x'\in X_x\setminus\{x\}$ and let $x_\circ'$ be any point from the interior of the smallest cube containing both $x$ and $x'$.
Since 
some 
open neighbourhood
of 
$x'$
in $X_x$
is contained in $X_{x_\circ'}$,
and
the inclusion
$(X_{x_\circ'},d_r{\hspace{-0.35em}\raisebox{0.45em}{$\scriptstyle X\!\raisebox{-0.05em}{$\scriptscriptstyle x_\circ^{\scaleobj{0.7}{\prime}}$}$}})\hookrightarrow(X_x,d_r^{X_x})$
is an 
isometry
on 
a neighbourhood of $x'$ 
(recall Remark \ref{r:lpmetric}),
and
for all metrics $d$ the $d$-balls are convex with respect to conical $d$-bicombings
(here we consider $d=d_r^{X_x}$ and $d=d_r{\hspace{-0.35em}\raisebox{0.45em}{$\scriptstyle X\!\raisebox{-0.05em}{$\scriptscriptstyle x_\circ^{\scaleobj{0.7}{\prime}}$}$}}$), 
by uniqueness of bicombings (see Remark \ref{r:ccc}\ref{r:ccc1})
for all $p\in[1,\infty]$
there exists 
an open
neighbourhood 
$U^p_{x'}$
of 
$x'$
in $X_x$
such that 
$\im\,\sigma^2_{ab}=\im\,\sigma^p_{ab}$ 
for all 
$a,b\in U_{x'}$.

\smallskip
Now consider any $a,b\in X_x$. We shall prove that
$\sigma_{ab}^2=\sigma_{ab}^p$ for all $p\in[1,\infty]$.
The case of $a=x=b$ is clear.
Next, we consider the case when
$a\neq x=b$. 
Let $a_\circ$ belong to the interior of the smallest cube that contains 
both
$a$ and $x$.
Let $\gamma$ be the (straight-line) geodesic from the convex $d_2^{\:\!C(a_\circ)}$-bicombing on $C(a_\circ)$ that begins in $a$ and ends in $x$.
Since the inclusion 
$(C(a_\circ),d_2^{\:\!C(a_\circ)})\hookrightarrow(X_x,d_2^{X_x})$ 
is a local isometry 
on 
the interior of the cube $C(a_\circ)$
(recall Remark \ref{r:lpmetric}), 
the restrictions $\gamma|_{[s,t]}$ are 
local 
$\sigma^2$-geodesics 
for $0<s\leq t<1$
(recall that CAT(0)-geodesics are unique).
Therefore, 
by considering the neighbourhoods $U^p_{x'}$,
one obtains that for
all
$0<s\leq t<1$ the image $\im\,\gamma|_{[s,t]}$ is the image of a local $\sigma^p$-geodesic;
then, by Proposition \ref{f:loctoglobgeod}, 
we have that these local $\sigma^p$-geodesics are (global) $\sigma^p$-geodesics, hence $\im\,\sigma^2_{\gamma(s),\gamma(t)}=\im\,\gamma|_{[s,t]}=\im\,\sigma^p_{\gamma(s),\gamma(t)}$; finally, 
passing to the limits $s\to 0$ and $t\to 1$,
we obtain that $\im\,\sigma^2_{ax}=\im\,\gamma=\im\,\sigma^p_{ax}$, as required.
Likewise follows the case 
when 
$a=x\neq b$.
In the case when $a\neq x\neq b$, we consider two subcases.
If
for some $r\in\{2,p\}$ 
there exists a local $\sigma^r$-geodesic 
$\gamma$ 
from $a$ to $b$ in 
$X_x$ 
that omits $x$,
then, by 
considering 
the neighbourhoods $U^p_{x'}$,
one obtains that 
the
image 
$\im\,\gamma$
is both 
the image of a local $\sigma^2$-geodesic and 
the image of a local $\sigma^p$-geodesic,
therefore
by
Proposition \ref{f:loctoglobgeod} 
$\im\,\sigma_{ab}^2=\im\,\gamma=\im\,\sigma_{ab}^p$.
Otherwise, 
for both choices of $r\in\{2,p\}$
the $\sigma^r$-geodesic 
$\sigma^r_{ab}$
passes through $x$ and must be the concatenation of the geodesics $\sigma^r_{ax}$ and $\sigma^r_{xb}$; the image of each of these geodesics does not depend on $r$,
as it was discussed in the case 
when 
$a\neq x=b$.
\end{proof}

\newcommand{\spcainn}
{
\begin{tikzpicture}
\begin{scope}[scale=0.15]
\draw (0,0) -- (2,0) -- (3.41,1.41) -- (4.82,0) -- (3.41,-1.41) -- (2,0);
\end{scope}
\end{tikzpicture}
}
\newcommand{\spcsqinn}{
\begin{tikzpicture}
\begin{scope}[scale=0.085]
\draw (0,0) -- (4,0) -- (4,4) -- (0,4) -- cycle;
\end{scope}
\end{tikzpicture}
}
\newcommand{\spcrys}{\!\!\spcainn\!}
\newcommand{\spc}{F}
\newcommand{\spcsq}{\!\spcsqinn\!}

\begin{remark}\label{r:ccctech}
Denote by 
$\spc$
the space consisting of a 1-cube 
\,$\spcint$\,
glued with a 2-cube 
\spcsq{}
in a vertex: 
\spcrys. Observe that, even though 
\,$\spcint$\,
and 
\spcsq{}
satisfy \ged, the 
lengths
of
the same 
bicombing-geodesic 
in 
$(\text{\spcsq},d_p)$
almost always differ
for $p\in[1,\infty]$,
which leads to the following.
\begin{enumerate}[(i)]

\item\label{r:ccctech1} By considering the space 
$\spc$,
one 
obtains
that
the ($\ged\!\!\;\Rightarrow\!\!\;\ged$)-version of Lemma~\ref{l:ccctech}\ref{l:ccctech4} does not hold. 

\item\label{r:ccctech2} 
The product of 
$\spc$
with the unit interval
does not satisfy the property $\traj$,
hence does not satisfy the conclusion of Theorem \ref{t:ccc}
and shows that
the $(\traj\Rightarrow\traj)$-version of Lemma \ref{l:ccctech}\ref{l:ccctech3} does not hold.

\end{enumerate}
	
We also note that the geodesics may `change dimension' 
like in the spaces $F$ and $F\times[0,1]$ considered above
even 
when
the complex is `of constant dimension'.
Namely, consider the space $F_5$ consisting of 
$5$ squares $\square_i$ for $0\leq i<5$ 
glued
in such a way that they share a single vertex, 
each
square 
$\square_i$
shares an edge with 
square
$\square_{i+1}$
(modulo $5$), and each edge belongs to at most $2$ of the squares $\square_i$.
This space contains the space $F$ as a $\sigma_p$-convex subcomplex (as the space consisting of a square and an opposite edge): it is well known that this holds in the $p=2$, CAT(0), case, and then one may use Theorem \ref{t:main} (or Lemma \ref{l:ccctech}, whose application it is) to justify this fact for other $p$.
Then the properties discussed in \ref{r:ccctech1} and \ref{r:ccctech2} 
above
hold for $F_5$ and the product of $F_5$ with the unit interval, respectively (\ref{r:ccctech2}, recall Lemma \ref{l:produkt} and Remark \ref{r:prodbicomb}).

\end{remark}

Now we are ready to prove the main theorem of this section.

\begin{proof} ({\sc of Theorem \ref{t:ccc}})
Denote by $\traj(n,m)$ the property that for every 
CAT(0)
cube complex $Y$ and every point $y\in Y$ of type $(n,m)$ the complex $Y_y$ satisfies \traj; likewise for \ged.
Lemma \ref{l:ccctech}\ref{l:ccctech12} 
states that
$\ged(0,0)$, 
$\ged(0,1)$
and $\ged(1,0)$
hold. 
Next, observe that 
for any element $y$ of a 
cube complex $Y$, upon 
denoting 
by $y^\perp$ the only element of the intersection $C(y)\cap C(y)^\perp$
(which is a single vertex), 
the equality $C(y)^\perp=(C(y)^\perp)_{y^\perp}$ holds 
(any cube $C^{\downarrow}$ in $C(y)^\perp$ is the intersection 
$C\cap C(y)^\perp$ for some cube $C$ of $Y$ containing $y$; 
then
the cube $C$ 
also
contains 
$C(y)$, 
to which 
$y^\perp$ belongs; hence $y^\perp$ belongs to $C\cap C(y)^\perp=C^\downarrow$);
and,
upon denoting by $(n,m)$ the type of $y$ in $Y$, 
the type of $y^\perp$ 
in the complex $C(y)^\perp$ is 
$(\dim (C(y)^\perp)_y-0,0)=(\dim C(y)^\perp,0)=(\dim Y_y-\dim C(y),0)=(n,0)$.
Therefore,
by Lemma \ref{l:ccctech}\ref{l:ccctech3}, 
$\ged(0,2)$ and $\ged(1,1)$ follow
from $\ged(0,0)$ and $\ged(1,0)$, respectively. 
Finally, 
it follows from
Lemma \ref{l:ccctech}\ref{l:ccctech4}
that 
\traj$(2,0)$ holds,
as for a 2-dimensional 
cube
complex $Y$ and $y\in Y$, 
for all $y_\circ\neq y$ belonging to the interior of 
a
cube 
$C(y_\circ)$
in $Y$ containing $y$
one has that 
the type $(n,m)$ of $y_\circ\in Y$ satisfies $n+m=\dim Y_{y_\circ}\leq\dim Y=2$ 
and
$m=\dim C(y_\circ)>0$. 

\smallskip
Then,
as 
the type $(n,m)$ of each element of $X$ satisfies $n+m\leq\dim X\leq 2$,
and for 
$p\in[1,\infty]$
the
$d_p$-balls are $\sigma^p$-convex, 
and 
for $x\in X$ 
the inclusion 
$(X_x,d_p^{X_x})\hookrightarrow(X,d_p^X)$ 
is 
a 
local isometry
on a neighbourhood of
$x$
(recall Remark \ref{r:lpmetric}),
we have a family of open subsets $\{U^p_x\subseteq X:x\in X\}$ such that $\im\,\sigma^2_{ab}=\im\,\sigma^p_{ab}$ for all $x\in X$ and $a,b\in U_x^p$. 
Therefore by Proposition \ref{f:loctoglobgeod} the bicombings $\sigma^2$ and $\sigma^p$ have the same trajectories, and the theorem follows.
\end{proof}

\newcommand{\chain}{N}
\newcommand{\chainprod}{P}
\newcommand{\chainline}{L}
\newcommand{\rline}{R}
\newcommand{\chainhalfplane}{H}
\newcommand{\chainlinerysinn}
{
	\begin{tikzpicture}
		\begin{scope}[scale=0.15]
			\newcommand{\dia}{1.4142}
			\newcommand{\rad}{3pt}
			\newcommand{\krop}{1}
			\filldraw (0,0) circle (\rad);
			\filldraw (2,0) circle (\rad);
			\filldraw (4,0) circle (\rad);
			\filldraw (4+2*\dia,0) circle (\rad);
			\filldraw (4+4*\dia,0) circle (\rad);
			\filldraw (6+4*\dia,0) circle (\rad);
			\filldraw (6+6*\dia,0) circle (\rad);
			\filldraw (6+8*\dia,0) circle (\rad);
			\filldraw (6+8*\dia+\krop,0) circle (0.5\rad);
			\filldraw (6+8*\dia+1.5*\krop,0) circle (0.5\rad);
			\filldraw (6+8*\dia+2*\krop,0) circle (0.5\rad);
			\draw (0,0) -- (2,0) -- (4,0) -- (4+\dia,\dia) -- (4+2*\dia,0) -- (4+3*\dia,\dia) -- (4+4*\dia,0) -- (6+4*\dia,0) -- (6+5*\dia,\dia) -- (6+6*\dia,0) -- (6+7*\dia,\dia) -- (6+8*\dia,0) -- (6+7*\dia,-\dia) -- (6+6*\dia,0) -- (6+5*\dia,-\dia) -- (6+4*\dia,0) -- (4+4*\dia,0) -- (4+3*\dia,-\dia) -- (4+2*\dia,0) -- (4+\dia,-\dia) -- (4,0);
		\end{scope}
	\end{tikzpicture}
}
\newcommand{\chainlinerys}{\!\chainlinerysinn\!}

\begin{remark}\phantomsection\label{r:cccdisc}
\begin{enumerate}[(i)]
\item\label{r:cccdisc1} 
The 
product 
of the
space 
$\spc$
from Remark \ref{r:ccctech}
with the unit interval
is an example of a finite CAT(0) cube complex of dimension $3$ for which the conclusion of Theorem \ref{t:ccc} does not hold.
Furthermore, its underlying idea can be used in the following construction of a locally finite CAT(0) cube complex $\chainprod$ of dimension $3$
such that 
the identity map $(\chainprod,d_2)\to(\chainprod,d_\infty)$ does not extend continuously to 
a map between compactifications
(which is the first conclusion of the Corollary \ref{c:ccc}).
Let $x_1,x_2,x_3,\ldots$ be such that $x_1=1$, $x_n\in\{1,\sqrt{2}\}$,
and the sequence $(\sum_{i=1}^nx_i)/n$ is not convergent.
Consider cubes $C_1,C_2,C_3,\ldots$ such that 
$C_n$ is a 1-cube iff $x_n=1$, and $C_n$ is a 2-cube iff $x_n=\sqrt{2}$. 
Let $\chain$ be the cube complex 
similar to 
\chainlinerys,\linebreak[1]
obtained from the sequence of cubes $(C_n)$
---
more precisely, we denote by $\chain$ 
the result of
gluing the cubes $C_n$ in such a way that the cubes $C_i$ and $C_j$ are glued to each other iff $|i-j|=1$, 
every gluing is a gluing along a single vertex,
and if $C_n$ is a 2-cube, then the 
two
vertices participating in the gluing are on the opposite sides of a diagonal $D_n$ of $C_n$.
Denote by $\chainline$ the half-line in $\chain$ obtained by 
taking the union of all of the cubes $C_n$ that are 1-cubes with all of the diagonals $D_n$ 
(of 
these of the $C_n$
that are 2-cubes).
Let $\chainprod$ be the cube complex resulting from taking the product of $\chain$ with the $\R$-line $\rline$ subdivided into unit intervals. 

By applying Lemma \ref{l:fixpoint} to the symmetry along the half-plane $\chainhalfplane:=\chainline\times\rline$,
one obtains that that 
$\chainhalfplane$
is $\sigma^2$-convex and $\sigma^\infty$-convex 
in
$\chainprod$.
For $p\in\{2,\infty\}$, the 
half-line 
$(\chainline,d_p^{\chain}|_{\chainline\times\chainline})$ may be isometrically identified with the half-line 
$[0,\infty)$
(with the usual metric),
giving 
rise to
an isometric identification of the half-plane 
$(\chainhalfplane,d_p^\chainhalfplane)$ 
with the $\ell^p$-product 
$[0,\infty)\times_p\R$
(recall Lemma \ref{l:produkt}); below we use this identification by writing $(x,y)\in(\chainhalfplane,d^{\chainhalfplane}_p)$ for $x\geq 0$ and $y\in\R$.
Upon these identifications, the identity 
of $\chainhalfplane$ sends the point $(\sum_{i=1}^nx_i,y)\in(\chainhalfplane,d^{\chainhalfplane}_2)$ to $(n,y)\in(\chainhalfplane,d^{\chainhalfplane}_\infty)$, where $n\in\N$ and $y\in\R$.
Let $a\in(-1/\sqrt{2},1/\sqrt{2})\setminus\{0\}$.
The map $[0,\infty)\ni t\mapsto(t,at)\in(\chainhalfplane,d_2^H)$ is a $\sigma^2$-ray in $\chainhalfplane$
(thus, in $\chainprod$), 
and 
the points $(\sum_{i=1}^nx_i,a\cdot\sum_{i=1}^nx_i)\in(\chainhalfplane,d_2^{\chainhalfplane})$ for $n\in\N$ are mapped by 
the identity to
the points $(n,a\cdot\sum_{i=1}^nx_i)\in(\chainhalfplane,d_\infty^{\chainhalfplane})$;
the $\sigma^\infty$-geodesic from $(0,0)\in(\chainhalfplane,d^{\chainhalfplane}_\infty)$ to the point $(n,a\cdot\sum_{i=1}^nx_i)\in(\chainhalfplane,d^{\chainhalfplane}_\infty)$ passes through the point $(1,a\cdot(\sum_{i=1}^nx_i)/n)\in(\chainhalfplane,d^{\chainhalfplane}_\infty)$.
Since $|a|\in(0,\sqrt{2})$,
we have that 
$|a\cdot(\sum_{i=1}^nx_i)/n|\leq |a|\sqrt{2}\leq 1$,
so this point is at 
$d_\infty^{\chainhalfplane}$-distance $1$
from $(0,0)$.
By the assumption on the sequence $(x_n)$,
the sequence 
$(1,a\cdot(\sum_{i=1}^nx_i)/n)$ is not 
convergent, 
therefore $(n,a\cdot\sum_{i=1}^nx_i)\in(\chainhalfplane,d_\infty^{\chainhalfplane})$ is also not a convergent sequence in the compactification $\bdry{\chainprod}{\sigma^\infty}$, while it is convergent in the compactification $\bdry{\chainprod}{\sigma^2}$.

\smallskip
We do not know if the `in particular' part of Corollary \ref{c:ccc} may not hold for a locally finite CAT(0) cube complex.
(In the above example, 
the boundary $\partial_{\sigma^p}\chainprod$
is homeomorphic to an interval for any $p\in[1,\infty]$, as may be seen using Lemma \ref{l:produkt}, the fact that each ray in $\chain$ is contained in $\chainline$, and Proposition \ref{p:join} below.)

\item\label{r:cccdisc2} 
The local-to-global approach from the above proof of Theorem \ref{t:ccc} 
works not only for (CAT(0)) cube complexes of dimension at most $2$, 
however 
such extensions
would lead to
making the statement of Theorem \ref{t:ccc} 
more technical.
For instance, in the induction in the 
`local' 
part of this proof, 
instead of considering only the very general $\traj(n,m)$ and $\ged(n,m)$ properties
(note that we have, in fact, proved $\ged(0,m)$ 
and $\ged(1,m)$
for 
arbitrary $m$, and $\traj(2,0)$),
one may start with a particular cube complex $X$, consider which subcomplexes of $X$ 
are needed
to carry the induction, and prove \traj{} or \ged{} only for these --- this approach 
works e.g.~if the complex $X$ has no elements of type $(n,m)$ with $n\geq 2$ and $m\geq 1$.

\end{enumerate}	
\end{remark}

In the course of this section we have made enough preparations to state the following proposition.

\begin{fact}\label{p:join}
Let $p\in[1,\infty]$,
and 
assume that 
proper metric spaces $(X,d_X)$ and $(Y,d_Y)$ 
admit 
\ccc{}
bicombings $\sigma^X$ and $\sigma^Y$, respectively. Then the boundary 
${\partial_{\prbi{\sigma^X\!}{\sigma^Y}}X\times_p Y}$
is homeomorphic to the join 
$\partial_{\sigma^X}X*\partial_{\sigma^Y}Y$
of 
$\partial_{\sigma^X}X$ and $\partial_{\sigma^Y}Y$.
\end{fact}

\begin{proof}
First, assume that the space $Y$ is compact; then $\partial_{\sigma^Y} Y=\emptyset$, and we need to show that the boundary $\partial_{\prbi{\sigma^X\!}{\sigma^Y}}X\times_p Y$ is homeomorphic to $\partial_{\sigma^X}X$.
Pick a basepoint $o=(o_X,o_Y)\in X\times Y$.
By the definition of the product bicombing, the set $X\times\{o_Y\}$ is $(\prbi{\sigma^X\!}{\sigma^Y})$-convex in $X\times Y$,
hence the map
sending a $\sigma^X$-ray $\xi^X$ based in $o_X$
to the $(\prbi{\sigma^X\!}{\sigma^Y})$-ray\; $t\mapsto(\xi^X(t),o_Y)$
induces a homeomorphic embedding
of
$\partial_{\sigma^X}X$
into 
$\partial_{\prbi{\sigma^X\!}{\sigma^Y}}X\times_p Y$.
We show that it is onto by showing that 
every
$(\prbi{\sigma^X\!}{\sigma^Y})$-ray $\gamma$ 
based in 
$o$
has its image contained in $X\times\{o_Y\}$.
Since $Y$ is compact,
there exists a constant 
$D>0$ 
such that $Y\subseteq B_Y(o_Y,D)$.
Then, for each $n\in\N$ there exist $a_n\in X\times \{o_Y\}$ such that $d_{X\times_pY}(a_n,\gamma(n))<D$.
Proposition \ref{f:neweq52} then implies that
for all $n\geq r>0$
we have that
\begin{equation*}
d_{X\times_pY}(X\times\{o_Y\},\gamma(r))\leq d_{X\times_pY}(\geod{o}{a_n}(r),\gamma(r))\leq 2D\cdot r/n,
\end{equation*}
which tends to $0$ as $n\to\infty$; 
therefore $\gamma(r)\in X\times\{o_Y\}$, and $\im\,\gamma\subseteq X\times\{o_Y\}$. 
Likewise follows the case when $X$ is compact.	
	
\medskip	
Now we 
proceed to the main case.
Assume that $X$ are $Y$ are non-compact. Then, as the spaces $X$ and $Y$ are proper, the compactifications $\bdry{X}{\sigma^X}$ and $\overline{Y}_{\!\!\sigma^Y}$ are compact, and the boundaries $\partial_{\sigma^X}X$ and $\partial_{\sigma^Y}Y$ are non-empty.
Let $\rGamma:=\{(a,b)\in\R^2:a,b\geq0,\norm{p}{(a,b)}=1\}$. 
Consider the map $q$ given by 
\begin{equation*}
	\partial_{\sigma^X}X\times\partial_{\sigma^Y}Y\times\rGamma\ni\bpl[\xi_X],[\xi_Y],(a,b)\bpr\mapsto{\big[}t\mapsto\bpl\xi_X(at),\xi_Y(bt)\bpr{\big]}\in\partial_{\prbi{\sigma^X\!}{\sigma^Y}}X\times_pY. 
\end{equation*}
One may check that $q$ is well-defined,
i.e.~the map $t\mapsto(\xi_X(at),\xi_Y(bt))$ is a
$(\prbi{\sigma^X\!}{\sigma^Y})$-ray
for all choices of a $\sigma^X$-ray $\xi^X$, a $\sigma^Y$-ray $\xi^Y$ and $(a,b)\in\rGamma$,
whose asymptotic class does 
not 
change when $\xi^X$ or $\xi^Y$ is replaced with 
any other
element of the asymptotic class of $\xi^X$ or $\xi^Y$, respectively.

Let 
$\pi$
be the 
quotient map 
of the equivalence relation $\sim_*$
on the set $\partial_{\sigma^X}X\times\partial_{\sigma^Y}Y\times\rGamma$, where $\sim_*$ 
is 
--- upon an identification of $\rGamma$ with $[0,1]$ ---
the 
equivalence relation
standardly used to define the join of the spaces $\partial_{\sigma^X}X$ and $\partial_{\sigma^Y}Y$, i.e.~$\sim_*$ is the smallest equivalence relation such that $(x,y,(1,0))\!\sim_*\!(x,y',(1,0))$ and $(x,y,(0,1))\!\sim_*\!(x',y,(0,1))$ for all $x,x'\in\partial_{\sigma^X}X$ and $y,y'\in\partial_{\sigma^Y}Y$.
Observe 
that for any pair $\xi,\zeta$ of $\sigma^X$-rays or of $\sigma^Y$-rays originating from 
a common 
basepoint, and numbers $a,a'>0$, the function $[0,\infty)\ni t\mapsto d(\xi(at),\zeta(a't))$ is zero iff $a=a'$ and $\xi=\zeta$, or $a=a'=0$.
This observation implies that
the map $q$ (setwise) factors
(in a unique way)
as $q=q_\pi\pi$,
with
the map 
$q_\pi\colon\partial_{\sigma^X}X*\partial_{\sigma^Y}Y\to\partial_{\prbi{\sigma^X\!}{\sigma^Y}}X\times_p Y$ 
being
one-to-one.
We shall show that $q_\pi$ is a homeomorphism:
since the set
$\partial_{\sigma^X}X\times\partial_{\sigma^Y}Y\times\rGamma$
is compact, and we have already shown that $q_\pi$ is one-to-one, 
it is sufficient to show that $q_\pi$ is a
continuous
surjection.

\smallskip
The
surjectivity of $q_\pi$ is equivalent to the surjectivity of $q$.
We check the latter below.
Take a 
$(\prbi{\sigma^X\!}{\sigma^Y})$-ray
$(\gamma^X,\gamma^Y)$ in $X\times_p Y$. 
By the definition of the product bicombing, we have that
$\gamma^X(\alpha t)=\gamma^X|_{[0,t]}(\alpha t)=\sigma_{\gamma^X(0)\gamma^X(t)}(\alpha)$ 
for all $t\geq0$ and $0\leq \alpha\leq1$;
therefore
for all 
$t>0$ and 
$0\leq s\leq s'\leq t$
we have that
$\im\,\gamma^X|_{[s,s']}=\im\,\sigma_{\gamma^X(s),\gamma^X(s')}$
and $d_X(\gamma^X(s),\gamma^X(s'))=(s'-s)a_t$
(where $a_t=d_X(\gamma^X(0),\gamma^X(t))/t$).
Dividing the second of these equalities by $(s'-s)$, we see
that $a_t$ is 
a constant 
independent of $t$
---
denote it by $a$;
\linebreak[1]
therefore $\gamma^X$ satisfies $\im\,\gamma^X|_{[s,s']}=\im\,\sigma_{\gamma^X(s),\gamma^X(s')}$
and $d(\gamma^X(s),\gamma^X(s'))=(s'-s)a$ for all $s'\geq s\geq 0$.
Hence:
if $a\neq0$, then $\xi^X\colon[0,\infty)\to X$ given by $\xi^X(t)=\gamma^X(a^{-1}t)$ is a $\sigma^X$-ray;
and if $a=0$,
then
$\gamma^X$ is constant, so
any $\sigma^X$-ray $\xi^X$ originating in $\gamma^X(0)$ 
satisfies
$\gamma^X(t)(=\gamma^X(0))=\xi^X(at)$
for all $t\geq0$.
Similarly, one obtains that there exists a $\sigma^Y$-ray $\xi^Y$ such that $\gamma^Y(t)=\xi^Y(bt)$ for all $t\geq0$, where $b=d_Y(\gamma^Y(0),\gamma^Y(1))$.  
The speed parameters
$a$, $b$
satisfy
\begin{multline*}
	1=d_{X\times_pY}{\big(}(\gamma^X(0),\gamma^Y(0)),(\gamma^X(1),\gamma^Y(1)){\big)}\\
	=\bnorm{p}{{\big(}d_X(\gamma^X(0),\gamma^X(1)),d_Y(\gamma^Y(0),\gamma^Y(1)){\big)}}
	=\norm{p}{(a,b)}.
\end{multline*}
The 
surjectivity of $q$
follows.

The
continuity of $q_\pi$ 
is equivalent to the continuity of $q$, which we check below at each point $(\bar{x},\bar{y},(a,b))\in\partial_{\sigma^X}X\times\partial_{\sigma^Y}Y\times\rGamma$.
Fix a basepoint $o=(o_X,o_Y)\in X\times Y$.
Consider any $R,\epsilon>0$. 
Then for any $\delta>0$, 
$\bar{x}'\in U_{o_X}(\bar{x},R,\delta)\cap\partial_{\sigma^X}X$, $\bar{y}'\in U_{o_Y}(\bar{y},R,\delta)\cap\partial_{\sigma^Y}Y$, and $(a',b')\in\rGamma$ with $|a'-a|,|b'-b|<\delta$, 
we have
\begin{multline*} 
	d_X(\geod{o_X}{\bar{x}}^X(aR),\geod{o_X}{\bar{x}'}^X(a'R))\leq d_X(\geod{o_X}{\bar{x}}^X(aR),\geod{o_X}{\bar{x}}^X(a'R))+d_X(\geod{o_X}{\bar{x}}^X(a'R),\geod{o_X}{\bar{x}'}^X(a'R))\\
	\leq\delta R+a'R\delta\leq \delta R+\delta R=2\delta R.
\end{multline*} 
Since a similar reasoning applies in $Y$, we jointly have 
\begin{equation*}
	d_{X\times_p Y}\bpl\geod{o}{q(\bar{x},\bar{y},(a,b))}^{X\times Y}(R),\geod{o}{q(\bar{x}',\bar{y}',(a',b'))}^{X\times Y}(R)\bpr\leq2^{1/p}\cdot2\delta R,
\end{equation*}
which is smaller than $\epsilon$ for 
sufficiently small $\delta$.
\end{proof}

\section{Quasisymmetric structure on boundary}\label{s:quasisym}
Let $(X,d)$ be a complete metric space that admits a 
\ccc{}
bicombing $\sigma$.
In this section we define a metric $d_{o,C}$ on the boundary  $\partial_{\sigma}X$ such that 
if a group $G$ acts on $X$ via isometries and $\sigma$ is $G$-equivariant, then
the induced action on $(\partial_\sigma X,d_{o,C})$ 
is via quasisymmetries
(see definitions below). Similar properties of a similar metric have been studied in the case of CAT(0) spaces in \cite[Section 3.1]{Moran16} and \cite[Proposition 9.6(1)]{OsSw15}, where, basically, only the 
\ccc-ness
of the CAT(0)-bicombing is used, and we reformulate and adapt the proofs to our setting.

For 
metric spaces $(X_1,d_1),(X_2,d_2)$,
a map $f\colon (X_1,d_1)\to(X_2,d_2)$ is a \emph{quasisymmetry} if $f$ is not constant and there exists a homeomorphism $\eta\colon [0,\infty)\to[0,\infty)$ such that for all $x,y,z\in X_1$ 
and
$t\geq0$, if $d_1(x,y)\leq td_1(x,z)$, 
then $d_2(f(x),f(y))\leq\eta(t)d_2(f(x),f(z))$. 
By \cite[Proposition 10.6]{Heinonen01}, 
quasisymmetries are closed under composition, 
each quasisymmetry is one-to-one, 
and the inverse (defined on its image) 
of a quasisymmetry
is also a quasisymmetry.

Let $C>0$ and $o\in X$. For $\bar{x}_1,\bar{x}_2\in\partial_\sigma X$ such that $\bar{x}_1\neq\bar{x}_2$, define 
$d_{o,C}(\bar{x}_1,\bar{x}_2)=t^{-1}$, where $t$ is the unique number such that $d(\geod{o}{\bar{x}_1}(t),\geod{o}{\bar{x}_2}(t))=C$ (recall 
Proposition 
\ref{f:asymptotic}\ref{f:asymptotic1}), and let $d_{o,C}(\bar{x},\bar{x})=0$ for all $\bar{x}\in\partial_\sigma X$. 

\begin{fact}\label{p:quasisym}
Let $(X,d)$ be a complete metric space that admits a 
\ccc{}
bicombing $\sigma$. Then the following hold for every $C,C'>0$ and $o,o'\in X$. 
\begin{enumerate}[(i)]
\item 
\label{p:quasisym1}
The map $d_{o,C}$ is a metric.

\item 
\label{p:quasisym2}
The topology induced by $d_{o,C}$ coincides with the topology on $\partial_\sigma X$ defined in Section \ref{s:deslan}.

\item 
\label{p:quasisym3}
The identity map $\mathrm{id}\colon (\partial_\sigma X,d_{o,C})\to(\partial_\sigma X,d_{o,C'})$ is a quasisymmetry.

\item 
\label{p:quasisym4}
The identity map $\mathrm{id}\colon (\partial_\sigma X,d_{o,C})\to(\partial_\sigma X,d_{o',C})$ is a quasisymmetry.

\item 
\label{p:quasisym5}
Let $G$ act on $X$ via isometries in such a way that $\sigma$ is $G$-equivariant. Then 
the extension of 
the
action of each element of $G$ to $\bdry{X}{\sigma}$ 
restricts to 
a quasisymmetry of $(\partial_\sigma X,d_{o,C})$. 
\end{enumerate}	

\end{fact}

\begin{proof}
\ref{p:quasisym1} Let $\bar{x}_1,\bar{x}_2,\bar{x}_3\in\partial X$. 
The function $d_{o,C}$ clearly satisfies $d_{o,C}(\bar{x}_1,\bar{x}_2)=0$ iff $\bar{x}_1=\bar{x}_2$, and is symmetric, so it remains to show that $d_{o,C}$ satisfies the triangle inequality.
Suppose
without loss of generality
that $d_{o,C}(\bar{x}_1,\bar{x}_3)\geq d_{o,C}(\bar{x}_1,\bar{x}_2),d_{o,C}(\bar{x}_2,\bar{x}_3)$, and let $t_{ij}:=d_{o,C}(\bar{x}_i,\bar{x}_j)^{-1}$.
We have that 
\begin{multline*}
C=d(\geod{o}{\bar{x}_1}(t_{13}),\geod{o}{\bar{x}_3}(t_{13}))
\leq d(\geod{o}{\bar{x}_1}(t_{13}),\geod{o}{\bar{x}_2}(t_{13}))
+d(\geod{o}{\bar{x}_2}(t_{13}),\geod{o}{\bar{x}_3}(t_{13}))\\
\leq (t_{13}/t_{12})d(\geod{o}{\bar{x}_1}(t_{12}),\geod{o}{\bar{x}_2}(t_{12}))
+(t_{13}/t_{23})d(\geod{o}{\bar{x}_2}(t_{23}),\geod{o}{\bar{x}_3}(t_{23}))
=Ct_{13}(t_{12}^{-1}+t_{23}^{-1}), 
\end{multline*}
where the last inequality follows from conicality of $\sigma$. Therefore $t_{12}^{-1}\leq t_{13}^{-1}+t_{23}^{-1}$, 
and the claim follows.

\medskip
\ref{p:quasisym2} Observe that by the definition, for any $\epsilon,r>0$ 
and $\bar{x}_1\in\partial X$ 
we have
that 
$B_{d_{o,C}}(\bar{x}_1,\epsilon)=U_o(\bar{x}_1,1/\epsilon,C)\cap\partial_\sigma X$. 
To finish the proof of the claim, 
consider 
$\bar{x}_1,\bar{x}_2\in\partial_\sigma X$ and $R,\epsilon>0$
such that
$\bar{x}_2\in U_o(\bar{x}_1,R,\epsilon)\cap\partial_\sigma X$. 
There exists $\epsilon'$ such that $C>\epsilon'>0$ and $U_o(\bar{x}_2,R,\epsilon')\subseteq U_o(\bar{x}_1,R,\epsilon)$. Then, by conicality, $U_o(\bar{x}_2,R,\epsilon')\cap\partial_\sigma X\supseteq U_o(\bar{x}_2,RC/\epsilon',C)\cap\partial_\sigma X=B_{d_{o,C}}(\bar{x}_2,\epsilon'/RC)$.

\medskip
\ref{p:quasisym3} Let $\mu=\min(C,C')$ and $M=\max(C,C')$.  
By 
conicality of $\sigma$, for any $\bar{x}_1,\bar{x}_2\in\partial X$ we have $d_{o,M}(\bar{x}_1,\bar{x}_2)\leq d_{o,\mu}(\bar{x}_1,\bar{x}_2)\leq (M/\mu)d_{o,M}(\bar{x}_1,\bar{x}_2)$.
The claim follows.    	
	
\medskip
\ref{p:quasisym4} Since quasisymmetries are closed under composition, in view of 
\ref{p:quasisym3},
we can assume that $C>2d(o,o')$. Let $\bar{x}_1,\bar{x}_2\in\partial X$ be different and $t:=d_{o,C}(\bar{x}_1,\bar{x}_2)^{-1}$.  
If $d(\geod{o'}{\bar{x}_1}(t),\geod{o'}{\bar{x}_2}(t))\geq C$, 
then $d_{o',C}(\bar{x}_1,\bar{x}_2)\geq t^{-1}=d_{o,C}(\bar{x}_1,\bar{x}_2)$. 
Otherwise, observe that by the triangle inequality and 
Proposition \ref{f:asymptotic}\ref{f:asymptotic1},
\begin{multline*}
d(\geod{o'}{\bar{x}_1}(t),\geod{o'}{\bar{x}_2}(t))
\geq d(\geod{o}{\bar{x}_1}(t),\geod{o}{\bar{x}_2}(t))\\
-d(\geod{o}{\bar{x}_1}(t),\geod{o'}{\bar{x}_1}(t))-d(\geod{o}{\bar{x}_2}(t),\geod{o'}{\bar{x}_2}(t))
\geq C-2d(o,o').
\end{multline*}
By conicality of $\sigma$, 
\vspace{0.3em}
\begin{equation*}
d\!\!\:\left(\!\geod{o'}{\bar{x}_1}\!\left(\frac{tC}{C-2d(o,o')}\right)\!,\geod{o'}{\bar{x}_2}\!\left(\frac{tC}{C-2d(o,o')}\right)\!\right)
\geq \frac{C}{C-2d(o,o')}d(\geod{o'}{\bar{x}_1}(t),\geod{o'}{\bar{x}_2}(t))\geq C,
\end{equation*} 

\smallskip\noindent
therefore $d_{o',C}(\bar{x}_1,\bar{x}_2)^{-1}\leq tC/(C-2d(o,o'))$.
 
Summarising
both cases, $d_{o,C}(\bar{x}_1,\bar{x}_2)\leq (C/(C-2d(o,o')))d_{o',C}(\bar{x}_1,\bar{x}_2)$. Therefore, as we can 
swap
$o$ with $o'$ in the above reasoning, the claim follows.

\medskip
\ref{p:quasisym5} Observe that 
the action of each element $g\in G$ on $X$
induces an
an isometry between $(\partial_\sigma X,d_{o,C})$ and $(\partial_\sigma X,d_{go,C})$, therefore the claim follows by 
\ref{p:quasisym4}
and the fact that quasisymmetries are closed under composition.
\end{proof}

\newcommand{\emp}[1]{\textcolor{magenta}{#1}}
\newcommand{\cola}[1]{\textcolor{red}{#1}}
\newcommand{\colb}[1]{\textcolor{orange}{#1}}
\newcommand{\colc}[1]{\textcolor{blue}{#1}}
\newcommand{\cold}[1]{\textcolor{violet}{#1}}

\section{Axes, flats, and the topology of boundary}\label{s:flats}

We begin this section 
with a brief recap of
some standard terminology.
For an isometry $\varphi$ of a metric space $(X,d)$, we define $|\varphi|:=\inf_{x\in X}d(x,\varphi(x))$, $\Min(\varphi):=\{x\in X:d(x,\varphi(x))=|\varphi|\}$, and we call an isometric embedding $\gamma\colon \R\to X$ an \emph{axis} of $\varphi$ if there exists a number $T>0$ such that $\varphi(\gamma(t))=\gamma(t+T)$ for all $t\in\R$. 
By the triangle inequality, for any $x\in X$ and $n\in\N$ 
we have that \begin{multline*}
nT=d(\gamma(0),\varphi^n(\gamma(0)))\leq d(\gamma(0),x)+d(x,\varphi^n(x))+d(\varphi^n(x),\varphi^n(\gamma(0)))\\
\leq2d(\gamma(0),x)+nd(x,\varphi(x)), 
\end{multline*} 
therefore $T=|\varphi|$, 
and $\im\,\gamma\subseteq\Min(\varphi)$.
For a bicombing $\sigma$ on $X$, we
call an axis $\gamma$ of $\varphi$ a 
\emph{$\sigma$-axis} 
if $\im\,\gamma|_{[s,t]}=\im\,\sigma_{\gamma(s),\gamma(t)}$ for 
all
real numbers $s<t$.

\begin{fact}[Proposition \ref{p:amalgamatedintro}]\label{p:amalgamated}
Let $G=G_1*_ZG_2$ (with $G_1\neq Z\neq G_2$), 
where $Z$ is virtually $\Z$, 
act geometrically on a proper metric space $(X,d_X)$ that admits a 
\ccc,
reversible, $G$-equivariant bicombing $\sigma$. Then there exists a separating pair of points in the boundary $\partial_\sigma X$.
\end{fact}

\begin{proof}
We briefly recall the normal form for the amalgamated product as in e.g.~\cite[Theorem 1 in I.1.2]{SerreBook80}, which we use in 
this
proof. 
For $i=1,2$,
choose a set $R_i$ of 
representatives of non-trivial right cosets of $Z$ in $G_i$.
Then each element $g\in G$ may be represented in a unique way as $g=z\cdot r_1\cdot r_2\cdot \ldots\cdot r_k$, where $z\in Z$, $r_j\in R_1\cup R_2$ for $1\leq j\leq k$, and $r_j\in R_1$ iff $r_{j+1}\in R_2$ for $1\leq j\leq k-1$.
We refer to $k$ as the length of the normal form representation of $g$, and for $1\leq j\leq k$ we refer to $r_j$ as the $j$-th term of this representation. 

Since $G$ acts geometrically on a proper metric space, it is finitely generated.
Let $\Gamma$ be the Cayley graph for $G$ over a finite set of generators $S$ contained in $G_1\cup G_2$.
Let 
\begin{equation*}
A_i:=\{g\in G:\text{the first term in the normal form representation for $g$ belongs to $G_i$}\}
\end{equation*}
for $i=1,2$.
Then $G$ is the disjoint union of $A_1$, $Z$ and $A_2$.
Since the generating set $S$ is a subset of $G_1\cup G_2$, the lengths of the normal form representations for any two group elements connected by an edge in the graph $\Gamma$
differ by at most $1$, 
and,
for $i=1,2$, 
the neighbours in the graph $\Gamma$ of the elements of $A_i$ are contained in the set $A_i\cup Z$, which implies that any path in $\Gamma$ from an element of $A_1$ to an element of $A_2$ must pass through 
the set $Z$.

Let $M$ be a cyclic subgroup of finite index of $Z$, 
and let 
$m$ be a generator of $M$.
By \cite[Proposition 5.5]{DeLa16}, 
$m$ has a $\sigma$-axis $\mu$.
Let $\alpha$ be the quasi-isometry given by $G\ni g\mapsto g\mu(0)\in X$.
Let $C>0$ be such that $M$ is $C$-dense in $Z$ (with respect to the metric $d_\Gamma$ from the graph $\Gamma$), the image $\alpha(G)$ is $C$-dense in $X$, and for all $g,g'\in G$ the inequality $C^{-1}d_\Gamma(g,g')-C\leq d_X(\alpha(g),\alpha(g'))\leq Cd_\Gamma(g,g')+C$ holds.
Put $X_i:=B(\alpha(A_i\cup Z),C+1)$ for $i=1,2$.
Clearly, $X_1$ and $X_2$ are open subsets of $X$ 
such that $X_1\cup X_2=X$.
Consider any element $x\in X_1\cap X_2$.
Then there exist $a_i\in A_i\cup Z$, where $i=1,2$,
such that $d_X(\alpha(a_i),x)<C+1$. Therefore 
$d_X(\alpha(a_1),\alpha(a_2))<2C+2$, so $d_\Gamma(a_1,a_2)<3C^2+2C$.
Since, as discussed above, the path in $\Gamma$ from $a_1$ to $a_2$ necessarily passes through $Z$, we have that $d_\Gamma(a_i,Z)<3C^2+2C$ for $i=1,2$, which, as $\alpha(M)\subseteq\im\,\mu$, gives that
\begin{multline*}
d_X(x,\im\,\mu)\leq d_X(x,\alpha(M))\leq d_X(x,\alpha(a_1))+d_X(\alpha(a_1),\alpha(M))\\
<C+1
+Cd_\Gamma(a_1,M)+C
\leq 2C+1+C(d_\Gamma(a_1,Z)+C)\\
<2C+1+C(3C^2+2C+C)=:\mathbf{C}.
\end{multline*}
Therefore $X_1\cap X_2\subseteq B(\im\,\mu,\mathbf{C})$.

Let $\xi_+,\xi_-\colon [0,\infty)\to X$ be the $\sigma$-rays defined by 
$\xi_+(t):=\mu(t)$ 
and 
$\xi_-(t):=\mu(-t)$. 
We 
show that the pair of points $[\xi_+],[\xi_-]$ disconnects the boundary $\partial_\sigma X$.
Consider any 
\linebreak[2]
$\sigma$-ray $\zeta\not\in\{\xi_+,\xi_-\}$ based in $\mu(0)$. By Proposition \ref{f:asymptotic}\ref{f:asymptotic3}, 
there exists $r>0$ such that $d(\zeta(t),\im\,\mu)\geq \mathbf{C}+2$ for any $t\geq r$, in particular there exists unique $i\in\{1,2\}$ such that $\zeta(t)\in X_i$ for 
all
$t\geq r$.
Denote 
\begin{equation*}
E_i:=\{[\zeta]:\zeta\text{ is a }\sigma\text{-ray},\zeta(0)=\mu(0),(\exists r>0)(\forall t\geq r) (\zeta(t)\in X_i)\};
\end{equation*}
we have that $\partial_\sigma X$ is the disjoint union of $E_1$, $E_2$ and $\{[\xi_-],[\xi_+]\}$. 
It remains to show that both $E_1$ and $E_2$ are non-empty open subsets of the boundary $\partial_\sigma X$. 
First, we prove the openness. Let $\zeta$, $r$ and $i$ be as above. Consider any $\sigma$-ray $\eta$ such that $d(\eta(r),\zeta(r))<1$. 
By conicality and the triangle inequality, $d(\eta(t),\im\,\mu)\geq d(\eta(r),\im\,\mu)\geq\mathbf{C}+1$ for 
all
$t\geq r$, therefore $\eta(t)\in X_i$, and $[\eta]\in E_i$. 
Now we prove that both sets $E_i$ are non-empty.
Below we prove that $E_1\neq\emptyset$, the case of $E_2$ is symmetric.
Let $g_1\in R_1$ and $g_2\in R_2$, and let $g:=g_1g_2$.
Then for all $n\in\N\setminus\{0\}$, since the normal form for $g^n$ is the $n$-fold concatenation of the normal form for $g$, we have that $g^n\in A_1$, and, 
by the discussion in the second paragraph of this proof, 
$d_\Gamma(g^n,Z)\geq 2n$.
Let $\gamma$ be a $\sigma$-axis for $g$ in $X$,
see \cite[Proposition 5.5]{DeLa15},
and define the $\sigma$-ray $\zeta$ to be the $\sigma$-ray originating in $\mu(0)$ asymptotic to $\gamma|_{[0,\infty)}$.
Proposition \ref{f:asymptotic}\ref{f:asymptotic1} 
and  
the $g$-invariance of the metric $d_X$
give that
\begin{multline*}
d_X(\zeta(|g|n),\alpha(g^n))\leq d_X(\zeta(|g|n),\gamma(|g|n))+d_X(\gamma(|g|n),\alpha(g^n))\\
\leq d_X(\zeta(0),\gamma(0))+d_X(g^n\gamma(0),g^n\mu(0))=2d_X(\gamma(0),\mu(0))=:D.	
\end{multline*}
We also have the following chain of inequalities
\begin{multline*}
	d_X(\alpha(g^n),\im\,\mu)\geq d_X(\alpha(g^n),\alpha(M))-|m|\\
	\geq d_X(\alpha(g^n),\alpha(Z))-|m|
	\geq C^{-1}d_\Gamma(g^n,Z)-C-|m|\geq C^{-1}\cdot 2n-C-|m|.
\end{multline*}
Jointly,
the last two chains of inequalities have
the following two consequences, 
which 
together
give that $[\zeta]$ belongs
to
$E_1$.
First, 
for any $n\in\N$,
if $\zeta(|g|n)$ belongs to $X_2$,
then 
openness of $X_1$ and $X_2$ 
gives
that
the $\sigma$-geodesic
from 
the point
$\zeta(|g|n)$
to 
the point
$\alpha(g^n)$, belonging to $X_1$,
passes through $X_1\cap X_2$;
therefore
\begin{equation*}
D\geq d(\alpha(g^n),\zeta(|g|n))\geq
d_X(\alpha(g^n),X_1\cap X_2)
\geq d_X(\alpha(g^n),\im\,\mu)-\mathbf{C},
\end{equation*} 
which tends to $\infty$ when $n\to\infty$; 
therefore 
$\zeta(|g|n)$ belongs to $X_2$ only for finitely many $n\in\N$,
so
$[\zeta]$ does not belong to $E_2$.
Second,
the inequality 
$d_X(\zeta(|g|n),\im\,\mu)\geq1$ 
holds for sufficiently large $n$, therefore
$\zeta$ is not asymptotic to any of the $\sigma$-rays $\xi_-$, $\xi_+$. 
\end{proof}

\begin{fact}[Proposition \ref{p:sphereinbdryintro}]\label{p:sphereinbdry}
Let $G$ be group that contains a free abelian subgroup $\Z^n\cong A<G$ and acts geometrically on a proper metric space $X$ that admits a
\ccc,
reversible, $G$-equivariant bicombing $\sigma^X$. Then $\partial_{\sigma^X} X$ contains a homeomorphic copy of $S^{n-1}$. 
Moreover, 
if $A$ is of finite index in $G$, then $\partial_{\sigma^X} X\cong S^{n-1}$.
\end{fact}

\begin{proof}
By \cite[Theorem 1.2]{DeLa16} $X$ contains an isometric copy $F$ of an $n$-dimensional normed space on which $A$ acts geometrically by translations. 
Observe that 
$F$ admits a (unique) 
\ccc,
$A$-equivariant bicombing $\sigma^F$,
 which consists of linear segments 
(recall \cite[Theorem 3.3]{DeLa15}).
Note that in general $F$ is not $\sigma^X$-convex in $X$ (see \cite[Example 6.3]{DeLa16}) --- if it was so, then the assertion would easily follow, as we would have $\sigma^F=\sigma^X|_{F\times F\times [0,1]}$, which would allow us to view $S^{n-1}\cong\partial_{\sigma^F}F$ as a subset of $\partial_{\sigma^X}X$. 
We 
shall
define a homeomorphic embedding $\Phi\colon \partial_{\sigma^F}F\to \partial_{\sigma^X}X$. If, additionally, $A$ is of finite index in $G$, then $\Phi$ 
turns
out to be a surjection. 

\smallskip
Fix a basepoint $o\in F$.
For
each
$a\in A$,
denote by
$\xi_a^F$ 
the $\sigma^F$-ray that originates in $o$ and contains $ao$,
pick a $\sigma^X$-axis $\gamma_a^X$ in $X$ 
(see \cite[Proposition 5.5]{DeLa16}),
and
put $\xi_a^X:=\gamma_a^X|_{[0,\infty)}$.

\newcommand{\cla}{(A)}
\newcommand{\clb}{(B)}
\newcommand{\clc}{(B$\sharp$)}
\newcommand{\hcla}{\myhypertarget{cl:cla}}
\newcommand{\hclb}{\myhypertarget{cl:clb}}
\newcommand{\hclc}{\myhypertarget{cl:clc}}
\newcommand{\rcla}{\hyperlink{cl:cla}{\cla}}
\newcommand{\rclb}{\hyperlink{cl:clb}{\clb}}
\newcommand{\rclc}{\hyperlink{cl:clc}{\clc}}

The construction of the map $\Phi$, which consists in continuously extending the map induced by sending $\xi_a^F$ to $\xi_a^X$ for each $a\in A$, is presented in detail below the following claim,
which is used to justify the correctness of the construction and various properties of $\Phi$.
Picking one $\sigma^X$-axis $\gamma_a^X$ for each $a\in A$, rather than considering the set of all such $\sigma^X$-axes, is more of an editorial choice; in particular,
the defined map $\Phi$ does not depend on this choice, as may be seen 
using
part \rcla{} 
of the claim 
below.
The colours are used in the further text to highlight the key places of the formulas 
and 
to aid
the presentation
of the
flow 
of 
the
argument.

\bigskip
\begin{samepage}
\noindent
\hcla
\textbf{Claim.} 
{\it Let $a\in A$ and $\gamma_1,\gamma_2$ be axes (which are not necessarily $\sigma^X$-axes) of $a$. 
Then} 
\begin{flalign*}
	&\text{\cla}&&
	d(\gamma_1(t),\gamma_2(t))\leq 2|a|+d(\gamma_1(0),\gamma_2(0))=:
	C(|a|,\gamma_1(0),\gamma_2(0))\textit{ for any }t\in\R.
	&&
\end{flalign*}
\end{samepage}

\hclb\hclc
{\it Let $a_1,a_2\in A$ and $r>0$. 
Then}
\newcommand{\ccol}[1]{\omit\hfill $#1$\hfill}
\begin{flalign*}
	&\text{\clb}&&\ccol{{\big|}d(\geod{o}{[\xi^X_{\cola{a_1}}]}^X(r),\geod{o}{[\xi^X_{\colc{a_2}}]}^X(r))-d(\xi^F_{\cola{a_1}}(r),\xi^F_{\colc{a_2}}(r)){\big|}\leq C'(\cola{a_1})+C'(\colc{a_2}),}&&\\[-0.2em]
	&&&\ccol{\textit{where }C'(a)=C(|a|,\xi_{a}^X(0),\xi_{a}^F(0))+d(\xi_{a}^X(0),o);}&&
\end{flalign*}
\vspace{-1.8em}
\begin{flalign*}
	&\text{\clc}&&d(\geod{o}{[\xi^X_{\cola{a_1}}]}^X(r),\geod{o}{[\xi^X_{\colc{a_2}}]}^X(r))\leq d(\xi^F_{\cola{a_1}}(r),\xi^F_{\colc{a_2}}(r)).
	&&
\end{flalign*}

\medskip\noindent
\textbf{Proof.} \rcla{} 
For $0\leq t\leq|a|$ the claim follows by the triangle inequality: 
\begin{multline*}
d(\gamma_1(t),\gamma_2(t))\leq d(\gamma_1(t),\gamma_1(0))+d(\gamma_1(0),\gamma_2(0))+d(\gamma_2(0),\gamma_2(t))\\
\leq |a|+d(\gamma_1(0),\gamma_2(0))+|a|.
\end{multline*} 
The 
case 
of arbitrary $t\in\R$ 
follows from $a$-invariance of the metric $d$ and shifting along axes.

\begin{figure}[h]
\includegraphics[width=\textwidth]{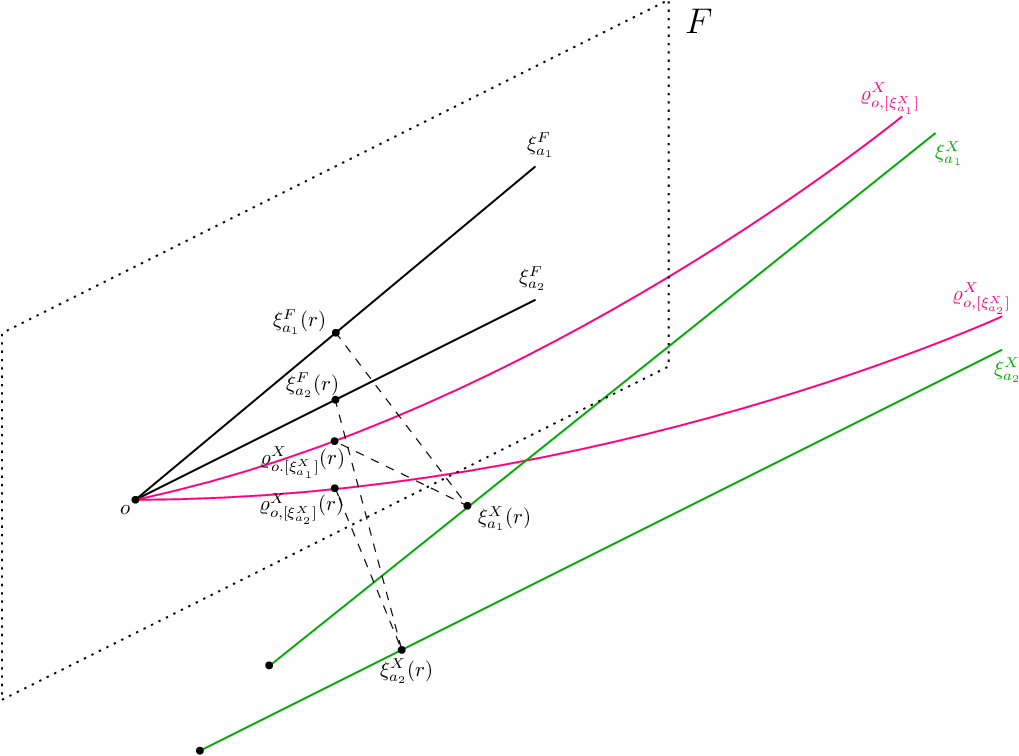}
\caption{Claim \protect\rclb.}	
\label{fig:claimb}
\end{figure}

\smallskip
\rclb{} 
By the triangle inequality, 
we obtain that 
\begin{align*}
&{\big|}d(\geod{o}{[\xi^X_{\cola{a_1}}]}^X(r),\geod{o}{[\xi^X_{\colc{a_2}}]}^X(r))-d(\xi^F_{\cola{a_1}}(r),\xi^F_{\colc{a_2}}(r)){\big|}\\
&\qquad\leq d(\geod{o}{[\xi^X_{\cola{a_1}}]}^X(r),\xi^F_{\cola{a_1}}(r))+d(\geod{o}{[\xi^X_{\colc{a_2}}]}^X(r),\xi^F_{\colc{a_2}}(r))\\
&\qquad\leq d(\geod{o}{[\xi^X_{\cola{a_1}}]}^X(r),\xi^X_{\cola{a_1}}(r))+d(\xi^X_{\cola{a_1}}(r),\xi^F_{\cola{a_1}}(r)) 
+d(\geod{o}{[\xi^X_{\colc{a_2}}]}^X(r),\xi^X_{\colc{a_2}}(r))+d(\xi^X_{\colc{a_2}}(r),\xi^F_{\colc{a_2}}(r)).
\end{align*}
The claim follows, 
as for any $i\in\{1,2\}$ the following two inequalities hold:  
by Proposition~\ref{f:asymptotic}\ref{f:asymptotic1},
we have that $d(\geod{o}{[\xi_{a_i}^X]}^X(r),\xi_{a_i}^X(r))\leq d(\geod{o}{[\xi_{a_i}^X]}^X(0),\xi_{a_i}^X(0))$;
and, 
by Claim \rcla{} applied to axes containing $\xi^X_{a_i}$ and $\xi^F_{a_i}$, 
we have that $d(\xi^X_{a_i}(r),\xi^F_{a_i}(r))\leq C(|a_i|,\xi^X_{a_i}(0),\xi^F_{a_i}(0))$.

\smallskip
\rclc{} Fix $R>r$. 
By scaling in $F$, $d(\xi^F_{a_1}(R),\xi^F_{a_2}(R))=d(\xi^F_{a_1}(r),\xi^F_{a_2}(r))\cdot R/r$. 
Therefore, 
by conicality of $\sigma^X$ 
and \rclb{} (applied for $R$ and $a_1,a_2$), 
\begin{multline*}
d(\geod{o}{[\xi^X_{\cola{a_1}}]}^X(r),\geod{o}{[\xi^X_{\colc{a_2}}]}^X(r))
\leq d(\geod{o}{[\xi^X_{\cola{a_1}}]}^X(R),\geod{o}{[\xi^X_{\colc{a_2}}]}^X(R))\cdot r/R\\
\leq(d(\xi^F_{\cola{a_1}}(R),\xi^F_{\colc{a_2}}(R))+C'(\cola{a_1})+C'(\colc{a_2}))\cdot r/R\\
= d(\xi^F_{\cola{a_1}}(r),\xi^F_{\colc{a_2}}(r))+(C'(\cola{a_1})+C'(\colc{a_2}))\cdot r/R.
\end{multline*}
The claim follows, 
as $R$ may be chosen arbitrarily large.

\bigskip
Let $F_A\subseteq\partial_{\sigma^F}F
$ be the asymptotic classes of $\sigma^F$-rays 
belonging the set 
$\{\xi^F_a:a\in A\}$. Define $\Phi_A\colon F_A\to\partial_{\sigma^X}X$ 
to be
the map induced by the map $\xi_a^F\mapsto \xi_a^X$.
Claim \rclc{}
implies that the map $\Phi_A$ is well-defined 
(i.e.~if two elements $a,b\in A$ are such that $\xi_a^F=\xi_b^F$, then $\xi_a^X$ and $\xi_b^X$ are asymptotic) 
and 1-Lipschitz from $(F_A,d_{o,1}^F)$ to $(\partial_{\sigma^X}X,d_{o,1}^X)$, where $d_{o,1}^F$ and $d_{o,1}^X$ are the metrics discussed in Section \ref{s:quasisym}. Therefore, since $F_A$ is dense in $\partial_{\sigma^F}F$, the map $\Phi_A$ can be extended continuously to a map $\Phi\colon \partial_{\sigma^F}F\to\partial_{\sigma^X}X$. 

\smallskip
We 
show that $\Phi$ is one-to-one. Since $\partial_{\sigma^F}F\cong S^{n-1}$ is compact, this will imply that $\Phi$ is a homeomorphic embedding onto its image. Let $\xi_1^F\neq\xi_2^F$ be $\sigma^F$-rays based in $o$, 
and let $a^i_n\in A$ for $i=1,2$ and $n\in\N$ 
be such that $[\xi_{a^i_n}^F]$ converges to $[\xi_i^F]$ in $\partial_{\sigma^F}F$ for $i=1,2$. We may choose $r>0$ such that $d(\xi_1^F(r),\xi_2^F(r))\geq8$, and $N\in\N$ such that for any $n\geq N$ and $i=1,2$ we have $d(\xi_{a^i_n}^F(r),\xi_i^F(r))\leq1$.
In particular, 
$d(\xi_{a^1_N}^F(r),\xi_{a^2_N}^F(r))\geq6$.
By Claim~\rclb{} and scaling in $F$, for any $R>r$,
\begin{multline*}
d(\geod{o}{[\xi_{\cola{a^1_N}}^X]}^X(R),\geod{o}{[\xi_{\colc{a^2_N}}^X]}^X(R))\geq d(\xi_{\cola{a^1_N}}^F(R),\xi_{\colc{a^2_N}}^F(R))-C'(\cola{a^1_N})-C'(\colc{a^2_N})\\
=d(\xi_{\cola{a^1_N}}^F(r),\xi_{\colc{a^2_N}}^F(r))\cdot R/r-C'(\cola{a^1_N})-C'(\colc{a^2_N})\geq 6R/r-C'(\cola{a^1_N})-C'(\colc{a^2_N}),
\end{multline*} 
which is greater than $5R/r$ for sufficiently large $R$.
Using Claim \rclc{} and scaling in $F$, 
we obtain that
for any $n\geq N$
\begin{multline*}
d(\geod{o}{[\xi_{\colb{a^i_n}}^X]}^X(R),\geod{o}{[\xi_{\cola{a^i_N}}^X]}^X(R))
\leq d(\xi_{\colb{a^i_n}}^F(R),\xi_{\cola{a^i_N}}^F(R))\\[-0.3em]
\leq d(\xi_{\colb{a^i_n}}^F(R),\xi_i^F(R))+d(\xi_i^F(R),\xi_{\cola{a^i_N}}^F(R))
\leq R/r+R/r=2R/r.
\end{multline*}
Therefore
\begin{multline*}
d(\geod{o}{[\xi_{\colb{a^1_n}}^X]}^X(R),\geod{o}{[\xi_{\cold{a^2_m}}^X]}^X(R))\\
\geq d(\geod{o}{[\xi_{\cola{a^1_N}}^X]}^X(R),\geod{o}{[\xi_{\colc{a^2_N}}^X]}^X(R))-d(\geod{o}{[\xi_{\colb{a^1_n}}^X]}^X(R),\geod{o}{[\xi_{\cola{a^1_N}}^X]}^X(R))-d(\geod{o}{[\xi_{\cold{a^2_m}}^X]}^X(R),\geod{o}{[\xi_{\colc{a^2_N}}^X]}^X(R))\\
\geq 5R/r-2R/r-2R/r
\geq R/r
\end{multline*} 
for any $n,m\geq N$ and sufficiently large $R$.
Therefore, passing to the limit with $n$ and $m$, one obtains that $d(\geod{o}{\Phi([\xi_1^F])}^X(R),\geod{o}{\Phi([\xi_2^F])}^X(R))
\geq R/r>0$, which implies that $\Phi([\xi_1^F])\neq\Phi([\xi_2^F])$.

\medskip
Now we prove the last part of the statement of this proposition:
assume 
additionally
that $A$ is of finite index in $G$;
we show that then $\Phi$ is onto.
Let $C>0$ be such that the set $Ao$ of $A$-translates of $o$ 
satisfies
$B(Ao,C)=X$.
Let $\xi^X$ be a $\sigma^X$-ray based in $o$. For every $n\in\N$ there exists $a_n\in A$ such that $d(a_no,\xi^X(n))\leq C$. 
It suffices to show that $\geod{o}{[\xi_{a_n}^X]}^X(r)$ converges to $\xi^X(r)$ when $n\to\infty$ for any $r>0$. 
First, observe that $||a_n|-n|=|d(a_no,o)-d(o,\xi^X(n))|\leq d(a_no,\xi^X(n))\leq C$, therefore, in particular, $|a_n|\to\infty$.
Second, since translating a $\sigma^X$-axis of $a\in A$ 
by
an element of the centraliser $C(a)\supseteq A$ produces a $\sigma^X$-axis of $a$, and $B(Ao,C)=X$, there exists a $\sigma^X$-axis $\theta_a^X$ of $a$ such that $d(\theta_a^X(0),o)\leq C$; put $\zeta_a^X:=\theta_a^X|_{[0,\infty)}$;
by Claim \rcla,
the 
$\sigma^X$-rays
$\zeta_a^X$ and $\xi_a^X$ are asymptotic.
By  
the triangle inequality
and 
Proposition \ref{f:asymptotic}\ref{f:asymptotic1},
we have that
\begin{multline*}
	d(\geod{o}{[\xi_{a_n}^X]}^X(|a_n|),\xi^X(|a_n|))=d(\geod{o}{[\emp{\zeta_{a_n}^X}]}^X(|a_n|),\xi^X(|a_n|))\\
	\leq d(\geod{o}{[\zeta_{a_n}^X]}^X(|a_n|),\zeta_{a_n}^X(|a_n|))+d(\zeta_{a_n}^X(|a_n|),a_no)+d(a_no,\xi^X(n))+d(\xi^X(n),\xi^X(|a_n|))\\
	\leq d(\geod{o}{[\zeta_{a_n}^X]}^X(0),\zeta_{a_n}^X(0))+d(a_n\zeta_{a_n}^X(0),a_no)+C+||a_n|-n|
	\leq C+C+C+C=4C;
\end{multline*} 
then the
conicality of $\sigma^X$ gives that $d(\geod{o}{[\xi_{a_n}^X]}^X(r),\xi^X(r))\leq 4Cr/|a_n|$ whenever $r\leq|a_n|$. Therefore, since $|a_n|\to\infty$, 
we have that
$d(\geod{o}{[\xi_{a_n}^X]}^X(r),\xi^X(r))\to0$
for all $r\geq0$.
\end{proof}

\newcommand{\nbhd}{neighbourhood}
\newcommand{\has}{\bar{H}}
\newcommand{\hasred}{\tilde{\bar{H}}}
\newcommand{\hc}{\check{H}}
\newcommand{\htpy}{\mathfrak{h}}

\section{Almost geodesic completeness}\label{s:agc}
A space $X$ that admits a bicombing $\sigma$ is 
\emph{almost $\sigma$-geodesically complete} 
if
for some (equivalently, for all --- see Proposition \ref{f:asymptotic}\ref{f:asymptotic1})
basepoint $o\in X$
there exists a universal constant $C>0$ such that for each point $x\in X$ there is a $\sigma$-ray $\xi$ such that $\xi(0)=o$ and $\im\,\xi\cap\overline{B}(x,C)\neq\emptyset$.
In this section we prove the following theorem. 

\begin{theorem}[Theorem \ref{t:agcc3intro}; cf.~{\cite[Corollary 3]{GeOn07}}]\label{t:agcc3}
Assume that $X$ is a proper 
finite-dimensional 
geodesic metric space that admits a 
\ccc{}
geodesic bicombing 
$\sigma$
and a cocompact group action via isometries,
such that $|\partial_\sigma X|\geq 2$. 
Then $X$ is 
almost 
$\sigma$-geodesically 
complete.	
\end{theorem}

\begin{remark}\phantomsection\label{u:oneptbdry1}
	\begin{enumerate}[(i)]
	\item The assumption that the boundary has at least two points can be relaxed to the condition that $\partial_\sigma X\neq\emptyset$ (i.e., equivalently, that $X$ is non-compact) if the group action in the statement of Theorem \ref{t:agcc3} is such that the bicombing $\sigma$ is additionally $G$-equivariant: a standard argument, which works in the CAT(0) realm, allows one to find 
	for a given $\sigma$-ray $\xi$ 
	a $\sigma$-ray pointing in `the opposite direction' compared 
	to $\xi$ by transporting (approximations of) small subsegments of $\xi$.
	
	\item We do not know whether the assumption that $|\partial_\sigma X|\geq2$ may be relaxed to the assumption that $\partial_\sigma X\neq\emptyset$ in general. 
	We note here that a space $X$ with $|\partial_\sigma X|=1$ 
	necessarily
	must not 
	be almost $\sigma$-geodesically complete, as otherwise it would be quasi-isometric to the real half-line, hence cannot be acted upon cocompactly by a group.
	\end{enumerate}	

\end{remark}

The proof from \cite{Ontaneda05,GeOn07}, where the space $X$ is assumed to be CAT(0),
can be 
translated to the context of
spaces 
admitting a 
\ccc{}
geodesic bicombing. 
We expand on it below.

\subsection{Preparatory lemmas}\label{sbs:agcprep}
 
Recall Definition \ref{d:ellexp} and Proposition \ref{p:ellexp}. 
Let $\mathrm{Cone}_o(A):=\exp_o(A\times[0,\infty])$ 
for any set $A\subseteq\partial_\sigma X$ and basepoint $o\in X$.

\begin{lemma}[cf.~proof of~{\cite[Main Theorem]{GeOn07}}]\label{l:agcD}
Let $X$ be a proper metric space that admits a 
\ccc{}
geodesic bicombing.
Then 
for every
non-empty
closed set $A\subseteq\partial_\sigma X$ there exists 
a closed set $D\subseteq\bdry{X}{\sigma}$ such that (i)~$D\cap\partial_\sigma X=A$, (ii)~$\mathrm{Cone}_o(A)\subseteq D$, 
(iii)~$D\cap X\subseteq \overline{B}_X(\mathrm{Cone}_o(A)\cap X,1)$, 
and
(iv)~$D$ is a strong deformation retract of $\bdry{X}{\sigma}$.

\end{lemma}

\begin{proof}
For any $\bar{a}\in A$ 
and 
$\bar{x}\in\bdry{X}{\sigma}$, 
we have by convexity of $\sigma$ that the function $\delta_{\bar{x}}^{\bar{a}}\colon[0,\infty)\cap[0,\ell_o(\bar{x})]\to\R$ given by $\delta_{\bar{x}}^{\bar{a}}(s)=d(\exp_o(\bar{a},s),\exp_o(\bar{x},s))$ is non-decreasing, strictly increasing on $(\delta_{\bar{x}}^{\bar{a}})^{-1}((0,\infty))$ and continuous.
Since $A$ is compact, the function $\mu_{\bar{x}}^A\colon[0,\infty)\cap[0,\ell_o(\bar{x})]\to\R$ given by $\mu_{\bar{x}}^A(s)=\min_{\bar{a}\in A}\delta_{\bar{x}}^{\bar{a}}(s)$ is well-defined, and is non-decreasing and strictly increasing on $(\mu_{\bar{x}}^A)^{-1}((0,\infty))$, as the functions $\delta_{\bar{x}}^{\bar{a}}$ are;
it is 
continuous 
by the following argument.
Let $s\in [0,\infty)\cap[0,\ell_o(\bar{x})]$.
Consider any sequence $(s_n)\subseteq[0,\infty)\cap[0,\ell_o(\bar{x})]$ converging to $s$ from above. Let $\bar{a}\in A$ be such that $\mu_{\bar{x}}^A(s)=\delta_{\bar{x}}^{\bar{a}}(s)$.
Then we have $\delta_{\bar{x}}^{\bar{a}}(s_n)\geq\mu_{\bar{x}}^A(s_n)\geq\mu_{\bar{x}}^A(s)=\delta_{\bar{x}}^{\bar{a}}(s)$.
As the left hand side converges to the right hand side as $n\to\infty$, we have that $\mu_{\bar{x}}^A(s_n)\to\mu_{\bar{x}}^A(s)$.
Now, assume that we have a sequence $(s_n)\subseteq[0,\infty)\cap[0,\ell_o(\bar{x})]$ approaching $s$ from below.
Let $\bar{a}_n$ be such that $\mu_{\bar{x}}^A(s_n)=\delta_{\bar{x}}^{\bar{a}_n}(s_n)$.
By compactness of $A$, each subsequence of $(n)_{n\in\N}$ admits a subsequence $(n_k)_{k\in\N}$ such that $\bar{a}_{n_k}$ is convergent to some $\bar{a}\in A$.
We have the following inequalities: $\mu_{\bar{x}}^A(s)\geq\mu_{\bar{x}}^A(s_{n_k})=d(\exp_o(\bar{x},s_{n_k}),\exp_o(\bar{a}_{n_k},s_{n_k}))$. By continuity of $\exp_o$, passing to the limit with $k$ 
we obtain that $\mu_{\bar{x}}^A(s)\geq\mu_{\bar{x}}^A(s_{n_k})\to\delta_{\bar{x}}^{\bar{a}}(s)\geq\mu_{\bar{x}}^A(s)$. 
Therefore $\mu_{\bar{x}}^A(s_n)\to\mu_{\bar{x}}^A(s)$.

Define $\omega\colon\bdry{X}{\sigma}\to[0,\infty]$ by
$\omega(\bar{x})=\sup\{s\in[0,\infty)\cap[0,\ell_o(\bar{x})]:\mu_{\bar{x}}^A(s)\leq 1\}$ 
(in slightly informal terms: `walk  
from $o$ along the $\sigma$-geodesic/ray to $\bar{x}$ 
until 
reaching $\bar{x}$ or
diverging 
to a ``sphere-wise'' distance at least 1
from all 
of the
$\sigma$-rays 
that begin in
$o$ 
and end in $A$;
the distance covered is $\omega(\bar{x})$')  
and let $D:=\{\bar{x}\in\bdry{X}{\sigma}:\ell_o(\bar{x})=\omega(\bar{x})\}$.
It easily follows that $D$ 
satisfies (ii) and (iii), and that $A\subseteq D\cap\partial_\sigma X$. 
The other inclusion 
required by
property (i) is satisfied, 
as
\begin{multline*}
D\cap\partial_\sigma X=\{\bar{x}\in\partial_\sigma X:(\forall s\geq0)(\exists \bar{a}\in A)(d(\geod{o}{\bar{x}}(s),\geod{o}{\bar{a}}(s))\leq1)\}\\
\subseteq \{\bar{x}\in\partial_\sigma X:(\forall s\geq0)(\exists \bar{a}\in A)(d_{o,2}(\bar{x},\bar{a})\leq s^{-1})\}
\subseteq\overline{A}
=A
\end{multline*}
(recall Proposition \ref{p:quasisym}\ref{p:quasisym1}).
It is sufficient to prove that $\omega$ is continuous, as then
it 
immediately follows that $D$ is closed, 
and 
that
property (iv) is satisfied,
as then we 
have the following retraction: $\bdry{X}{\sigma}\times[0,\infty]\ni(\bar{x},s)\mapsto\exp_o(\bar{x},\max(s,\omega(\bar{x})))\in\bdry{X}{\sigma}$.

\newcommand{\cli}{(A)}
\newcommand{\clii}{(B)}
\newcommand{\cliii}{(C)}
\newcommand{\hcli}{\myhypertarget{cl:8i}}
\newcommand{\hclii}{\myhypertarget{cl:8ii}}
\newcommand{\hcliii}{\myhypertarget{cl:8iii}}
\newcommand{\rcli}{\hyperlink{cl:8i}{\cli}}
\newcommand{\rclii}{\hyperlink{cl:8ii}{\clii}}
\newcommand{\rcliii}{\hyperlink{cl:8iii}{\cliii}}

\bigskip\noindent
\hcli\hclii\hcliii
\textbf{Claim.} {\it Assume that $(\bar{x}_n)\subseteq\bdry{X}{\sigma}$ converges to $\bar{x}\in\bdry{X}{\sigma}$ and $\omega(\bar{x}_n)$ converges to some $t\in[0,\infty]$. Then
\emph{\cli{}} if $\ell_o(\bar{x})>s>t$, then $\mu_{\bar{x}}^A(s)\geq 1$;
\emph{\clii{}} if $t>s\geq 0$ then $\mu_{\bar{x}}^A(s)\leq 1$;
\linebreak[1]
\emph{\cliii{}} $\omega(\bar{x})=t$.}

\smallskip\noindent
\textbf{Proof.} \rcli{} If not, then there exists $\bar{a}\in A$ such that $d(\exp_o(\bar{x},s),\exp_o(\bar{a},s))=\delta_{\bar{x}}^{\bar{a}}(s)<1$. Since $\bar{x}_n\to\bar{x}$, for large enough $n$ we have that $\delta_{\bar{x}_n}^{\bar{a}}(s)=d(\exp_o(\bar{x}_n,s),\exp_o(\bar{a},s))<1$, and, as $\ell_o$ is continuous, $s<\ell_o(\bar{x}_n)$. This implies that $\omega(\bar{x}_n)>s$ for sufficiently large $n$, thus $t\geq s$. Contradiction.

\smallskip
\rclii{} 
Let $\bar{a}_n\in A$ be such that $\delta_{\bar{x}_n}^{\bar{a}_n}(\omega(\bar{x}_n))=\mu_{\bar{x}_n}^A(\omega(\bar{x}_n))$. 
Then we have that $1\geq \delta_{\bar{x}_n}^{\bar{a}_n}(\omega(\bar{x}_n))=d(\exp_o(\bar{x}_n,\omega(\bar{x}_n)),\exp_o(\bar{a}_n,\omega(\bar{x}_n)))$. 
Since $\omega(\bar{x}_n)\to t$, for large enough $n$ we have that $\omega(\bar{x}_n)\geq s$, therefore $1\geq d(\exp_o(\bar{x}_n,s),\exp_o(\bar{a}_n,s))$. 
Therefore, by compactness of $A$, there exists $\bar{a}\in A$ such that $1\geq d(\exp_o(\bar{x},s),\exp_o(\bar{a},s))=\delta_{\bar{x}}^{\bar{a}}(s)
\geq\mu_{\bar{x}}^A(s)$.

\smallskip
\rcliii{} 
First, 
note that passing to the limit with $\omega(\bar{x}_n)\leq\ell_o(\bar{x}_n)$ gives that $t\leq\ell_o(\bar{x})$.
If 
$t=\ell_o(\bar{x})>0$, 
then 
by 
\rclii{}
we have that $\omega(\bar{x})\geq s$ for any $s<t$,
therefore $t=\ell_o(\bar{x})\geq\omega(\bar{x})\geq t$,
so $\omega(\bar{x})=t$.
If $t=\ell_o(\bar{x})=0$, then $\omega(\bar{x})=0=t$, 
as $0\leq\omega(\bar{x})\leq\ell_o(\bar{x})=0$. 
Otherwise, 
we have that
$t<\ell_o(\bar{x})$.
Claim 
\rcli{}
implies that 
$\mu_{\bar{x}}^A(t)\geq 1$.
Since $\ell_o(\bar{x}_n)\to\ell_o(\bar{x})$, we have for sufficiently large $n$ that $\omega(\bar{x}_n)<\ell_o(\bar{x}_n)$; 
also note that $\omega(\bar{x}_n)\geq 1/2$, which in the limit implies that $t\geq 1/2>0$, since $\mu_{\bar{x}_n}^A(\omega(\bar{x}_n))=1$ and the diameter of $B(o,1/2)$ is not greater than 1; 
therefore 
Claim 
\rclii{}
implies that $\mu_{\bar{x}}^A(t)\leq 1$. 
Since $\mu_{\bar{x}}^A$ is increasing on $(\mu_{\bar{x}}^A)^{-1}((0,\infty))$, $t$ is the unique number such that $\mu_{\bar{x}}^A(t)=1$, and $\omega(\bar{x})=t$.
\bigskip

Continuity of $\omega$ now follows, as
each subsequence of a convergent sequence 
in 
$\bdry{X}{\sigma}$ admits a subsequence that satisfies the assumptions of 
Claim 
\rcliii{}
above.
\end{proof}

For a simplicial complex $K$, below we consider it with 
the piecewise--unit-$\ell^\infty$ metric.
That is, we endow it with the gluing metric arising from the identification of each $k$-simplex $[v_0,\ldots,v_k]$ of $K$ with the subspace $\{(\lambda_0,\ldots,\lambda_k):\lambda_0,\ldots,\lambda_k\geq 0\text{ and }\lambda_0+\ldots+\lambda_k=1\}$ of $\R^{k+1}$ with the supremum metric.

\smallskip
Two spaces $X,Y$ 
have the same \emph{bounded homotopy type} 
if there exist (continuous) maps $f\colon X\to Y$ and $g\colon Y\to X$, and \emph{bounded homotopies} $\htpy^X\colon X\times[0,1]\to X$ between $g\circ f$ and $\mathrm{id}_X$, and $\htpy^Y\colon Y\times[0,1]\to Y$ between $f\circ g$ and $\mathrm{id}_Y$, i.e.~homotopies such that the diameters of the trajectories $\htpy^X(x,\cdot\,)$ for $x\in X$ and $\htpy^Y(y,\cdot\,)$ for $y\in Y$ are bounded by a constant independent of the choice of $x$ and $y$. The maps $f$ 
and 
$g$ above are called \emph{bounded homotopy equivalences}.

\begin{lemma}[cf.~{\cite[Lemma I.7A.15]{BrHae99}}]\label{l:agcI7A15}
	Let $X$ be a metric space 
	that admits 
	a locally finite open cover $\mathcal{U}=\{B(x_i,\epsilon):i\in I\}$ for some $\epsilon>0$.
	Assume that for $k\in\{1,3\}$ each ball $B(x_i,k\epsilon)$ admits a 
	continuous function
	$\sigma^{i,k\epsilon}\colon B(x_i,k\epsilon)\times B(x_i,k\epsilon)\times[0,1]\to B(x_i,k\epsilon)$ satisfying $\sigma^{i,k\epsilon}(x,x',0)=x$, $\sigma^{i,k\epsilon}(x,x',1)=x'$ for all $x,x'\in B(x_i,k\epsilon)$, 
	and that 
	these functions are
	such 
	that $\sigma^{i,k\epsilon}|_{(B(x_i,k\epsilon)\cap B(x_j,k\epsilon))^2\times[0,1]}=\sigma^{j,k\epsilon}|_{(B(x_i,k\epsilon)\cap B(x_j,k\epsilon))^2\times[0,1]}$ for 
	all
	${i,j\in I}$.
	Then the nerve $K$ of the 
	cover 
	$\mathcal{U}$ is a locally finite simplicial complex of the same bounded homotopy type as $X$.
	
	Moreover, if the space $X$ is proper, geodesic and admits a cocompact 
	group
	action via isometries, and $\dim K<\infty$, then the constructed bounded homotopy equivalences
	$f\colon X\to K$ and $g\colon K\to X$ 
	are quasi-isometries. 
\end{lemma}

\begin{remark}\label{u:agcI7A15}
If the space $X$ admits a conical bicombing $\sigma$, then
one may 
construct the
families $\{\sigma^{i,\epsilon}:i\in I\}$ and $\{\sigma^{i,3\epsilon}:i\in I\}$ 
satisfying the properties required for them in the statement above
by restricting $\sigma$ to appropriate balls.
\end{remark}

\begin{proof}
	The proof that $X$ and the nerve of $\mathcal{U}$ 
	are of the same bounded homotopy type
	can be done using the construction from the proof of \cite[Lemma I.7A.15]{BrHae99}, 
	with the change that instead of using the unique geodesic between a pair of points $x,x'\in B(x_i,k\epsilon)$, where $i\in I$ and $k\in\{1,3\}$, one may use the segment $t\mapsto\sigma^{i,k\epsilon}(x,x',t)$. We present its outline below.
	
	\smallskip
	Denote by $v_i$ the vertex in $K$ corresponding to the ball $B(x_i,\epsilon)\in\mathcal{U}$.
	
	The map $f\colon X\to K$ is constructed via a 
	partition of unity,
	which is 
	almost 
	subordinate 
	---
	it is 
	subordinate upon ignoring the (topological) boundaries of the supports
	---
	to 
	the
	(locally finite) open cover $\mathcal{U}$:
	given $x\in X$ and $i\in I$, define $\varphi_i(x):=\max(0,\epsilon-d(x_i,x))$, and 
	define $f(x)$ 
	to have the 
	$v_i$-coordinate 
	equal to
	$\varphi_i(x){\big/}\sum_{j\in I}\varphi_j(x)$.
	Observe that for any $i\in I$ we have 
	the inclusion
	$f(B(x_i,\epsilon))\subseteq\st(v_i)$, where  $\st(v)=\bigcup\{\intr\,\Delta:v\in\Delta,\,\Delta\text{ is a simplex of }K\}$ is the open star of the vertex $v$ in $K$. In particular, $f(B(x_i,\epsilon))\subseteq B(f(x_i),2)$.
	
	The map $g\colon K\to X$ is constructed inductively over the skeleta of $K$, maintaining the property that for each $v_i$ we have that $g(K^{(d)}\cap\St_{\max}(v_i))\subseteq B(x_i,\epsilon)$, where $K^{(d)}$ is the $d$-skeleton of $K$ and $\St_{\max}(v_i)$ consists of $y\in K$ whose 
	$v_i$-coordinate 
	is not smaller than 
	any
	other 
	of
	its
	$v_j$-coordinates (where $j$ ranges over $I$).
	Put $g(v_i):=x_i$.
	Assume that we have defined $g$ on $K^{(d)}$. We shall extend it to $K^{(d+1)}$ for each $(d+1)$-simplex $\Delta$ in $K$ separately.
	Let $y_c$ be the central point of $\Delta$ and pick any point $c\in\bigcap\{B(x_j,\epsilon):v_j\in\Delta^{(0)}\}$.
	Given a point $y\in\Delta^{(d)}\cap\St_{\max}(v_j)$, where $v_j\in\Delta^{(0)}$,  
	one may define $g$ to map the segment $[y,y_c]$ via the map $ty+(1-t)y_c\mapsto\sigma^{j,\epsilon}(g(y),c,t)\in B(x_j,\epsilon)$, where $t\in[0,1]$. This gives a well-defined continuous extension of $g$ to the whole $\Delta$, as the functions $\sigma^{j,\epsilon}$ are assumed to agree with each other on intersections of their domains, are continuous and their domains are open in $X\times X\times [0,1]$, and the considered segments $[y,y_c]$ cover the whole $\Delta$.
	Observe that for any $i\in I$ we have that
	\begin{multline*}
		g(\st(v_i))\subseteq g\bpl\!\!\:\bigcup\{\St_{\max}(v_j):v_j=v_i\text{ or }\{v_i,v_j\}\in K^{(1)}\}\bpr\\
		\subseteq g\bpl\!\!\:\bigcup\{\St_{\max}(v_j):d(x_i,x_j)<2\epsilon\}\bpr\subseteq \bigcup\{B(x_j,\epsilon):d(x_i,x_j)<2\epsilon\}\subseteq B(x_i,3\epsilon). 	
	\end{multline*}
	
	\smallskip
	Regarding the bounded homotopy between $g\circ f$ and the identity of $X$, observe that for any $i\in I$ we have that $g(f(B(x_i,\epsilon)))\subseteq g(\st(v_i))\subseteq B(x_i,3\epsilon)$.
	Therefore, for a point $x\in B(x_i,\epsilon)$, we may define the desired bounded homotopy to contain the map $(x,t)\mapsto\sigma^{i,3\epsilon}(g(f(x)),x,t)$. Similarly as above, it is a well-defined bounded homotopy due to the assumptions on the functions $\sigma^{i,3\epsilon}$.
	
	Regarding the bounded homotopy between $f\circ g$ and the identity of $K$, observe that for any simplex $\Delta$ of $K$ we have that
	\begin{multline*}
	f(g(\Delta))\subseteq f\bpl g\bpl\!\!\:\bigcup\{\St_{\max}(v_j):v_j\in\Delta^{(0)}\}\bpr\!\!\:\bpr\subseteq f\bpl\!\!\:\bigcup\{B(x_j,\epsilon):v_j\in\Delta^{(0)}\}\bpr\\
	\subseteq\bigcup\{\st(v_j):v_j\in\Delta^{(0)}\}.	
	\end{multline*}  
	It is now a standard fact that
	one may construct a homotopy $\htpy^K$ between $f\circ g$ and $\id_K$
	such that for each
	simplex $\Delta$ of $K$ 
	the homotopy 
	$\htpy^K$
	moves
	points of $\Delta$
	along piecewise-linear 
	segments contained in the union of open stars of vertices of $\Delta$;
	in particular, 
	$\htpy^K$
	is a bounded homotopy.
		
	\medskip
	Regarding the proof of the `moreover' part, note that we have proved above that the compositions $g\circ f$ and $f\circ g$ are at finite distance from the identity maps on $X$ and $K$, respectively, therefore it is sufficient to prove that $f$ and $g$ are coarsely Lipschitz. 
	
	\smallskip
	Regarding the proof 
	for $f$, first note that we have 
	that for any compact $A\subseteq X$
	there exists a constant $C_A$ such that
	the cardinality of any $\epsilon$-net in $A$ is bounded by $C_A$.
	Indeed, for any $(\epsilon/2)$-net $N$ and $\epsilon$-net $M$ in $A$ 
	we have that each $\epsilon$-ball centred in an element of $M$ contains an element of $N$, and each element of $N$ is contained in at most 
	one
	$\epsilon$-ball centred in an element of $M$; 
	therefore $|N|\geq|M|$, so it is sufficient to take $C_A$ equal to the cardinality of any $(\epsilon/2)$-net in $A$.  

\newcommand{\net}{\odot}
	
	Our first goal is to show that $\sup\{d(f(x),f(x')):x,x'\in X,d(x,x')\leq 1\}<\infty$.
	Take any $o\in X$. 
	Since $X$ admits a cocompact group action, there exists $R>0$ such that the translates of $B(o,R)$ cover $X$. Let $C:=C_{\overline{B}(o,R+1+\epsilon)}$ be as in the 
	paragraph
	above.
	Let $x,x'\in X$
	be
	such that $d(x,x')\leq 1$. 
	By picking a maximal subset of points of pairwise distances not smaller than $\epsilon$ from the set $\{x_i:i\in I,B(x_i,\epsilon)\cap \overline{B}(x,1)\neq\emptyset\}$,
	one obtains a set $I_\net\subseteq I$ such that 
	$|I_\net|\leq C$, 
	as any group element that translates $x$ into $B(o,R)$ translates the set $\{x_i:i\in I,B(x_i,\epsilon)\cap \overline{B}(x,1)\neq\emptyset\}$ into $B(o,R+1+\epsilon)$. 
	Since one may 
	connect
	$x$ 
	and
	$x'$ 
	with
	a geodesic contained in $\overline{B}(x,1)$,
	there exists a chain of points $x_{i_1},\ldots,x_{i_k}\in B(x,1+\epsilon)$ such that 
	$i_j\in I$
	for $1\leq j\leq k$, 
	$B(x_{i_j},\epsilon)\cap B(x_{i_{j+1}},\epsilon)\neq\emptyset$ 
	for $1\leq j<k$, $x\in B(x_{i_1},\epsilon)$ 
	and $x'\in B(x_{i_k},\epsilon)$.
	Observe that 
	for each $i_j$, where $1\leq j\leq k$, there exists $i_j^\net\in I_\net$ such that $d(x_{i_j},x_{i_j^\net})<\epsilon$; 
	in particular, $f(x_{i_j^\net})$ belongs to $B(f(x_{i_j}),2)$, 
	thus $f(B(x_{i_j},\epsilon))\subseteq B(f(x_{i_j}),2)\subseteq B(f(x_{i_j^\net}),4)$.
	Then 
	the chain of points $x_{i_1^\net},\ldots,x_{i_k^\net}$ is such that $i_j^\net\in I_\net$ for $1\leq j\leq k$, $B(f(x_{i_j^\net}),4)\cap B(f(x_{i_{j+1}^\net}),4)\neq\emptyset$ for $1\leq j<k$, $f(x)\in B(f(x_{i_1^\net}),4)$ 
	and $f(x')\in B(f(x_{i_k^\net}),4)$.
	By taking the shortest among such chains, one may assume that $k\leq |I_\net|\leq C$;
	then we have that $d(f(x),f(x'))\leq 4k\leq 4C$.
	Finally, for arbitrary $x,x'\in X$, by considering 
	a
	geodesic between $x$ and $x'$, and the triangle inequality, one may obtain that $d(f(x),f(x'))\leq 4C\lceil d(x,x')\rceil\leq 4Cd(x,x')+4C$, where $\lceil\cdot\rceil$ is the ceiling function.
	
	\smallskip
	Regarding the proof for $g$, 
	recall 
	the identification of $d$-simplices with appropriate subset of $\R^{d+1}$ with the supremum metric,
	and 
	observe that for any $i\in I$, 
	assigning to an element of the 
	open
	star $\st(v_i)$
	the 
	value of 
	its
	$v_i$-coordinate
	is
	a 1-Lipschitz function.  	
	Therefore, for any $y\in K$
	and $i\in I$ such that $y\in\st(v_i)$ we have that $B(y,\lambda_i)\subseteq\st(v_i)$, where $\lambda_i$ is 
	the 
	value of the
	$v_i$-coordinate 
	of $y$.
	Therefore, as $\dim K<\infty$, for each $y\in K$ there exists a vertex $v_{i_y}$ such that $B(y,1/\dim(K))\subseteq\st(v_{i_y})$. Consider $y,y'\in K$. If $d(y,y')\leq1/(2\dim(K))$, then $d(f(y),f(y'))\leq 6\epsilon$ (as $f(\st(v_{i_y}))\subseteq B(x_{i_y},3\epsilon)$). Therefore for arbitrary $y,y'\in K$ we have that 
	\begin{equation*}
	d(f(y),f(y'))\leq\left\lceil\frac{d(y,y')}{(2\dim(K))^{-1}}\right\rceil\cdot6\epsilon
	\leq 12\dim(K)\epsilon d(y,y')+6\epsilon.\qedhere
	\end{equation*}
\end{proof}

\medskip
The default 
cohomology theory
in this 
section
is
the Alexander--Spanier 
cohomology
$\has^*$.
We note that all of the reasonings in the remaining part of this subsection 
also
work with the simplicial 
cohomology in the place of the Alexander--Spanier 
cohomology;
the extra properties of the latter cohomology theory, mainly the consequences of admitting more so-called taut pairs, see \cite[above 6.1.7, and Section 6.6]{SpanierBook66}, will be used mainly in Remark \ref{r:cohoeq} and the proof of 
Theorem
\ref{t:agcmt}.

For a topological space $X$, the 
\emph{(Alexander--Spanier) cohomology with compact support $\has_c^*(X)$}
is
defined as the direct limit of the system 
$\{\csch{\has}{*}{X}{K}:K\subseteq X\text{ compact}\}$ (with the 
homomorphism 
in this system
being the maps induced by inclusions), see \cite[Theorem 6.6.15]{SpanierBook66}.
We note here that 
excision, \cite[Theorem 6.4.4]{SpanierBook66}, allows us to view the groups $\has_c^*(V)$, where $V\subseteq X$ is open, 
also
as 
the 
following
direct limit: 
\begin{equation}\label{eq:cschopen}
	\has_c^*(V)\cong
	\lim_{\longrightarrow}\{\has^*(X,X\setminus K):K\subseteq V\text{ compact}\}.
\end{equation} 

\smallskip
Let $X$
be
a proper metric space, $i\in\N$ and $T$ be a function $[0,+\infty)\to[0,+\infty)$. Then we say that the group $\has^i_c(X)$ of 
cohomology
with compact support
is:
\begin{enumerate}[$\bullet$]
\item 
\emph{$T$-uniformly trivial}, 
if
the
map 
$H^i_c(B(x,r))\to H^i_c(B(x,r+T(r)))$ 
induced by 
(the system of) inclusions,
recall \eqref{eq:cschopen}, 
is trivial;

\item 
\emph{uniformly trivial}, if it is $T$-uniformly trivial for some function $T$;

\item 
\emph{$T$-\nbhd-uniformly trivial}, 
if
for each $r>0$ 
and
compact set $A$ 
contained in a 
closed 
ball of radius $r$
the 
map
$\csch{\has}{i}{X}{A}\to\csch{\has}{i}{X}{\overline{B}(A,T(r))}$ 
induced by inclusion 
is trivial.
\end{enumerate}

\smallskip
Let $K$ be a locally finite simplicial complex.
By the 
\emph{simplicial cohomology with compact support}
$H_{\spl,c}^*(K)$ we mean the 
one
resulting from 
the (co)chain complex 
\begin{equation*}
C_{\spl,c}^*(K)=\{\varphi\in C_\spl^*(K):\varphi\text{ is supported in a finite subcomplex of $K$}\}.
\end{equation*}
For $i\in\N$ and a function $T\colon[0,\infty)\to[0,\infty)$, 
we say that 
the group $H_{\spl,c}^i(K)$ is 
\linebreak[4]
\emph{$T$-\nbhd-uniformly trivial} if 
for each $r>0$ 
and
$i$-cocycle $\varphi$ 
supported in 
a finite subcomplex $L$
contained in a 
closed
ball of 
of 
radius
$r$ 
there exists an $(i-1)$-cochain $\psi$ 
supported in 
(a, necessarily finite, subcomplex of $K$ contained in)
the compact set $\overline{B}(L,T(r))$.
(Recall that we equip simplicial complexes with a piecewise--unit-$\ell^\infty$ metric; however, in the remainder of this section, we only really use the fact that the metric on a simplicial complex 
is such that each simplex is of diameter at most $1$, 
and 
that 
Lemma~\ref{l:agcI7A15} holds.)

\begin{lemma}\label{l:cohouninbhd}
	\begin{enumerate}[(i)]
		\item\label{l:cohouninbhd0} 
		Let $X$ be a proper metric space.
		If $\has_c^i(X)$ is $T$-uniformly trivial, then 
		it
		is $(r\mapsto 2T(r+1)+2r+2)$-\nbhd-uniformly trivial.
		
		\item\label{l:cohouninbhd1} 
		Let $(X,d_X)$ and $(Y,d_Y)$ be proper metric spaces. Let $f\colon X\to Y$ and $g\colon Y\to X$ be bounded homotopy equivalences.
		Assume that $f$ and $g$ are quasi-isometries, and that $\has^i_c(Y)$ is $T$-\nbhd-uniformly trivial.
		Then $\has^i_c(X)$ is $(r\mapsto CT(Cr+2C^2)+2C)$-\nbhd-uniformly trivial for some $C>0$.

		\item\label{l:cohouninbhd2} 
		Let $K$ be a 
		locally finite 
		simplicial complex and $T\colon [0,+\infty)\to[0,+\infty)$. 
		(a) If 
		$\has_c^i(X)$ is $T$-\nbhd-uniformly trivial, then $H_{\spl,c}^i(X)$ is $(r\mapsto
T(r+1)+3)$-\nbhd-uniformly trivial.		
		(b) If 
		$H_{\spl,c}^i(X)$ is $T$-\nbhd-uniformly trivial, then $\has_c^i(X)$ is $(r\mapsto T(r+1)+4)$-\nbhd-uniformly trivial.
	\end{enumerate}
\end{lemma}

\begin{proof}
\ref{l:cohouninbhd0} 
Consider $\emptyset\neq A\subseteq\overline{B}(x,r)$.
Then the map
\begin{equation*} \csch{\has}{i}{X}{A}\to\csch{\has}{i}{X}{\overline{B}(x,T(r+1)+r+1)},
\end{equation*}
in view of \eqref{eq:cschopen},
factors as 
\begin{multline*}
\csch{\has}{i}{X}{A}
\to \has_c(B(x,r+1))
\to \has_c(B(x,T(r+1)+r+1))\\
\to\csch{\has}{i}{X}{\overline{B}(x,T(r+1)+r+1)},
\end{multline*}
therefore is trivial, as the middle arrow is trivial.
The claim follows, as $\overline{B}(x,T(r+1)+r+1)$ 
is of diameter at most $2(T(r+1)+r+1)$,
hence 
\begin{equation*}
A\subseteq \overline{B}(x,T(r+1)+r+1)\subseteq\linebreak[1]\overline{B}(A,2(T(r+1)+r+1)).
\end{equation*}

\medskip	
\ref{l:cohouninbhd1} 
Denote by 
$\htpy^X\colon X\times[0,1]\to X$ the bounded 
homotopy 
between $g\circ f$ and $\mathrm{id}_X$. 
Combining various assumptions of this lemma, 
we obtain that there exists a constant $C>0$ 
such that
the diameter of the set $\htpy^X(\{x\}\times[0,1])$ is smaller than $C$
for all $x\in X$ (in particular, the maps $g\circ f$ and $\id_X$ are $C$-close),
the image of $g$ is $C$-dense in $Y$,
and
for all $y,y'\in Y$ the inequality 
$C^{-1}d_Y(y,y')-C\leq d_X(g(y),g(y'))\leq Cd_Y(y,y')+C$ 
holds.

\smallskip	
For any compact set $A\subseteq X$ 
contained in a ball $B(x_A,r)$
we have the following diagram.
\begin{equation*}
\begin{tikzcd}[row sep=1.25em,column sep=0.5em,cramped,font=\small,baseline=1em]
\csch{\has}{i}{X}{A}\arrow[r,"g^*"]\arrow[drr,"(gf)^*"',end anchor={[xshift=-5.8ex,yshift=-1.8ex]}]\arrow[dd,"\mathrm{id}^*=(gf)^*" description] & \csch{\has}{i}{Y}{g^{-1}(A)}\arrow[r] & \csch{\has}{i}{Y}{\overline{B}(g^{-1}(A),T(Cr+2C^2))}\arrow[d,"f^*"]\\  
& 
& \csch{\has}{i}{X}{f^{-1}(\overline{B}(g^{-1}(A),T(Cr+2C^2)))}\arrow[d]\\
\csch{\has}{i}{X}{\overline{B}(A,C)}\arrow[rr] & & 
\csch{\has}{i}{X}{\overline{B}(A,CT(Cr+2C^2)+2C)}
\end{tikzcd}
\end{equation*}

\medskip
\noindent
All unlabelled arrows are induced by inclusions. 
The 
maps 
$(gf)^*,\mathrm{id}^*\colon{\csch{\has}{i}{X}{A}}\to {\csch{\has}{i}{X}{\overline{B}(A,C)}}$ 
are equal
by
the fact that the bounded homotopy
$\htpy^X$
between $g\circ f$ and $\mathrm{id}_X$ induces a homotopy between the maps 
${(X,X\setminus\overline{B}(A,C))}\to{(X,X\setminus A)}$ 
induced by $g\circ f$ and $\mathrm{id}_X$ (see \cite[Theorem 6.5.6]{SpanierBook66}).
The vertical map 
induced by 
inclusion in the second column is well-defined by the following argument.
Denote 
by
$A'$
the set
$f^{-1}(\overline{B}(g^{-1}(A),T(Cr+2C^2)))$,
and consider
$x'\in A'$. 
Then there exists $y\in Y$ 
such that $d_Y(y,f(x'))\leq T(Cr+2C^2)$ 
and $g(y)\in A$. 
Then
we have that
\begin{multline*}
	d_X(x',A)\leq d_X(x',g(y))\leq d_X(x',g(f(x')))+d_X(g(f(x')),g(y))\\
	\leq C+Cd_Y(f(x'),y)+C\leq CT(Cr+2C^2)+2C,
\end{multline*}
so the discussed map in the diagram is indeed well-defined.

The 
map 
in the first row
that is induced by inclusion is zero,
as the
set $g^{-1}(A)$ is contained in a ball of radius 
$Cr+2C^2$:
let $y_A\in Y$ be such that $d_X(g(y_A),x_A)\leq C$;
then
\vspace{-0.2em} 
\begin{equation*}
g^{-1}(A)\subseteq g^{-1}(B(x_A,r))\subseteq g^{-1}(B(g(y_A),r+C))\subseteq B(y_A,C(r+C)+C^2).
\end{equation*}

Finally, joining various pieces together, we obtain that
the map
\vspace{-0.2em} 
\begin{equation*}
\csch{\has}{i}{X}{A}\to\csch{\has}{i}{X}{\overline{B}(A,CT(Cr+2C^2)+2C)}
\end{equation*}

\vspace{-0.2em}
\noindent
induced by inclusion
(the result of 
going 
the down-right route in the diagram)
is 
trivial 
(as
the result of 
going 
the right-right-down-down route in the diagram).

\newcommand{\pomp}{\!\scaleobj{1.5}{\circ}}
	
	\medskip	
	\ref{l:cohouninbhd2} 
	There is a canonical isomorphism between the simplicial 
	cohomology
	and the Alexander--Spanier 
	cohomology
	of a simplicial complex relative to its subcomplex, see e.g.~\cite[Section 6.5 and Theorem 4.6.8]{SpanierBook66}.
	For a subcomplex $L$ of $K$ we 
	denote by $L_{\simcpl}$ the subcomplex of $K$ consisting of simplices contained in $K\setminus L$;
	observe
	that 
	we have the following inclusions: $L_\simcpl\subseteq K\setminus L$ and $K\setminus\overline{B}(L,1)\subseteq L_\simcpl$.
	
	\smallskip
	(a)
	Take $\varphi\in Z_{\spl,c}^i(K)$ supported in
	a finite subcomplex 
	$L$
	of $K$ that is
	contained in a 
	closed
	ball of radius $r$.
	Let $L^{\pomp}$ be the finite subcomplex of $K$ consisting of the simplices intersecting 
	the set $\overline{B}(L,T(r+1)+1))$
	(in particular, $L^{\pomp}\subseteq\overline{B}(L,T(r+1)+2)$). 
	We have the following diagram.  
\smallskip
\begin{equation*}
\begin{tikzcd}[row sep=1.25em,column sep=1.1em,cramped,font=\small,baseline=1em]
	H^i_{\spl}(K,L_{\simcpl}) \arrow[rrr]\arrow[d,"\cong"] & & &	H^i_{\spl}(K,L_{\simcpl}^{\pomp})\arrow[d,"\cong"]\\
	\has^i(K,L_{\simcpl})\arrow[r] & \csch{\has}{i}{K}{\overline{B}(L,1)}\arrow[r] & \csch{\has}{i}{K}{\overline{B}(L,T(r+1)+1)}\arrow[r] & \has^i(K,L_{\simcpl}^{\pomp})
\end{tikzcd}
\end{equation*}

\medskip
\noindent
	The horizontal arrows are 
	induced by
	inclusions.
	The vertical arrows are canonical isomorphisms between cohomology theories. 
	By the assumption, the middle arrow in the second row is zero, therefore the arrow in the first row is zero.
		One has that
		$K\setminus L_{\simcpl}^{\pomp}\subseteq\overline{B}(L^{\pomp},1)\subseteq\overline{B}(L,T(r+1)+3)$,
	which finishes the proof.

	\smallskip
	(b)
	Let $A\subseteq K$ be compact
	and contained in a 
	closed
	ball of radius $r$.
	Let $L$ be the smallest subcomplex of $K$ that contains 
	$A$.
	Observe that $L\subseteq\overline{B}(A,1)$, therefore is contained in a ball of radius 
	$r+1$.
	Let $L^{\pomp}$ be the smallest subcomplex of $K$ containing $\overline{B}(L,T(r+1)+1)$ (so, in particular, $L^{\pomp}\subseteq\overline{B}(L,T(r+1)+2)\subseteq\overline{B}(A,T(r+1)+3)$).
	We have the following diagram.
	\smallskip
\begin{equation*}
\begin{tikzcd}[row sep=1.25em,column sep=1.1em,cramped,font=\small,baseline=1em]
	\csch{\has}{i}{K}{A}\arrow[r] & \has^i(K,L_{\simcpl})\arrow[r]\arrow[d,"\cong"] & \has^i(K,L_{\simcpl}^{\pomp})\arrow[r]\arrow[d,"\cong"] & \csch{\has}{i}{K}{\overline{B}(A,T(r+1)+4)}\\
	& H^i_{\spl}(K,L_{\simcpl})\arrow[r] & H^i_{\spl}(K,L_{\simcpl}^{\pomp}) &
\end{tikzcd}
\end{equation*}

\medskip
\noindent
	The horizontal arrows are 
	induced by
	inclusions. The vertical arrows are canonical isomorphisms between cohomology theories. 
	By the assumption, 
	the arrow in the second row is zero, therefore composition of the arrows in the top row is also zero.	
\end{proof}

\begin{lemma}[cf.~{\cite[Theorem 2, Proposition 1 and its Corollary]{GeOn07}}]\label{l:agct2p1c}
	Let $X$ be a finite-dimensional proper metric space that admits a 
	\ccc{}
	geodesic bicombing $\sigma$ and a cocompact group action via isometries.
	
	\begin{enumerate}[(i)]
		\item\label{l:agct2p1c1}
		There is a finite-dimensional, countable, locally finite simplicial complex $K$ such that $X$ and $K$ are of the same bounded homotopy type.
		Moreover, 
		the 
		bounded homotopy equivalences
		justifying this fact 
		may be chosen so that they are quasi-isometries. 
		
		\item\label{l:agct2p1c2} If $\has^i_c(X)$ is trivial, then it is uniformly trivial.
		
		\item\label{l:agct2p1c3} Assume that $\has^i_c(X)$ is trivial for 
		all
		$i>k$. Then there exists a number $\bar{t}$ such that 
		the cohomology groups
		$\has^i_c(X)$
		for $i>k$
		are 
		$(r\mapsto \bar{t}\,)$-\nbhd-uniformly trivial,
		i.e.~for all compact $A\subseteq X$ and $i>k$
		the map $\csch{\has}{i}{X}{A}\to\csch{\has}{i}{X}{\overline{B}(A,\bar{t}\,)}$
		induced by inclusion
		is trivial.
		
		\item\label{l:agct2p1c3cor}
		Under the assumptions of \ref{l:agct2p1c3},
		there 
		exists a number $t$ such that the maps $\has_c^i(U)\to\has_c^i(B(U,t))$ (induced by 
		the
		system of 
		inclusions) 
		are trivial for all open subsets $U\subseteq X$ and $i>k$.
	\end{enumerate}
\end{lemma}

\begin{proof}
	\ref{l:agct2p1c1}
	Consider a maximal set $E\subseteq X$ such that for each pair of different points $x,x'\in E$ we have $d(x,x')\geq1$. Then the family $\{B(x,1):x\in E\}$ is an open cover of $X$.
	Denote by $K$ its nerve.
	Then, by the first two paragraphs of the proof of \cite[Theorem 2]{GeOn07}, we have that $K$ is countable, locally finite and finite-dimensional.
	The remaining properties of the complex $K$ follow by Lemma \ref{l:agcI7A15} (see Remark \ref{u:agcI7A15}).
	
	\medskip
	\ref{l:agct2p1c2}
	The proof of \cite[Proposition 1]{GeOn07} may be applied directly, since the
	CAT(0)-assumption 
	is used only to deduce the conclusion of \ref{l:agct2p1c1}.
	
	\medskip
	\ref{l:agct2p1c3}
	We are working with the following picture: $\has^*_c(X)\cong \has^*_c(K)\cong H^*_{\spl,c}(K)$, where $K$ is the simplicial complex from \ref{l:agct2p1c1}. 
	Put
	$n:=\dim K$.
	By \ref{l:agct2p1c2}, the bounded cohomology groups $\has^i_c(X)$ for $i>k$ are uniformly trivial, and by a subsequent application of Lemma \ref{l:cohouninbhd}\ref{l:cohouninbhd0}, \ref{l:cohouninbhd1} and \ref{l:cohouninbhd2}(a), the bounded cohomology groups $H^i_{\spl,c}(K)$ for $i>k$ are 
	$T_i$-\nbhd-uniformly trivial 
	for some functions $T_i$, so we may pick a function $T$ such that each 
	simplicial
	$i$-cocycle for $k<i\leq n$ with support contained in a 
	closed
	ball of radius $r$ cobounds in the 
	closed
	$T(r)$-neighbourhood of its support.
	Now it suffices to show that there exists a number $\bar{t}$ such that for all $i>k$ each simplicial $i$-cocycle in $K$ of bounded support cobounds in the 
	closed
	$\bar{t}$-neighbourhood of its support, as then the claim follows from subsequent application of Lemma \ref{l:cohouninbhd}\ref{l:cohouninbhd2}(b) and \ref{l:cohouninbhd1}.
	
	To this end, 
	we construct a chain homotopy $D=D^i\colon C^i_{\spl,c}(K)\to C^{i-1}_{\spl,c}(K)$ between the identity map $\mathrm{id}$ and the zero map for $i>k$ that satisfies an additional property: for each $i$ there exists a number $\tau(i)$ such that for each $c\in C^i_{\spl,c}(K)$ we have that $D^ic$ is supported in the 
	closed
	$\tau(i)$-neighbourhood of $c$.
	First, define $D^i=0$ and take $\tau(i)=0$ for $i>n$. This clearly satisfies the required properties.
	Next, take $i$ such that $k<i\leq n$, 
	assume that we 
	have
	defined $D^{i+1}$ and $\tau(i+1)$, 
	and let $c$ be an $i$-cochain that is 
	supported in a simplex.
	Then $D^i$ must satisfy $c=\delta D^{i+1}c+D^i\delta c$, therefore $\delta D^ic=c- D^{i+1}\delta c$. Observe that the right-hand side is a cocycle: 
	\begin{equation*}
	\delta(c- D^{i+1}\delta c)=\delta c-(\delta D^{i+1})(\delta c)=\delta c-\delta c+D^{i+2}(\delta\delta c)=0+D^{i+2}(0)=0, 
	\end{equation*}
	therefore it cobounds in the
	closed 
	$T(\tau(i+1)+2)$-neighbourhood
	of its support, since, by the assumption, it is supported in the
	closed
	$(\tau(i+1)+1)$-neighbourhood
	of $c$. Define $D^ic$ to be any $(i-1)$-cochain that is supported 
	in
	the closed
	$T(\tau(i+1)+2)$-neighbourhood
	of the support of $c$ and satisfies the equation $\delta D^ic=c- D^{i+1}\delta c$. Next, extend $D^i$ linearly to $C_{\spl,c}^i(K)$. One may easily verify that it suffices to take 
	$\tau(i):=T(\tau(i+1)+2)$.
	
	Finally, let $\bar{t}:=\max_{i>k}\tau(i)$.
	Given a simplicial $i$-cocycle $c$, 
	we have that $c=\delta Dc+D\delta c=\delta Dc$, and, by the construction, $Dc$ is supported in the 
	closed
	$\bar{t}$-neighbourhood of $c$, as required.

\medskip	
\ref{l:agct2p1c3cor}
Statement
\ref{l:agct2p1c3} gives
triviality of the maps $\csch{\has}{i}{X}{A}\to\csch{\has}{i}{X}{\overline{B}(A,\bar{t}\,)}$ for $A\subseteq U$ compact, which, in view of \eqref{eq:cschopen}, implies 
the
triviality of the limit map $\has_c^i(U)\to\has_c^i(B(U,\bar{t}+1))$. 
Therefore it is sufficient to take $t:=\bar{t}+1$.
\end{proof}

\subsection{Main lemmas and proof of Theorem \ref{t:agcc3} (Thm.~\ref{t:agcc3intro})}\label{sbs:agcmain}
\label{subs:mainlems}

\begin{remark}[cf.~{\cite[Remark 2]{GeOn07}}]\label{r:cohoeq}
Let $X$ be 
a 
proper metric space that admits a 
\ccc{}
geodesic bicombing $\sigma$.
Then $\has^{*+1}_c(X)$ and 
the reduced Alexander--Spanier cohomology
group
$\hasred^{*}(\partial_\sigma X)$ 
are isomorphic. Indeed,
fix $o\in X$;
then for any $R>0$ and $i\in\N$ we have the following fragment of the 
exact sequence of the pair $(\bdry{X}{\sigma},\bdry{X}{\sigma}\setminus\overline{B}_X(o,R))$: 
\begin{equation*}
\has^{i}(\bdry{X}{\sigma})\to 
\has^{i}(\bdry{X}{\sigma}\setminus\overline{B}_X(o,R))\to 
\has^{i+1}(\bdry{X}{\sigma},\bdry{X}{\sigma}\setminus\overline{B}_X(o,R))\to 
\has^{i+1}(\bdry{X}{\sigma}). 
\end{equation*}
For $i\geq 1$
the first and the last of these groups are trivial,
as $\bdry{X}{\sigma}$ is contractible; therefore the middle arrow is an isomorphism, which together with excision gives that 
\begin{equation*}
\has^{i}(\bdry{X}{\sigma}\setminus\overline{B}_X(o,R))\cong 
\has^{i+1}(\bdry{X}{\sigma},\bdry{X}{\sigma}\setminus\overline{B}_X(o,R))\cong 
\has^{i+1}(X,X\setminus\overline{B}_X(o,R));
\end{equation*}
passing to the limit, we obtain the desired isomorphism (see \cite[Theorem 6.6.2]{SpanierBook66}).
Similarly follows the 
case of
$i=0$, 
where
we have 
$\Z$ 
as the left term of the exact sequence above.

In fact, the above argument 
works in a more general setting.
Consider a 
compact subset $Z$ 
of 
a 
compact 
space $\mathfrak{X}$, such
that $Z$ is a Z-set in $\mathfrak{X}$. 
Let $\{\htpy_t\colon\mathfrak{X}\to\mathfrak{X}:t\in[0,1]\}$ be the homotopy from the definition of 
Z-set. 
Assume that 
$\mathfrak{X}$ is contractible and
$\bigcup_{t>0}\htpy_t(\mathfrak{X})=\mathfrak{X}\setminus Z$.
Then
$\has^{*+1}_c(\mathfrak{X}\setminus Z)\cong\hasred^{*}(Z)$ 
--- 
it suffices to consider 
the sets
$(\htpy_t(\mathfrak{X}))_{t>0}$ in the place of balls $(\overline{B}(o,R))_{R>0}$ 
in the argument above,
as
the family $\{\htpy_t(\mathfrak{X}):t>0\}$ is cofinal in the family of all compact subsets of $\mathfrak{X}\setminus Z$.
	
\end{remark}

\begin{theorem}[Theorem \ref{t:agcmtintro}; cf.~{\cite[Main Theorem]{GeOn07}}]\label{t:agcmt}
	Let $X$ be a 
	finite-dimensional proper metric space 
	that admits a
	\ccc{}
	geodesic bicombing $\sigma$ and a cocompact group action via isometries,
	such that $|\partial_\sigma X|\geq 2$. 
	Then the
	reduced Alexander--Spanier cohomology 
	group $\hasred^{\dim \partial_\sigma X}(\partial_\sigma X)$ is non-zero.
\end{theorem}

\begin{remark}\label{u:oneptbdry2}
	For more discussion on the assumption that $|\partial_\sigma X|\geq 2$,
	see Remark \ref{u:oneptbdry1}.
	(And clearly, a space with $|\partial_\sigma X|=1$ does not satisfy the conclusion of this theorem.)
\end{remark}

\begin{remark}\label{u:findimbdry}
	Recall
	that in the proof of Theorem \ref{t:main} in Section \ref{s:deslan} 
	we proved that for a finite-dimensional proper metric space that admits a 
	\ccc{}
	bicombing 
	we have that $\dim\bdry{X}{\sigma}\leq\dim X$.
	In particular, $\dim\partial_\sigma X<\infty$, so the cohomology group $\hasred^{\dim \partial_\sigma X}(\partial_\sigma X)$ in the statement above is well-defined.
\end{remark}

\begin{remark}[cf.~{\cite[below Remark 3]{GeOn07}}]\label{u:groupch}
If the action of the group, $G$, in the statement above is geometric (i.e.~it is additionally proper), then we have the isomorphism $\has^*_c(X)\cong H^*(G,\Z G)$, 
which 
by 
Remark \ref{r:cohoeq} 
gives 
the isomorphism $\hasred^*(\partial_\sigma X)\cong H^{*+1}(G,\Z G)$.
Indeed, by taking the nerve
of a locally finite open cover 
$\mathcal{U}$
of $X$ by balls (of the same radius), 
such that
$\mathcal{U}$
is closed under the action of $G$,
one obtains a simplicial complex $K$ with a geometric (and simplicial) action of~$G$,
boundedly homotopic to the space $X$ (see Lemma~\ref{l:agcI7A15}). Then 
by \cite[Exercise VIII.7.4]{BrownBook82}
it follows that $H^*(G,\Z G)\cong\has^*_c(K)\cong\has^*_c(X)$. 	
\end{remark}

\begin{proof}
	There is the following cohomological definition of dimension of a topological space $Y$:
	\vspace{-0.3em} 
	\begin{equation*}
		\dim_\Z Y=\sup\{i\in\N: (\exists A\subseteq Y)(A\text{ is closed}, \has^i(Y,A)\neq 0)\}.
	\end{equation*}
	(Usually one uses the \v{C}ech 
	cohomology 
	in the above definition, but these coincide with the Alexander--Spanier 
	cohomology
	---
	see \cite[Corollary 6.8.8 and Exercise 6.D]{SpanierBook66} or \cite[Theorem 2]{Dowker52}.)
	The studies on (co)homological notions of dimension date back to Alexandrov \cite{Alexandroff32}.
	The above definition of dimension coincides with the covering dimension for separable metric spaces of finite covering dimension, see e.g.~\cite[below Corollary 1.9.9]{EngelkingBook78}.
	
	\smallskip
	Since the space $\bdry{X}{\sigma}$ is of finite dimension
	(see Remark \ref{u:findimbdry}),
	the topological dimension $n$ of $\partial_\sigma X$ is also finite, 
	therefore its cohomological dimension $\dim_\Z\partial_\sigma X$ is also $n$.
	Let $d$ be the largest number such that $\has_c^{d+1}(X)\cong\hasred^d(\partial_\sigma X)\neq 0$ (recall Remark \ref{r:cohoeq} for the isomorphism).
	By considering 
	the set 
	$A$ 
	consisting of a single point (in the definition of the cohomological dimension above),
	we have $d\leq n$. 
	Assume a contrario that $d<n$.
	Take 
	$\emptyset\neq A\subseteq\partial_\sigma X$ 
	closed such that $\has^n(\partial_\sigma X,A)\neq 0$ 
	(for the non-emptiness of $A$, one may use the exact sequence of the pair $(X,A)$ for $n\geq 1$, or \cite[Theorem 6.4.5]{SpanierBook66} for $n=0$).	
	Fix a basepoint $o\in X$.
	By \cite[Lemma 1]{GeOn07},
	there exist open balls $B_k\subseteq\partial_\sigma X\setminus A$ of radius $1/k$ (with respect to some metric on the boundary $\partial_\sigma X$, it does not matter which one)
	for 
	sufficiently large 
	$k\in\N$, and 
	$\bar{\gamma}\in\partial_\sigma X\setminus A$, such that the balls $B_k$ 
	converge to 
	$\bar{\gamma}$,
	and the maps $\has^n(\partial_\sigma X,\partial_\sigma X\setminus B_k)\to\has^n(\partial_\sigma X,A)$ induced by inclusion are all non-zero. 
	
	The 
	outline of the 
	remaining part of the
	proof
	of
	the current 
	theorem
	is as follows.
	Using 
	the $\mathrm{Cone}_o$ operation, 
	one may deduce from
	the 
	non-zeroness of the map ${\has^n(\partial_\sigma X,\partial_\sigma X\setminus B_k)}\to{\has^n(\partial_\sigma X,A)}$ for a sufficiently large $k$
	---
	where we have the set $A$ contained in a `very large' set $\partial_\sigma X\setminus B_k$
	---
	that we have a similar situation in $X$, namely that
	the map $\has^{n+1}_c(X\setminus E)\to\has^{n+1}_c(X\setminus D)$ 
	induced by inclusion 
	is non-zero
	for some $E,D\subseteq X$ 
	such that a ball of large diameter around $X\setminus E$ is contained in $X\setminus D$. 
	This,
	however,
	contradicts Lemma \ref{l:agct2p1c}\ref{l:agct2p1c3cor}.
	
	\smallskip
	Now we proceed to the details.
	Let 
	the set 
	$D$ be 
	related to the set $A$
	as in the statement of
	Lemma \ref{l:agcD}.	
	Pick a constant $t$ as guaranteed by Lemma \ref{l:agct2p1c}\ref{l:agct2p1c3cor}.
	Since $\bar{\gamma}\not\in A$, by 
	compactness of $A$ and 
	Proposition \ref{f:asymptotic}\ref{f:asymptotic2}, 
	one may choose $s_{\bar{\gamma}}$ such that $d(\geod{o}{\bar{\gamma}}(s_{\bar{\gamma}}),\mathrm{Cone}_o(A))\geq t+3$.
	Since $B_k$ converges to $\bar{\gamma}$, one may choose $k$ such that $\exp_o(B_k\times\{s_{\bar{\gamma}}\})\subseteq B(\geod{o}{\bar{\gamma}}(s_{\bar{\gamma}}),1)$.
	Let $B:=B_k$.
	Define 
	\begin{equation*}
	E:=\overline{B}(o,s_{\bar{\gamma}})\cup\{\bar{x}\in\bdry{X}{\sigma}:d(\geod{o}{\bar{x}}(s_{\bar{\gamma}}),\geod{o}{\bar{\gamma}}(s_{\bar{\gamma}}))\geq 1\}\cup\mathrm{Cone}_o(\partial_\sigma X\setminus B).
	\end{equation*}
	Observe that it is a closed set such that $E\cap\partial_\sigma X=\partial_\sigma X\setminus B$.
	Furthermore, $d(D\cap X,{X\setminus E})\geq t+1$: take $a\in\mathrm{Cone}_o(A)\cap X$ and $x\in X\setminus E$; 
	then $d(x,o)>s_{\bar{\gamma}}$, so by conicality of $\sigma$ we have that 
	\begin{multline*}
		d(a,x)\geq d(\sigma_{o,a}(s_{\bar{\gamma}}/d(o,x)),\sigma_{o,x}(s_{\bar{\gamma}}/d(o,x)))=d(\sigma_{o,a}(s_{\bar{\gamma}}/d(o,x)),\geod{o}{x}(s_{\bar{\gamma}}))\\
		\geq d(\sigma_{o,a}(s_{\bar{\gamma}}/d(o,x)),\geod{o}{\bar{\gamma}}(s_{\bar{\gamma}}))-d(\geod{o}{\bar{\gamma}}(s_{\bar{\gamma}}),\geod{o}{x}(s_{\bar{\gamma}}))\geq t+3-1=t+2; 
	\end{multline*}
	the claim follows as $D\cap X\subseteq\overline{B}(\mathrm{Cone}_o(A)\cap X,1)$.
	
	We have the following diagram.
	\smallskip	
	\begin{equation*}
	\begin{tikzcd}[row sep=0.85em,column sep=1.1em,cramped,font=\small,baseline=1em]
		& |[orange]| \has^n(\partial_\sigma X\cup E,E) \arrow[r, orange]\arrow[dddd,blue] & |[orange]| \has^{n+1}(\bdry{X}{\sigma},\partial_\sigma X\cup E)\arrow[dddd,blue]\\ 
		\has^n(\partial_\sigma X,\partial_\sigma X\setminus B) \arrow[ur,"\cong"]\arrow[dd] & & & |[darkgreen]| \has^{n+1}_c(X\setminus E)\arrow[ul,darkgreen,"\cong"']\arrow[d,darkgreen]\\
		& & & |[darkgreen]| \has_c^{n+1}(B(X\setminus E,t+1))\arrow[d,darkgreen]\\
		\has^n(\partial_\sigma X,A)\arrow[dr,"\cong"] & & & |[darkgreen]| \has_c^{n+1}(X\setminus D)\arrow[dl,darkgreen,"\cong"']\\
		|[magenta]| 0=\has^n(\bdry{X}{\sigma},D)\arrow[r,magenta] & |[magenta]| \has^n(\partial_\sigma X\cup D,D)\arrow[r,magenta,"\cong"] & |[magenta]|  \has^{n+1}(\bdry{X}{\sigma},\partial_\sigma X\cup D)\arrow[r,magenta] & |[magenta]|  \has^{n+1}(\bdry{X}{\sigma},D)=0
	\end{tikzcd}
	\end{equation*}
	
	\medskip
	\noindent
	The upper and the lower rows are fragments of the exact 
	sequences 
	of the triples $(\bdry{X}{\sigma},E\cup\partial_\sigma X,E)$ and $(\bdry{X}{\sigma},D\cup\partial_\sigma X,D)$, respectively. 
	In the lower row, 
	the middle arrow is an isomorphism, 
	since $D$ is a deformation retract of $\bdry{X}{\sigma}$.
	The two long vertical arrows in the middle are induced by the inclusion of triples $(\bdry{X}{\sigma},D\cup\partial_\sigma X,D)\hookrightarrow (\bdry{X}{\sigma},E\cup\partial_\sigma X,E)$. 
	The square on the left is the 
	excision of $X$ 
	(the so-called strong excision property, 
	see \cite[Theorem 6.6.5]{SpanierBook66}).
	The isomorphisms on the right follow by the fact that, in compact spaces, the 
	cohomology
	relative to a closed set may be identified with the compactly supported 
	cohomology 
	of its complement, see \cite[Lemma 6.6.11]{SpanierBook66}.
	The vertical arrows on the right follow from the inclusions $X\setminus E\subseteq B(X\setminus E,t+1)\subseteq X\setminus D$.
	
	We obtain a contradiction in the following way. The vertical arrow on the left is non-zero by the assumption on the sets
	$B_k$ for $k\in\N$,
	therefore, 
	moving 
	step-by-step 
	from the left to the right, one 
	shows that each vertical arrow in the diagram is
	also non-zero. 
	In particular,
	the map $\has_c^{n+1}(X\setminus E)\to \has_c^{n+1}(B(X\setminus E,t+1))$ is non-zero, which gives a contradiction with 
	the definition of the number $t$. 
\end{proof}

We note that the argument in the proof of 
the
lemma below works also for the simplicial 
cohomology.

\begin{lemma}[cf.~{\cite[Theorem A]{Ontaneda05}}]\label{l:agcta}
	Let $X$ be a 
	proper non-compact metric space that admits a 
	\ccc{}
	geodesic bicombing $\sigma$ and a cocompact group action via isometries.
	If any of the groups 
	$\has_c^i(X)$ is non-zero, then $X$ is almost 
	$\sigma$-geodesically 
	complete.	
\end{lemma}

\begin{proof}
	The proof \cite[Theorem A]{Ontaneda05} can be adapted by changing each occurrence of `(the unique) CAT(0)-geodesic (ray)' to `(the) 
	$\sigma$-geodesic(/ray)'.
	In particular, the objects: $\ell_\alpha$ (which later on we call $L_\alpha$, to avoid confusing it with the function $\ell_o$ 
	from Definition \ref{d:ellexp}), 
	$f^{p,s}\colon X\to X$, 
	and $[p,x]$, $[p,x_0)$, where $\alpha$ is a $\sigma$-geodesic, $p,x\in X$, $s>0$ and $x_0\in\partial_\sigma X$, become 
	\begin{equation*}
	(L_\alpha=)\ell_\alpha=\sup\{t\in\R:\text{there exists a }\sigma\text{-geodesic of length }t\text{ that extends } \alpha\},
	\end{equation*}
	$f^{p,s}(x)=\exp_p(x,\max(d(p,x)-s,0))$, 
	$[p,x]=\geod{p}{x}$ 
	and $[p,x_0)=\geod{p}{x_0}$.
	The last two notations will not be used further in this proof.
	The function $f^{p,s}$ is the key object in this proof, and 
	the value $f^{p,s}(x)$
	can be 
	described
	as the point reached 
	in the following walk: `starting in $x$,
	go 
	backwards
	along the $\sigma$-geodesic
	that begins in
	$p$
	and ends in
	$x$,
	towards 
	$p$,
	with unit speed for time $s$, unless you 
	reach $p$ earlier --- then stop'.
	
	\smallskip
	The outline is as follows.
	Assume
	that
	the space $X$ is not almost 
	$\sigma$-geodesically 
	complete;
	we shall show that
	then 
	every
	element
	$\varphi\in\csch{H}{i}{X}{\overline{B}(x,R)}$,
	where $x\in X$ and $R>0$,
	is zero in $\has^*_c(X)$
	(i.e.~the image of $\varphi$ under the canonical map is zero). 
	Fix a basepoint $o\in X$.
	The 
	action 
	of $G$ on $X$
	via isometries 
	induces an action of $G$ on $\has^*_c(X)$ via isomorphisms.
	Therefore, since the action of $G$ on $X$ is cocompact,
	one may assume
	without loss of generality 
	that 
	the ball $\overline{B}(x,R)$ 
	intersects no 
	$\sigma$-ray
	originating in $o$:
	the $G$-orbit of $x$ is $D$-dense in $X$ for some $D>0$,
	and, as $X$ is not almost 
	$\sigma$-geodesically 
	complete, there exists $x_\odot\in X$ such that $\overline{B}(x_\odot,D+R)$ intersects no 
	\linebreak[2]
	$\sigma$-ray; let $g$ be such that $d(gx,x_\odot)\leq D$; then $g^*\varphi\in\csch{\has}{i}{X}{\overline{B}(gx,R)}$
	is non-zero in $\has^*_c(X)$ iff $\varphi$ is,
	and $\overline{B}(gx,R)$
	intersects no $\sigma$-ray.
	Using properness of $X$, one may obtain that $C:=\sup\{L_{\geod{o}{x'}}:x'\in\overline{B}(x,R)\}$
	is finite:
	if there existed $x_n\in X$ such that $d(o,x_n)\to\infty$ and $t_n\leq d(o,x_n)$ such that $\geod{o}{x_n}(t_n)\in\overline{B}(x,R)$, then by compactness of $\overline{B}(x,R)$ and $\bdry{X}{\sigma}$, one could choose a subsequence $x_{n_k}$ convergent to some $\bar{x}\in\partial_\sigma X$ such that $\geod{o}{x_{n_k}}(t_{n_k})\to a$ for some $a\in\overline{B}(x,R)$; then $a=\lim_k\geod{o}{x_{n_k}}(t_{n_k})=\geod{o}{\bar{x}}(d(o,a))$ (recall Proposition \ref{p:ellexp}), which is a contradiction.
	Since the identity $\id_X$ and the function $f^{o,C}$ are boundedly homotopic, the (images under the canonical maps of the) elements $\varphi=(\id_X)^*\varphi\in\csch{\has}{i}{X}{\overline{B}(x,R)}$ and $(f^{o,C})^*\varphi\in\csch{\has}{i}{X}{(f^{o,C})^{-1}(\overline{B}(x,R))}$ are equal 
	in $\has_c^i(X)$, 
	see \cite[Theorem 6.5.6]{SpanierBook66}.
	Since the image of $f^{o,C}$ omits $\overline{B}(x,R)$, the 
	latter
	cohomology group is trivial.
	Therefore $\varphi$ is trivial in $\has_c^i(X)$.
\end{proof}

\begin{proof}{\sc (of Theorem \ref{t:agcc3} (Thm.~\ref{t:agcc3intro}))}
Follows from 
Theorem \ref{t:agcmt} 
and 
Lemma 
\ref{l:agcta}, and an application of the isomorphism considered in Remark \ref{r:cohoeq}.
\end{proof}

\section{Problems and open questions}\label{s:open}

Below we collect and present 
some
problems and open questions arising from this article, 
which 
constitute a natural continuation of the 
topics of
research
discussed 
in 
this paper.

\begin{enumerate}[ref=\thesection.Q\arabic*, label=Q\arabic*.,listparindent=0em,parsep=0.4em,itemsep=0.2em]
	\item\label{q:Helly}
	Does a counterpart of Theorem \ref{t:nonunique} (Thm.~\ref{t:nonuniqueintro}) hold for Helly groups?\ ---
	that is: does there exist a group acting an two Helly graphs 
	such that the boundaries of their injective hulls (recall the proof of Corollary \ref{c:injective}\ref{c:injective2}) are non-homeomorphic?
	
	A discussion on adapting the example by Croke and Kleiner \cite{CrKl00}, used in the proof of Theorem~\ref{t:nonunique}, to the Helly case
	is made in Remark \ref{r:nonunique}\ref{r:nonunique2}, where it is pointed out that a problem lies in local non-Hellyness around diagonal gluing-lines --- does a local Hellyfication around such lines solve the problem?
	
	\item\label{q:unctbly} Does there exist a group acting geometrically on more than two (especially, uncountably many) injective metric spaces 
	with pairwise non-homeomorphic boundaries?
	
	A discussion on using the Wilson's approach \cite{Wilson05} to answer this question positively is made in Remark \ref{r:nonunique}\ref{r:nonunique1}.
	
	\item\label{q:ccc} %
	May the second conclusion of Corollary \ref{c:ccc} (Thm.~\ref{t:cccintro}) not hold for a CAT(0) cube complex of arbitrary dimension?
	---
	that is: does there exist a 
	CAT(0) cube complex $X$ such that 
	for some $p,r\in[1,\infty]$
	the boundaries of $X$ 
	endowed 
	with the 
	piecewise-$\ell^p$ metric and the 
	piecewise-$\ell^r$ metric
	are non-homeomorphic?
	
	May the first conclusion of Corollary \ref{c:ccc} not hold for a 
	CAT(0) cube complex (of arbitrary dimension) that admits a cocompact group action?\ %
	---
	that is:  does there exist a 
	CAT(0) cube complex $X$ admitting a cocompact group action such that 
	for some $p,r\in[1,\infty]$
	the 
	identity of $X$
	does not extend to a homeomorphism
	between the 
	boundary-compactification of $X$
	endowed 
	with the 
	piecewise-$\ell^p$ metric and 
	the 
	boundary-compactification of $X$ endowed with the
	piecewise-$\ell^r$ metric?
	
	A related discussion is made in Remark \ref{r:cccdisc}.
	
	\item\label{q:pasw} Do the results relating topological properties of boundaries of CAT(0) spaces and algebraic properties of groups acting upon them, e.g.~in the spirit of Swenson \cite{Swenson99} or Papasoglu--Swenson \cite{PaSw09}, extend to the realm of spaces admitting \ccc{} geodesic bicombings?\ %
	---
	for example: does the boundary of a \ccc-bicombable space acted upon geometrically by a 1-ended group have no cut-points?
	does the converse of Proposition~\ref{p:amalgamated} (Prop.~\ref{p:amalgamatedintro}) hold, i.e.~%
	does a group acting geometrically on a \ccc-bicombable space $X$ split as an amalgamated product over 
	a
	2-ended subgroup when the boundary of $X$ has a local cut-point?
	
	\item \label{q:examples}
	Describe the topology of boundaries of 
	some
	examples of
	\ccc-bicombable spaces acted upon in a controlled way by a group. 
	For 
	example, 
	the 
	C(4)--T(4) small cancellation groups and the FC-type Artin 
	groups
	are Helly groups \cite{CCGHO25,HuOs21}, 
	hence they act geometrically on \ccc-bicombable spaces 
	(see Corollary \ref{c:injective}\ref{c:injective2}), 
	which
	provides these classes of groups with the first 
	notion of 
	boundary
	(of an associated space)
	known 
	for them.

	\item \label{q:oneptbdry}
	Does there exist a proper \ccc-bicombable space that admits a cocompact group action, 
	such that 
	its 
	boundary consists of exactly $1$ point? 
	
	See Remarks \ref{u:oneptbdry1} and \ref{u:oneptbdry2} for more context and detail.
	
\end{enumerate}

\printbibliography

\end{document}